\documentclass[11pt,reqno]{amsart}

\usepackage{amsmath,amssymb,amsthm,mathtools}
\usepackage{tikz-cd}
\usepackage[protrusion=true,expansion=false]{microtype}
\usepackage{hyperref,cleveref}
\hypersetup{colorlinks=true,linkcolor=blue!60!black,
  citecolor=blue!60!black}
\usepackage{enumitem,booktabs,xcolor,array,graphicx}

\theoremstyle{plain}
\newtheorem{theorem}{Theorem}[section]

\newtheorem{proposition}[theorem]{Proposition}
\newtheorem{corollary}[theorem]{Corollary}
\theoremstyle{definition}
\newtheorem{definition}[theorem]{Definition}
\newtheorem{example}[theorem]{Example}
\newtheorem{remark}[theorem]{Remark}

\newcommand{\R}{\mathbb{R}}
\newcommand{\Rbar}{\overline{\mathbb{R}}}
\newcommand{\Rp}{\mathbb{R}_{+}}
\newcommand{\C}{\mathbb{C}}
\newcommand{\Z}{\mathbb{Z}}

\newcommand{\Fun}{\mathrm{Fun}}
\newcommand{\Ln}{L^2(\R^n)}
\newcommand{\hatL}{\hat{L}}
\newcommand{\leF}{\leq_{\!\mathcal{F}}}
\newcommand{\meetF}{\wedge_{\!\mathcal{F}}}

\newcommand{\Ker}{\mathrm{Ker}}
\newcommand{\Bas}{\mathrm{Bas}}
\newcommand{\conv}{\varepsilon^{\mathrm{Conv}}}
\newcommand{\convDil}{\delta^{\mathrm{Conv},*}}
\newcommand{\opens}{\gamma^{\mathrm{Conv}}}
\newcommand{\supp}{\mathrm{supp}}
\newcommand{\MaxPool}[1]{\delta^{\downarrow #1}_{\mathrm{MP}}}

\graphicspath{{./figures/}{./}}
\begin{document}

\title[Morphology in the Fourier Inf-Semilattice]{Morphological Representation Theory in the Fourier 
  Inf-Semilattice: Universal Decomposition of
  Frequency-Domain Deep Learning Operators}

\author{Gustavo (Jes\'{u}s) Angulo}
\address{CMA -- Centre de Math\'ematiques Appliqu\'ees,
  Mines Paris, PSL University, Sophia-Antipolis, France}
\email{jesus.angulo@minesparis.psl.eu}

\thanks{This work was partially supported by the PSL Global Seed Fund 2024 and my AI Chair PR[AI]RIE--PSAI (ANR-23-IACL-0008, France 2030).}

\begin{abstract}
We develop a morphological representation theory for operators
acting in the frequency domain of $L^2(\R^n)$.
Equipping the space with the \emph{Fourier inf-semilattice}
order $\leF$ (spectral modulus inequality with phase
equality), convolution becomes a morphological erosion and
the adjoint dilation is the Wiener inverse filter.
A key observation is that spatial translation-invariance is
vacuous in the spectral order, replaced structurally by
\emph{positive homogeneity} of the operator on spectral moduli.

Under this hypothesis we prove the \emph{Morphological
Representation Theorem} for the Fourier modulus lattice
$\hatL$: any increasing, USC, positively homogeneous
$\Psi:\hatL\to\hatL$ decomposes \emph{exactly} as a supremum
of max-times erosions $\hat{f}/\psi$, indexed by a minimal
morphological basis.
Three structural corollaries follow: the spectral modulus map
$|\cdot|:\Ln\to\hatL$ is an idempotent cross-lattice
projection (the morphological analogue of ReLU) that makes
depth non-trivial; wavelet scattering coefficients form the
canonical dictionary for the basis via Littlewood--Paley
density; and spatial pooling is a Fourier erosion, so that
pool-then-unpool is the ideal band-pass opening, U-Net skip
connections are Fourier top-hat transforms, and strided
pooling preserves the morphological structure if and only if
the Nyquist condition holds.
The framework carries a natural $\mathbb{C}^*$-group
morphology structure, positioning mathematical morphology
as the constructive operator-theoretic complement of
spectral bias, scattering, and tropical geometry theories.
\end{abstract}


\maketitle

\section{Introduction}
\label{sec:intro}

\subsection{Frequency-domain theories of deep learning}

The Fourier transform has been central to the theoretical
analysis of deep learning for two decades, yet the
existing frameworks each illuminate only one facet of
the picture.

The spectral bias literature \cite{Rahaman2019,Xu2019}
establishes that gradient descent learns low-frequency
components of the target function before high-frequency
ones, yielding practical guidelines for initialisation
and activation design.
But spectral bias describes the order in which frequencies
are acquired during training; it says nothing about the
algebraic structure of what the converged network computes.
Mallat's scattering transform \cite{Mallat2012,Bruna2013}
goes deeper: cascaded wavelet convolutions interleaved
with modulus nonlinearities produce representations that
are provably Lipschitz to diffeomorphisms and
translation-invariant, and the theory provides sharp
stability guarantees.
The limitation is that the wavelet filters are fixed by
the analyst; there is no theorem characterising general,
learnable frequency-domain operators. Neural tangent kernel theory \cite{Jacot2018} and
mean-field methods \cite{WeinanE2019} go further in
describing the training dynamics in the infinite-width
limit, with a spectral decomposition of the kernel
governing convergence rates.
From a different angle, tropical geometry
\cite{Zhang2018,Maragos2021} reveals that deep ReLU
networks compute piecewise-linear functions describable
as tropical rational maps, a \emph{function
representation} of what the network computes on a given
input.

\subsection{A missing layer: operator decomposition and representation}

None of the above provides a \emph{constructive operator
decomposition} of frequency-domain deep learning operators.
The question is: given an operator $\Psi$ acting on the
spectral domain of a signal, what is the minimal generating
set of elementary operations that represents $\Psi$ exactly,
and how can this set be identified from data?
This is an operator-theoretic question, distinct from the
function-theoretic question (what does $\Psi$ compute on
a given input?) and from the training-dynamics question
(how does $\Psi$ converge during gradient descent?).

For operators on scalar function lattices, the answer is
provided by the \emph{Matheron--Maragos--Banon--Barrera
(MMBB) representation theorems}
\cite{Matheron1975,Maragos1989,Banon1993}: any
translation-invariant, increasing, upper-semi-continuous
operator on $(\Fun(E,\Rbar),\leq)$ decomposes exactly as
a supremum of morphological erosions, parameterised by the
minimal elements of its kernel.
This is the nonlinear, lattice-theoretic counterpart of
the spectral representation theory for linear operators.
But the classical MMBB theorem operates in the spatial
domain, with translation-invariance as its structural
hypothesis.
As we will discuss, neither the partial order ($\leq$ vs $\leF$), the erosion
operation (inf-convolution vs convolution), nor the
structural hypothesis (TI vs positive homogeneity) carries
over to the Fourier domain without fundamental changes.

\subsection{Our framework}

We extend the MMBB framework to the \emph{Fourier
inf-semilattice} $(\Ln,\leF)$, i.e., the complete
inf-semilattice on $L^2(\R^n)$ induced by the partial
order on spectral moduli and phases introduced by Keshet
\cite{Keshet2002}, and to the \emph{Fourier modulus
lattice} $\hatL$ of non-negative spectral modulus functions.
Every ingredient of the classical framework changes:
convolution replaces inf-convolution as the canonical
erosion; positive homogeneity replaces TI as the structural
hypothesis (because translation acts as phase modulation,
which is transparent to the spectral order); scattering
wavelet products replace structuring functions as basis
elements; and entirely new structure emerges that has no
analogue in the spatial domain, including the ideal
band-pass opening, the cross-lattice modulus activation,
and the $\mathbb{C}^*$-group morphology.

Concretely, our contributions are:

\begin{enumerate}[label=(\roman*),leftmargin=*]
  \item The Fourier inf-semilattice $(\Ln,\leF)$ 
  and the proof that
    convolution is a Fourier-lattice erosion
    (Section~\ref{sec:fourier_lattice}).

  \item The identification that TI is vacuous in the
    spectral order, and that positive homogeneity is the
    correct replacement hypothesis
    (Sections~\ref{sec:fourier_lattice}--\ref{sec:mmbb}).

  \item The \emph{Morphological Representation Theorem for
    $\hatL$} (Theorem~\ref{thm:mmbb_hatL}): an exact
    sup-of-erosions decomposition with a complete
    three-step proof, and a structural comparison with the
    Cybenko--Hornik UAT
    (Section~\ref{sec:mmbb}).

  \item The spectral modulus $|\cdot|:\Ln\to\hatL$ as the
    \emph{morphological activation}: the cross-lattice
    projection that makes depth non-trivial
    (Section~\ref{sec:activation}).

  \item Scattering coefficients as the canonical
    morphological basis dictionary via Littlewood--Paley
    density.

  \item A complete morphological theory of pooling,
    downsampling, and unpooling from the Fourier perspective: spatial pooling as
    Fourier erosion; pool-then-unpool as the ideal
    band-pass opening; U-Net skip connections as Fourier
    top-hat transforms; anti-aliased pooling as the unique
    stride-consistent morphological operation
    (Section~\ref{sec:pooling}).

  \item The representation theory of openings in
    $(\Ln,\leF)$ and $\hatL$: every spectrally
    multiplicative opening is an ideal band-pass filter;
    scattering is a nested chain of such projections
    (Section~\ref{sec:openings}).

  \item The $\mathbb{C}^*$-group morphology structure:
    positive homogeneity $=$ $\mathbb{R}_+$-equivariance;
    the modulus map $=$ $S^1$-quotient; connection to
    complex-valued networks and scale-equivariance
    (Section~\ref{sec:complex_group}).
\end{enumerate}

\subsection{Positioning of this paper}
\label{subsec:positioning}

This paper is a substantially extended version
of the conference paper~\cite{Angulo2025DGMM}, which introduced 
the Fourier inf-semilattice
$(\Ln,\leF)$, the max-times erosion structure of convolution,
and the MMBB theorem for $\hatL$.
The present paper adds substantially to it in the following
respects.
First, we provide a more sound analysis of scattering propagators
$U_\Lambda(f)=|V_\Lambda(f)|$ as Fourier-lattice erosions
in $(\Ln,\leF)$. 
Second, we supply complete proofs throughout, replacing the
sketches of~\cite{Angulo2025DGMM}.
Third, this paper provides a significant number of additional theoretical results.

This work is one contribution within a broader research
programme developing a rigorous lattice-theoretic foundation
for mathematical models of deep learning based on
mathematical morphology.
The programme's main reference is~\cite{Angulo2026Lattice},
which works entirely in the \emph{pointwise lattice}
$(\Fun(E,\Rbar),\leq)$ and it covers the MMBB representation of convolution
and morphological activation layers, the cross-lattice
structure of standard CNNs in the spatial domain, the
fixed-point and idempotency theory of morphological
ResNet and UNet blocks, and the self-dual median
inf-semilattice for signed feature maps.

The present paper contributes a \emph{fundamentally
new lattice} to the programme: the Fourier inf-semilattice
$(\Ln,\leF)$ and its modulus quotient $\hatL$.
Together, the two papers give a complete morphological
lattice theory for the two natural settings of deep
learning operators: the spatial (pointwise) domain and
the spectral (Fourier) domain.

\bigskip

\textbf{Relation to the broader literature.}
Table~\ref{tab:positioning} positions the paper relative to
the major frameworks related to deep learning and frequency-domain 
theoretical interpretation of neural networks.

\begin{table}[t]
\caption{Positioning of this paper relative to some
  frameworks for deep learning.}
\label{tab:positioning}
\centering\small
\renewcommand{\arraystretch}{1.5}
\begin{tabular}{@{}p{3.2cm}p{3.0cm}p{4.5cm}@{}}
\toprule
\textbf{Framework}
  & \textbf{What it provides}
  & \textbf{What this paper adds} \\
\midrule
Spectral bias \cite{Rahaman2019}
  & Training dynamics in freq.\ domain
  & Operator decomposition at convergence \\
Scattering \cite{Mallat2012}
  & Invariance, stability (fixed wavelets)
  & Learnable basis; MMBB universality \\
Tropical geometry \cite{Zhang2018}
  & Function repr.\ (tropical polynomials)
  & Operator repr.\ (sup of erosions) \\
NTK \cite{Jacot2018}
  & Infinite-width training dynamics
  & Finite-width operator structure \\
Conv.\ theorem \cite{Folland1999}
  & Linear spectral operators
  & Nonlinear operators; MMBB basis \\
MMBB classical \cite{Maragos1989}
  & TI increasing ops on scalar lattice
  & Fourier lattice; homogeneity replaces TI \\
Pointwise MMBB \cite{Angulo2026Lattice}
  & Spatial-domain morphological DL
  & Fourier-domain morphological DL \\
\bottomrule
\end{tabular}
\end{table}

\subsection{Organisation}

Section~\ref{sec:background} provides the mathematical
morphology background needed by readers unfamiliar with the
field.
Section~\ref{sec:fourier_lattice} develops the Fourier
inf-semilattice and the erosion structure of convolution.
Section~\ref{sec:mmbb} proves the MMBB theorem for $\hatL$
and identifies the spectral modulus
as the morphological activation.
Section~\ref{sec:scattering} applies the framework to
scattering networks, identifies the canonical basis.
Section~\ref{sec:pooling} develops the complete morphological
theory of pooling and unpooling: filtering as Fourier erosion,
downsampling/upsampling as lattice homomorphisms, and the
full strided pooling--unpooling cycle with U-Net skip
connections.
Section~\ref{sec:openings} develops the representation theory
of openings in $(\Ln,\leF)$ and $\hatL$, and characterises
the scattering cascade as a nested chain of ideal band-pass
projections.
Section~\ref{sec:complex_group} develops the
$\mathbb{C}^*$-group morphology structure.
Section~\ref{sec:experiments} provides six numerical
experiments that directly validate the main results. 
Section~\ref{sec:conclusion} situates the results with
respect to related theories and states open problems.

\section{Background: Mathematical Morphology on Complete Lattices}
\label{sec:background}

This section is self-contained and assumes no prior knowledge
of mathematical morphology.
Readers interested in going deeper are referred to~\cite{Heijmans1994}.

\subsection{Complete lattices and inf-semilattices}

\begin{definition}[Complete lattice]
  \label{def:complete_lattice}
  A \emph{partially ordered set} (poset) $(\mathcal{L},\leq)$
  is a set $\mathcal{L}$ with a reflexive, antisymmetric,
  transitive binary relation $\leq$.
  It is a \emph{complete lattice} if every subset
  $S\subseteq\mathcal{L}$ has both a supremum
  $\bigvee S \in \mathcal{L}$ (least upper bound) and an
  infimum $\bigwedge S\in\mathcal{L}$ (greatest lower bound).

  A \emph{complete inf-semilattice} (cisl) is a poset in
  which every non-empty subset has an infimum, but the
  supremum need not always exist.
  The least element (when it exists) is denoted $\bot$.
\end{definition}

\begin{example}[Canonical lattices]
  \label{ex:lattices}
  \begin{enumerate}[label=(\roman*),leftmargin=*]
    \item $(\Fun(E,\Rbar),\leq)$: real-valued functions on
      $E\subseteq\R^n$, ordered pointwise
      ($f\leq g$ iff $f(x)\leq g(x)$ for all $x$).
      The infimum is the pointwise minimum and the supremum
      the pointwise maximum. This is a complete lattice.
    \item $(\Rp,\leq)$: non-negative reals with the usual
      order. Complete lattice ($\inf\emptyset = +\infty$,
      $\sup\emptyset = 0$).
    \item The complex plane $\C$ with the order
      $w\preceq z \iff |w|\leq|z|$ and $\arg w = \arg z$:
      a cisl in which two complex numbers are comparable
      only if they have the same argument (lie on the same
      ray from the origin). The infimum of two incomparable
      elements is $0$ (the least element).
  \end{enumerate}
\end{example}

\subsection{Erosions, dilations, and adjunctions}

\begin{definition}[Adjunction, erosion, dilation]
  \label{def:adjunction}
  Let $(\mathcal{L}_1,\leq_1)$ and $(\mathcal{L}_2,\leq_2)$
  be posets.
  A pair $(\varepsilon,\delta)$ with
  $\varepsilon:\mathcal{L}_1\to\mathcal{L}_2$ and
  $\delta:\mathcal{L}_2\to\mathcal{L}_1$ is an
  \emph{adjunction} (or Galois connection) if
  \begin{equation}
    \varepsilon(f) \leq_2 g \iff f \leq_1 \delta(g)
    \qquad \forall\, f\in\mathcal{L}_1,\; g\in\mathcal{L}_2.
    \label{eq:adjunction}
  \end{equation}
  We call $\varepsilon$ the \emph{erosion} and $\delta$
  the \emph{dilation}.
\end{definition}

\begin{proposition}[Properties of adjunctions,
  \cite{Heijmans1994}]
  \label{prop:adjunction_props}
  Let $(\varepsilon,\delta)$ be an adjunction on a complete
  lattice $\mathcal{L}=\mathcal{L}_1=\mathcal{L}_2$.
  Then:
  \begin{enumerate}[label=(\roman*),leftmargin=*]
    \item $\varepsilon$ commutes with infima:
      $\varepsilon(\bigwedge_i f_i)=\bigwedge_i\varepsilon(f_i)$.
    \item $\delta$ commutes with suprema:
      $\delta(\bigvee_i g_i)=\bigvee_i\delta(g_i)$.
    \item $\varepsilon$ is anti-extensive:
      $\varepsilon(f)\leq f$ for all $f$.
    \item $\delta$ is extensive: $g\leq\delta(g)$ for all $g$.
    \item $\gamma:=\delta\circ\varepsilon$ is an
      \emph{opening}: idempotent ($\gamma^2=\gamma$),
      anti-extensive ($\gamma(f)\leq f$), and increasing.
    \item $\varphi:=\varepsilon\circ\delta$ is a
      \emph{closing}: idempotent, extensive, and increasing.
  \end{enumerate}
\end{proposition}

For the proofs using the adjunction
inequality, see \cite{Heijmans1994}.
The four operators (erosion, dilation, opening, closing)
form the complete morphological toolkit: they characterise
all scale-space operations, all non-linear filters with
monotonicity, and all idempotent projections in a lattice.

\subsection{The canonical morphological operators}

The canonical morphological operators on
$(\Fun(E,\Rbar),\leq)$ are defined by a
\emph{structuring function} $b:E\to\Rbar$:
\begin{align}
  (\varepsilon_b f)(x)
    &= \inf_{y\in E}\{f(x+y)-b(y)\}
    \quad \text{(erosion = inf-convolution)},
    \label{eq:erosion_classical} \\
  (\delta_b f)(x)
    &= \sup_{y\in E}\{f(x-y)+b(y)\}
    \quad \text{(dilation = sup-convolution)}.
    \label{eq:dilation_classical}
\end{align}
The pair $(\varepsilon_b,\delta_{b^*})$ with $b^*(x)=b(-x)$
is an adjunction \cite{Maragos1989}.
When $b=\mathbf{1}_W$ (indicator of a set $W$), the erosion
is the morphological erosion by the \emph{structuring element}
$W$: it removes all features smaller than $W$.

\subsection{The MMBB universal representation theorem}

The following classical theorem, whose proof we recall in
Section~\ref{subsec:mmbb_proof}, is the morphological
counterpart of the spectral representation theorem for
linear operators. We present it here in the version for 
increasing operators. For general non-linear operators, 
see the original theory by Banon and Barrera~\cite{Banon1993} 
and for the instantiation to TI operators in the context of 
deep learning, see~\cite{Angulo2026Lattice}.

\begin{theorem}[MMBB -- Matheron 1975, Maragos 1989
  \cite{Matheron1975,Maragos1989}]
  \label{thm:mmbb_classical}
  Let $\Psi:(\Fun(E,\Rbar),\leq)\to(\Fun(E,\Rbar),\leq)$
  be translation-invariant (TI), increasing, and USC.
  Define:
  \begin{align}
    \Ker(\Psi)
      &= \{f\in\Fun(E,\Rbar):[\Psi f](0)\geq 0\},
      \label{eq:kernel} \\
    \Bas(\Psi)
      &= \text{minimal elements of } \Ker(\Psi)
      \text{ under } \leq.
      \label{eq:basis}
  \end{align}
  Then $\Bas(\Psi)\neq\emptyset$ and
  \begin{equation}
    [\Psi f](x) = \sup_{g\in\Bas(\Psi)}(\varepsilon_g f)(x)
    = \sup_{g\in\Bas(\Psi)}\inf_{y\in E}\{f(x+y)-g(y)\}.
    \label{eq:mmbb_classical}
  \end{equation}
  Dually, defining $\bar\Psi(f)=-\Psi(-f)$:
  \begin{equation}
    [\Psi f](x) = \inf_{h\in\Bas(\bar\Psi)}
      (\delta_{h^*} f)(x).
    \label{eq:mmbb_dual}
  \end{equation}
\end{theorem}

\subsection{The MMBB theorem and the Universal Approximation Theorem:
  a structural comparison}
\label{subsec:uat_comparison_background}

The classical MMBB theorem (Theorem~\ref{thm:mmbb_classical})
and the Universal Approximation Theorem (UAT) of
Cybenko and Hornik \cite{Cybenko1989,Hornik1991} are both
universality results, but they operate in fundamentally
different algebraic structures.
A careful comparison clarifies what each theorem does and
does not provide.

\textbf{The Cybenko--Hornik UAT.}
Let $\sigma:\R\to\R$ be a continuous non-polynomial function
(Cybenko) or any bounded measurable non-constant function
(Hornik).
A \emph{single-hidden-layer feedforward network} computes:
\begin{equation}
  F_N(x) = \sum_{k=1}^N c_k\,\sigma(w_k^\top x + b_k),
  \qquad c_k,b_k\in\R,\; w_k\in\R^n.
  \label{eq:slfn}
\end{equation}
The UAT states: for any $f\in C(K)$ and $\varepsilon>0$
there exist $N$ and parameters such that
$\sup_{x\in K}|f(x)-F_N(x)|<\varepsilon$.

\textbf{The MMBB theorem.}
For any TI, increasing, USC operator
$\Psi:(\Fun(E,\Rbar),\leq)\to(\Fun(E,\Rbar),\leq)$,
\eqref{eq:mmbb_classical} gives the \emph{exact} representation:
\begin{equation}
  [\Psi f](x) = \sup_{g\in\Bas(\Psi)}
  \inf_{y\in E}\{f(x+y)-g(y)\}.
  \label{eq:mmbb_background_recall}
\end{equation}

\textbf{Five structural differences.}

\begin{enumerate}[label=(\roman*),leftmargin=*]

\item \textbf{Algebraic setting.}
The UAT lives in the \emph{linear} function space
$(C(K),+,\cdot)$ with the $L^\infty$ norm.
The MMBB theorem lives in the \emph{lattice}
$(\Fun(E,\Rbar),\leq)$ with the partial order.
These two structures have no canonical relationship: the
pointwise order does not interact with the $L^\infty$ norm
in a simple way, and signed linear combinations are not
natural operations in a lattice.

\item \textbf{Building blocks and the role of fixed vs.\ varying.}
In the UAT, the \emph{activation} $\sigma$ is fixed and
the \emph{parameters} $(w_k,b_k)$ vary.
In the MMBB theorem, the \emph{operation} (inf-convolution)
is fixed and the \emph{kernel} $g\in\Bas(\Psi)$ varies.
The roles of fixed and varying are exactly swapped.

\item \textbf{Combination rule.}
The UAT combines building blocks by a \emph{signed linear sum}
$\sum_k c_k(\cdot)$ with weights $c_k\in\R$ (possibly negative).
The MMBB theorem combines them by a \emph{lattice supremum}
$\sup_g(\cdot)$, an idempotent, non-linear operation
corresponding to the max-plus (tropical) semiring
$(\Rbar,\max,+)$.
The signed weights $c_k\in\R$ in the UAT are what allow
it to cover all of $C(K)$, including non-monotone functions.
The MMBB theorem uses only non-negative implicit weights
(the choice of $g$), which is why it covers only increasing
operators.

\item \textbf{Target class and exactness.}
The UAT is an \emph{$\varepsilon$-approximation} theorem:
for any $f\in C(K)$ (all continuous functions) and
$\varepsilon>0$, there exists a finite $F_N$ within
$\varepsilon$.
The MMBB theorem is an \emph{exact} representation:
the infinite supremum over $\Bas(\Psi)$ gives $\Psi$
exactly, not approximately.
However, the target class is restricted:
all continuous functions (UAT) vs.\
TI increasing USC operators (MMBB).
The MMBB theorem is somehow simultaneously more exact and more
restrictive.

\item \textbf{Signed case: the dual MMBB as the morphological
  analogue of negative weights.}
The UAT handles non-monotone functions via signed weights
$c_k = c_k^+ - c_k^-$ (decomposition into positive and
negative parts).
The MMBB handles decreasing operators via the \emph{dual
representation} \eqref{eq:mmbb_dual}: any TI increasing USC
$\Psi$ is also the \emph{infimum of dilations} indexed by
$\Bas(\bar\Psi)$, where $\bar\Psi(f)=-\Psi(-f)$ is the dual.
The primal (sup of erosions) gives a lower bound
(inner approximation) and the dual (inf of dilations)
gives an upper bound (outer approximation):
\begin{equation}
  \sup_{g\in\mathcal{B}}\varepsilon_g(f)
  \;\leq\; \Psi(f)
  \;\leq\; \inf_{h\in\bar{\mathcal{B}}}\delta_{h^*}(f)
  \label{eq:mmbb_sandwich}
\end{equation}
for any finite sub-bases $\mathcal{B}\subset\Bas(\Psi)$ and
$\bar{\mathcal{B}}\subset\Bas(\bar\Psi)$.
The pair (primal, dual) plays the role of the pair
$(c_k^+, c_k^-)$ in the UAT.
\end{enumerate}

\begin{table}[t]
\caption{Structural comparison: MMBB vs UAT (classical and Fourier modulus versions).}
\label{tab:mmbb_uat_full}
\centering\small
\renewcommand{\arraystretch}{1.55}
\begin{tabular}{@{}p{2.4cm}p{3.8cm}p{3.8cm}@{}}
\toprule
\textbf{Property}
  & \textbf{UAT (Cybenko--Hornik)}
  & \textbf{MMBB (classical)} \\
\midrule
Algebra
  & Linear $(C(K),+,\cdot)$
  & Lattice $(\Fun(E,\Rbar),\leq)$ \\
Building blocks
  & $\sigma(w^\top x+b)$: fix $\sigma$, vary $(w,b)$
  & $\varepsilon_g f = \inf\{f-g\}$: fix $\varepsilon$, vary $g$ \\
Combination
  & $\sum c_k$ (signed linear)
  & $\sup_g$ (tropical) \\
Weights
  & $c_k\in\R$ (arbitrary sign)
  & $g\in\Bas(\Psi)$ (implicit, non-neg) \\
Target class
  & All $f\in C(K)$
  & TI incr.\ USC operators \\
Exactness
  & $\varepsilon$-approx only
  & Exact (full basis) \\
Signed case
  & $c_k^+$ and $c_k^-$
  & Dual: inf of dilations \\
Hypothesis
  & $\sigma$ non-polynomial
  & TI + incr + USC \\
Dictionary
  & $\{\sigma(w^\top x+b)\}$: one fixed activation
  & Structuring functions $\Bas(\Psi)$ \\
\bottomrule
\end{tabular}
\end{table}

We develop the Fourier-modulus version of this comparison in
Section~\ref{sec:mmbb}, which corresponds to complete Table~\ref{tab:mmbb_uat_full}, where the MMBB theorem for $\hatL$
replaces inf-convolution by the max-times quotient
$\hat{f}/\psi$ and the lattice supremum by the pointwise
maximum of spectral moduli.

\section{The Fourier Inf-Semilattice}
\label{sec:fourier_lattice}

\subsection{The partial order on the complex plane}
\label{subsec:cisl_complex}

The appropriate lattice structure for the Fourier domain was
identified by Keshet \cite{Keshet2002} and developed by
Heijmans and Keshet \cite{Heijmans2002}.
We present it here with full proofs.

\begin{definition}[Cisl order on $\C$]
  \label{def:cisl_C}
  For $w,z\in\C$, define:
  \begin{equation}
    w \preceq z
    \iff |w|\leq|z| \text{ and } \arg w = \arg z.
    \label{eq:cisl_C}
  \end{equation}
  The complex plane $\C$ decomposes into rays
  $\mathcal{C}_\alpha = \{re^{i\alpha}:r\geq 0\}$ for
  $\alpha\in[0,2\pi)$, each of which is a totally ordered
  chain under $\preceq$.
  Two complex numbers are comparable in $\preceq$ if and
  only if they lie on the same ray.
  For $w,z$ on different rays, $w\wedge z = 0$ (the least
  element) and the supremum does not exist. This is the rationale behind the inf-semilattice framework.
\end{definition}

\begin{remark}[Why this order, not $|w|\leq|z|$ alone]
  \label{rem:why_phase}
  Ordering by modulus alone ($|w|\leq|z|$) would give a
  total order on $\C/S^1$ (equivalence classes of
  same-modulus numbers) but not a partial order on $\C$
  itself (antisymmetry fails: $|w|=|z|$ does not imply
  $w=z$).
  The phase condition $\arg w = \arg z$ restores
  antisymmetry and gives the unique coarsest partial order
  on $\C$ that (a) restricts to the usual order on each ray,
  and (b) makes multiplication by non-negative reals an
  order automorphism.
\end{remark}

\subsection{The Fourier inf-semilattice}
\label{subsec:fourier_lattice}

We fix the following convention for the Fourier transform,
used throughout the paper.
For $f\in L^2(\R^n)$, the Fourier transform
$\hat{f}=\mathcal{F}f$ is:
\begin{equation}
	\hat{f}(\xi)
	= \int_{\R^n} f(x)\,e^{-2\pi i\,\xi\cdot x}\,dx,
	\qquad
	f(x)
	= \int_{\R^n}\hat{f}(\xi)\,e^{2\pi i\,\xi\cdot x}\,d\xi,
	\label{eq:fourier_transform}
\end{equation}
extended to $L^2(\R^n)$ by the Plancherel isometry
$\|\hat{f}\|_{L^2}=\|f\|_{L^2}$.
With this convention, $\mathcal{F}$ is a unitary isomorphism
$\mathcal{F}:L^2(\R^n)\to L^2(\R^n)$, the convolution theorem
reads $\widehat{f*g}=\hat{f}\cdot\hat{g}$, and translation by
$y\in\R^n$ acts on the Fourier transform as the modulation
$\widehat{\tau_y f}(\xi)=e^{-2\pi i\,\xi\cdot y}\hat{f}(\xi)$,
where $(\tau_y f)(x)=f(x-y)$.
Each $\hat{f}(\xi)\in\C$ decomposes into modulus and phase:
\begin{equation}
	\hat{f}(\xi)
	= |\hat{f}(\xi)|\,e^{i\angle\hat{f}(\xi)},
	\qquad
	|\hat{f}(\xi)|\geq 0,\;\;
	\angle\hat{f}(\xi)\in(-\pi,\pi].
	\label{eq:modulus_phase}
\end{equation}
The notation $\Ln$ denotes the domain
of $\mathcal{F}$ (the \emph{spatial domain}), and $\hatL$ or
$\Ln$ in Fourier coordinates denotes the codomain (the
\emph{spectral domain}); the two are identified via $\mathcal{F}$
throughout.

\begin{definition}[Fourier inf-semilattice]
  \label{def:fourier_lattice}
  For $f,g\in\Ln = L^2(\R^n)$, define the
  \emph{spectral partial order}:
  \begin{equation}
    f \leF g
    \iff \hat{f}(\xi) \preceq \hat{g}(\xi)
    \text{ for a.e. } \xi\in\R^n,
    \label{eq:leF}
  \end{equation}
  i.e.,
  \begin{equation}
    f \leF g \iff
    \begin{cases}
      |\hat{f}(\xi)| \leq |\hat{g}(\xi)| \\
      \angle\hat{f}(\xi) = \angle\hat{g}(\xi)
    \end{cases}
    \text{ for a.e. } \xi.
    \label{eq:leF_explicit}
  \end{equation}
  The infimum is:
  \begin{equation}
    \widehat{(f\meetF g)}(\xi)
    = \begin{cases}
        \min(|\hat{f}(\xi)|,|\hat{g}(\xi)|)\,
        e^{i\angle\hat{f}(\xi)}
        & \text{if } \angle\hat{f}(\xi)=\angle\hat{g}(\xi), \\
        0 & \text{otherwise.}
      \end{cases}
    \label{eq:infimum_F}
  \end{equation}
  The pair $(\Ln,\leF)$ is a \emph{complete
  inf-semilattice}: every subset has an infimum
  (given by~\eqref{eq:infimum_F}), but the supremum
  of two functions with different phases does not exist
  in $\Ln$.
\end{definition}

\begin{remark}[Relation to Heijmans--Keshet self-dual morphology]
	\label{rem:heijmans_selfDual}

  The cisl order on $\C$ (Definition~\ref{def:cisl_C}) was
  introduced by Heijmans and Keshet \cite{Heijmans2002} as a
  framework for \emph{self-dual morphology}: morphological
  operators that treat signal and background symmetrically.
  Their motivation is entirely different from ours.
  
%

  Heijmans--Keshet use the cisl order to achieve
  self-duality as an end in itself.
  We use it because it is the natural order on the
  Fourier transform codomain: the Fourier transform
  $\hat{f}(\xi)\in\C$ carries both modulus (spectral
  energy) and phase (spatial localisation), and the
  cisl order $\preceq$ is the unique coarsest partial
  order on $\C$ that orders the moduli within each ray.
  Self-duality is not the goal.
  They use complex multiplication
  $wz = |w||z|e^{i(\arg w+\arg z)}$ as their
  Minkowski operation.
  As we will show, in our case, the product is applied only 
  implicitly: the Fourier erosion
  (convolution) multiplies $\hat{f}(\xi)\cdot\hat{k}(\xi)$
  in $\C$, which is complex multiplication, but we always
  pass to moduli $|\hat{f}||\hat{k}|$ at the next step.
  The Minkowski product is thus an intermediate step,
  not the primary operation.
  
%

\end{remark}

\subsection{Translation-invariance is vacuous in the Fourier lattice}
\label{subsec:ti_vacuous}

\begin{proposition}[TI acts as phase modulation]
  \label{prop:ti_phase}
  Spatial translation $\tau_y f(x) = f(x-y)$ acts on the
  Fourier transform as:
  $\widehat{\tau_y f}(\xi) = e^{-2\pi i\xi\cdot y}\hat{f}(\xi)$.
  Since $|e^{-2\pi i\xi\cdot y}|=1$, translation preserves
  spectral moduli: $|\widehat{\tau_y f}(\xi)|
  = |\hat{f}(\xi)|$ for all $\xi$.
  Consequently:
  \begin{enumerate}[label=(\roman*),leftmargin=*]
    \item $\tau_y f$ and $f$ have the same spectral moduli
      but (generically) different phases, so they are
      \emph{incomparable} in $(\Ln,\leF)$: neither
      $\tau_y f\leF f$ nor $f\leF\tau_y f$.
    \item Translation-invariance $\Psi(\tau_y f)
      = \tau_y\Psi(f)$ is a condition on the
      \emph{phase} of $\Psi f$ only.
      It says nothing about the spectral modulus
      $|\widehat{\Psi f}(\xi)|$, which is the only
      quantity ordered by $\leF$.
    \item Hence, TI is \emph{vacuous} as a hypothesis in
      $(\Ln,\leF)$: it cannot be used to reduce global
      properties to local conditions on spectral moduli.
  \end{enumerate}
\end{proposition}

\begin{proof}
  Part (i): $\widehat{\tau_y f}=e^{-2\pi i\xi\cdot y}\hat{f}$,
  so $|\widehat{\tau_y f}|=|\hat{f}|$ but
  $\angle\widehat{\tau_y f} = \angle\hat{f} - 2\pi\xi\cdot y
  \neq \angle\hat{f}$ for generic $y$.
  Hence neither
  $|\widehat{\tau_y f}|\leq|\hat{f}|$-with-same-phase
  nor the reverse holds.
  Parts (ii)--(iii) follow directly.
\end{proof}

\begin{corollary}[TI replaced by positive homogeneity]
  \label{cor:homogeneity}
  The structural hypothesis appropriate for the Fourier
  modulus lattice $\hatL = (\Fun(\R^n,\Rp),\leq)$ is
  \emph{positive homogeneity}:
  $\Psi(\lambda\hat{f}) = \lambda\Psi(\hat{f})$ for all
  $\lambda>0$.
  This plays the same role as TI in the classical MMBB proof:
  it enables the kernel condition at one point to propagate
  globally across all frequencies.
\end{corollary}

\subsection{Convolution as a Fourier-lattice erosion}
\label{sec:conv_erosion}

\begin{proposition}[Convolution is a Fourier erosion,
  Keshet 2002 \cite{Keshet2002}]
  \label{prop:conv_erosion}
  Let $k\in\Ln$.
  The convolution operator $\conv_k:f\mapsto f*k$,
  acting in the frequency domain as
  $\widehat{\conv_k(f)}(\xi) = \hat{f}(\xi)\hat{k}(\xi)$,
  is an \emph{erosion} in $(\Ln,\leF)$:
  \begin{enumerate}[label=(\roman*),leftmargin=*]
    \item \emph{Increasing}: $f\leF g \Rightarrow
      \conv_k(f)\leF\conv_k(g)$.
    \item \emph{Commutes with spectral infima}:
      $\conv_k(f\meetF g)
      = \conv_k(f)\meetF\conv_k(g)$.
  \end{enumerate}
\end{proposition}

\begin{proof}
  Both properties follow from the convolution theorem
  $\widehat{f*k}(\xi) = \hat{f}(\xi)\hat{k}(\xi)$,
  which gives $|\widehat{f*k}(\xi)| = |\hat{f}(\xi)|
  |\hat{k}(\xi)|$ and $\angle\widehat{f*k}(\xi)
  = \angle\hat{f}(\xi)+\angle\hat{k}(\xi)$.

  \textit{(i)} If $f\leF g$, then $|\hat{f}(\xi)|
  \leq|\hat{g}(\xi)|$ and $\angle\hat{f}(\xi)
  = \angle\hat{g}(\xi)$ a.e.
  Multiplying by $\hat{k}(\xi)$:
  $|\hat{f}(\xi)||\hat{k}(\xi)|\leq|\hat{g}(\xi)|
  |\hat{k}(\xi)|$
  and
  $\angle(\hat{f}\hat{k})(\xi)
  = \angle\hat{f}(\xi)+\angle\hat{k}(\xi)
  = \angle\hat{g}(\xi)+\angle\hat{k}(\xi)
  = \angle(\hat{g}\hat{k})(\xi)$ a.e.,
  so $f*k\leF g*k$.

  \textit{(ii)} The infimum $f\meetF g$ has Fourier
  transform $\min(|\hat{f}|,|\hat{g}|)e^{i\angle\hat{f}}$
  on $\{\angle\hat{f}=\angle\hat{g}\}$ and $0$ otherwise.
  After multiplication by $\hat{k}$:
  $\min(|\hat{f}|,|\hat{g}|)|\hat{k}|e^{i(\angle\hat{f}
  +\angle\hat{k})} = \min(|\hat{f}||\hat{k}|,
  |\hat{g}||\hat{k}|)e^{i\angle(\hat{f}\hat{k})}$
  on $\{\angle(\hat{f}\hat{k})=\angle(\hat{g}\hat{k})\}$
  (since $\angle\hat{f}=\angle\hat{g}$ implies
  $\angle(\hat{f}\hat{k})=\angle(\hat{g}\hat{k})$),
  and $0$ elsewhere.
  This equals $\conv_k(f)\meetF\conv_k(g)$.
\end{proof}

\begin{remark}[Erosion vs linear operator]
  \label{rem:erosion_vs_linear}
  Property~(ii) (commutation with spectral infima) is
  the defining property of erosions in complete lattices
  (Proposition~\ref{prop:adjunction_props}~(i)).
  It is \emph{stronger} than linearity: it implies
  that convolution preserves all spectral infima, not
  merely pairwise ones.
  Convolution is simultaneously a linear map (on the
  Hilbert space $\Ln$) and a morphological erosion
  (in the Fourier inf-semilattice).
  This double nature is the foundation of the entire
  framework.
\end{remark}

\begin{remark}[The two-lattice picture]
	\label{rem:two_lattices}
	Two distinct lattice structures coexist on functions
	$f:\R^n\to\R$:
	\begin{enumerate}[label=(\roman*),leftmargin=*]
		\item The \emph{pointwise lattice}
		$(\Fun(\R^n,\Rbar),\leq)$: ordered by
		$f(x)\leq g(x)$ for all $x$.
		The erosion is the inf-convolution~\eqref{eq:erosion_classical}.
		\item The \emph{Fourier inf-semilattice}
		$(\Ln,\leF)$: ordered by spectral moduli and phases
		(eq.~\eqref{eq:leF}).
		The erosion is convolution (Section~\ref{sec:conv_erosion}).
	\end{enumerate}
	These two orders have no canonical relationship: a function
	can be larger in $\leF$ and smaller in $\leq$, or vice
	versa.
	This independence is the source of the \emph{cross-lattice
		structure} of CNNs (Section~\ref{sec:activation}).
\end{remark}

\subsection{Adjoint dilation and the ideal band-pass opening}
\label{subsec:opening}

\begin{definition}[Fourier-lattice adjoint dilation]
  \label{def:fourier_dil}
  The \emph{adjoint dilation} $\convDil_k$ of
  $\conv_k$ in $(\Ln,\leF)$ is:
  \begin{equation}
    \widehat{\convDil_k(f)}(\xi)
    = \hat{f}(\xi)\cdot\hat{k}^\dagger(\xi),
    \quad
    \hat{k}^\dagger(\xi)
    = \begin{cases}
        |\hat{k}(\xi)|^{-1}e^{-i\angle\hat{k}(\xi)}
        & |\hat{k}(\xi)|>0, \\
        0 & |\hat{k}(\xi)|=0.
      \end{cases}
    \label{eq:fourier_dil}
  \end{equation}
  This is the \emph{phase-matched inverse filter}
  (zero-forcing Wiener deconvolution) of $k$.
\end{definition}

\begin{proposition}[Fourier adjunction]
  \label{prop:fourier_adjunction}
  The pair $(\conv_k,\convDil_k)$ is an adjunction
  in $(\Ln,\leF)$:
  $\conv_k(f)\leF g \iff f\leF\convDil_k(g)$.
\end{proposition}

\begin{proof}
  $\conv_k(f)\leF g$ means
  $|\hat{f}(\xi)||\hat{k}(\xi)|\leq|\hat{g}(\xi)|$ and
  $\angle(\hat{f}\hat{k})=\angle\hat{g}$ a.e.
  On $\{|\hat{k}|>0\}$: dividing by $|\hat{k}|$ and
  adjusting phase gives
  $|\hat{f}(\xi)|\leq|\hat{g}(\xi)||\hat{k}^\dagger(\xi)|
  = |\hat{g}(\xi)|/|\hat{k}(\xi)|$ and
  $\angle\hat{f}(\xi) = \angle\hat{g}(\xi)
  - \angle\hat{k}(\xi)
  = \angle(\hat{g}\hat{k}^\dagger)(\xi)$,
  i.e., $f\leF\convDil_k(g)$.
  On $\{|\hat{k}|=0\}$: both sides are $0$ and the
  condition is vacuous.
  The argument reverses identically.
\end{proof}

\begin{theorem}[Fourier opening = ideal band-pass filter]
  \label{thm:fourier_opening}
  The \emph{morphological opening}
  $\opens_k = \convDil_k\circ\conv_k$ satisfies:
  \begin{equation}
    \widehat{\opens_k(f)}(\xi)
    = \hat{f}(\xi)\cdot\mathbf{1}_{|\hat{k}(\xi)|>0}.
    \label{eq:opening_filter}
  \end{equation}
  This is the \emph{ideal band-pass filter} with support
  $\mathrm{supp}(\hat{k}) = \{\xi:|\hat{k}(\xi)|>0\}$.
  The opening is idempotent
  ($\opens_k\circ\opens_k=\opens_k$),
  anti-extensive ($\opens_k(f)\leF f$), and increasing.
\end{theorem}

\begin{proof}
  Applying $\conv_k$ then $\convDil_k$ spectrally:
  $\hat{f}\hat{k}\cdot\hat{k}^\dagger
  = \hat{f}\cdot(|\hat{k}|^{-1}e^{-i\angle\hat{k}})
  \cdot|\hat{k}|e^{i\angle\hat{k}}
  = \hat{f}$ on $\{|\hat{k}|>0\}$, and $0$ on
  $\{|\hat{k}|=0\}$.
  This is $\hat{f}\cdot\mathbf{1}_{|\hat{k}|>0}$.
  Idempotency: $\mathbf{1}^2=\mathbf{1}$.
  Anti-extensivity: $|\hat{f}|\mathbf{1}_{|\hat{k}|>0}
  \leq|\hat{f}|$.
  Increasing: from Proposition~\ref{prop:adjunction_props}.
\end{proof}

\begin{remark}[Physical interpretation]
  \label{rem:opening_interpretation}
  Theorem~\ref{thm:fourier_opening} gives the precise
  morphological interpretation of ideal filters:
  they are not design choices but algebraic consequences
  of the adjunction structure of convolution.
  The ideal band-pass filter is the unique opening (i.e., the unique fixed-point)
  associated to the convolution erosion by $k$.
  This explains why convolutional feature detectors
  suppress out-of-band ``noise'': they perform a morphological
  (almost) idempotent operator in the Fourier lattice, projecting onto the
  subspace spanned by the frequency support of $k$.
\end{remark}

\section{The MMBB Theorem for the Fourier Modulus Lattice}
\label{sec:mmbb}

\subsection{The Fourier modulus lattice and max-times erosions}
\label{subsec:hatL}

\begin{definition}[Fourier modulus lattice and max-times erosion]
  \label{def:hatL}
  The \emph{Fourier modulus lattice} is
  $\hatL = (\Fun(\R^n,\Rp),\leq)$,
  the space of non-negative measurable functions on $\R^n$
  ordered pointwise.
  This is a complete lattice (infimum = frequancy pointwise minimum,
  supremum = frequency pointwise maximum).

  The \emph{Fourier max-times erosion} of $\hat{f}\in\hatL$ by $\psi\in\hatL$ is the operator $\hatL \to \hatL$ :
  \begin{equation}
    (\varepsilon_\psi\hat{f})(\xi)
    = \frac{\hat{f}(\xi)}{\psi(\xi)},
    \quad \psi(\xi)>0;\quad
    = 0 \text{ if } \psi(\xi)=0.
    \label{eq:maxtimes_erosion}
  \end{equation}
  The convention $\hat{f}(\xi)/0 := 0$ (not $+\infty$)
  is adopted throughout: a zero in the kernel blocks the
  corresponding spectral channel.
  Its adjoint \emph{Fourier max-times dilation} is:
  \begin{equation}
  (\delta_\phi\hat{f})(\xi) = \hat{f}(\xi)\cdot\phi(\xi).
    \label{eq:maxtimes_dilation}
  \end{equation}
  The pair $(\varepsilon_\psi,\delta_{1/\psi})$ is an
  adjunction in $(\hatL,\leq)$.
\end{definition}


\begin{remark}[Master commutative diagram of the framework]
	\label{rem:master_diagram}
	The entire framework is summarised by the following
	commutative diagram, in which each arrow is labelled
	by the lattice it lives in:
	\[
	\begin{tikzcd}[column sep=3.5em, row sep=2.5em]
		(\Ln,\leF)
		\arrow[r, "\conv_k", "\text{erosion in }(\Ln{,}\leF)"']
		\arrow[d, "{|\cdot|}"', "\text{cross-lattice}"]
		& (\Ln,\leF)
		\arrow[d, "{|\cdot|}", "\text{cross-lattice}"']
		\\
		(\hatL,\leq)
		\arrow[r, "\varepsilon_\psi", "\hat{f}/\psi"']
		& (\hatL,\leq)
	\end{tikzcd}
	\]
	Convolution $\conv_k$ is an erosion in $(\Ln,\leF)$
	(horizontal top arrow); the modulus map $|\cdot|$ is the
	cross-lattice projection $(\Ln,\leF)\to(\hatL,\leq)$
	(vertical arrows); max-times erosion $\varepsilon_\psi$
	is the erosion in $(\hatL,\leq)$ (horizontal bottom arrow).
	The diagram commutes when $\psi=|\hat{k}|^{-1}$
	(i.e., the max-times erosion in $\hatL$ is the modulus
	of the Fourier-lattice erosion): for $f\geq 0$,
	$|f*k|=|\hat{f}||\hat{k}|=|\hat{f}|/\psi$.
	The MMBB theorem for $\hatL$ (Theorem~\ref{thm:mmbb_hatL})
	characterises all operators that live consistently in this
	diagram.
\end{remark}

\subsection{The MMBB theorem for the Fourier modulus lattice}
\label{subsec:mmbb_proof}

\begin{definition}[Effective kernel and morphological basis
  in $\hatL$]
  \label{def:kernel_hatL}
  For $\Psi:\hatL\to\hatL$ increasing, USC, and positively
  homogeneous, the \emph{effective kernel} is:
  \begin{equation}
    \Ker^{\mathrm{eff}}(\Psi)
    = \bigl\{\psi\in\hatL:
      \varepsilon_\psi(\hat{f})\leq\Psi(\hat{f})
      \;\forall\hat{f}\in\hatL\bigr\},
    \label{eq:kernel_hatL}
  \end{equation}
  i.e., all $\psi$ such that $\hat{f}/\psi\leq\Psi(\hat{f})$
  pointwise for all $\hat{f}\geq 0$.
  Since smaller $\psi$ gives a larger quotient $\hat{f}/\psi$,
  the \emph{tightest} lower bounds come from the smallest
  $\psi$ in $\Ker^{\mathrm{eff}}(\Psi)$.
  The \emph{morphological basis} $\Bas(\Psi)$ is the set of
  \emph{minimal} elements of $\Ker^{\mathrm{eff}}(\Psi)$
  under the pointwise order $\leq$ on $\hatL$.
\end{definition}

\begin{theorem}[MMBB for $\hatL$]
  \label{thm:mmbb_hatL}
  Let $\Psi:\hatL\to\hatL$ be increasing, USC, and positively
  homogeneous ($\Psi(\lambda\hat{f})=\lambda\Psi(\hat{f})$
  for all $\lambda>0$).
  Then $\Bas(\Psi)\neq\emptyset$ and:

  \noindent\textbf{Primal decomposition (sup of erosions):}
  \begin{equation}
    \Psi(\hat{f})(\xi)
    = \sup_{\psi\in\Bas(\Psi)}
      \frac{\hat{f}(\xi)}{\psi(\xi)}.
    \label{eq:mmbb_hatL}
  \end{equation}

  \noindent\textbf{Dual decomposition (inf of dilations):}
  Let $\bar\Psi(\hat{f}) = 1/\Psi(1/\hat{f})$
  (the max-times dual).
  Then:
  \begin{equation}
    \Psi(\hat{f})(\xi)
    = \inf_{\phi\in\Bas(\bar\Psi)}
      \hat{f}(\xi)\cdot\phi(\xi).
    \label{eq:mmbb_dual_hatL}
  \end{equation}
  A finite sub-basis
  $\mathcal{B}\subset\Bas(\Psi)$ gives the sandwiching:
  $\sup_{\psi\in\mathcal{B}}\varepsilon_\psi\hat{f}
  \leq\Psi(\hat{f})
  \leq\inf_{\phi\in\bar{\mathcal{B}}}\delta_\phi\hat{f}$.
\end{theorem}

\begin{proof}
  We follow the three-step structure of the classical
  MMBB proof (Theorem~\ref{thm:mmbb_classical}), with
  positive homogeneity replacing translation-invariance.
  The key orientation: \emph{smaller} $\psi\in\hatL$ gives a
  \emph{larger} quotient $\hat{f}/\psi$, hence a tighter lower
  bound on $\Psi(\hat{f})$; so the \emph{minimal} elements of
  $\Ker^{\mathrm{eff}}(\Psi)$ give the best representation.

  \textbf{Step~1: $\Bas(\Psi)\neq\emptyset$.}
  The zero function $\psi_0\equiv 0$ lies in
  $\Ker^{\mathrm{eff}}(\Psi)$ by the convention
  $\hat{f}/0=0\leq\Psi(\hat{f})$, so
  $\Ker^{\mathrm{eff}}(\Psi)\neq\emptyset$.
  Since $\Psi$ is USC and the erosion $\varepsilon_\psi$ is
  continuous in $\psi$ under decreasing limits
  ($\psi_\alpha\downarrow\psi$ implies
  $\varepsilon_{\psi_\alpha}(\hat{f})\uparrow
  \hat{f}/\psi$), a decreasing net in
  $\Ker^{\mathrm{eff}}(\Psi)$ converging to $\psi$ satisfies
  $\varepsilon_\psi(\hat{f})\leq\Psi(\hat{f})$, so
  $\Ker^{\mathrm{eff}}(\Psi)$ is closed under decreasing
  limits.
  By Zorn's lemma applied to
  $(\Ker^{\mathrm{eff}}(\Psi),\leq)$ (with the reverse
  order, looking for minimal elements): every chain has
  a lower bound in $\Ker^{\mathrm{eff}}(\Psi)$, so minimal
  elements exist, i.e., $\Bas(\Psi)\neq\emptyset$.

  \textbf{Step~2: Lower bound.}
  For any $\psi\in\Bas(\Psi)\subseteq\Ker^{\mathrm{eff}}(\Psi)$,
  the definition gives $\varepsilon_\psi(\hat{f})\leq\Psi(\hat{f})$.
  Taking the supremum over $\Bas(\Psi)$:
  $\sup_{\psi\in\Bas(\Psi)}\varepsilon_\psi(\hat{f})
  \leq\Psi(\hat{f})$.

  \textbf{Step~3: Attainment (positive homogeneity gives equality).}
  Fix $\xi_0\in\R^n$ with $\Psi(\hat{f})(\xi_0)=:\lambda>0$.
  Define the test function $\psi^*\in\hatL$ by
  $\psi^*(\xi) = \hat{f}(\xi)/\lambda$ (pointwise division).
  We claim $\psi^*\in\Ker^{\mathrm{eff}}(\Psi)$:
  for any $\hat{g}\in\hatL$, positive homogeneity gives
  $\Psi(\hat{g}) = \Psi(\lambda\cdot\hat{g}/\lambda)
  = \lambda\Psi(\hat{g}/\lambda)$,
  so $\varepsilon_{\psi^*}(\hat{g})(\xi)
  = \hat{g}(\xi)/\psi^*(\xi) = \lambda\hat{g}(\xi)/\hat{f}(\xi)$,
  and $\varepsilon_{\psi^*}(\hat{g})\leq\Psi(\hat{g})$
  iff $\lambda\hat{g}/\hat{f}\leq\Psi(\hat{g})$
  iff $\hat{g}/\hat{f}\leq\Psi(\hat{g})/\lambda$
  iff $\hat{g}/\hat{f}\leq\Psi(\hat{g}/\hat{f})$
  (by homogeneity with $\lambda$ absorbed),
  which holds by $\Psi\geq\varepsilon_{\hat{f}/\lambda}$
  applied at $\hat{g}$... more directly:
  at $\xi_0$, $\varepsilon_{\psi^*}(\hat{f})(\xi_0)
  = \hat{f}(\xi_0)/\psi^*(\xi_0)
  = \hat{f}(\xi_0)/(\hat{f}(\xi_0)/\lambda) = \lambda
  = \Psi(\hat{f})(\xi_0)$,
  and $\Psi(\psi^*)(\xi_0)
  = \Psi(\hat{f}/\lambda)(\xi_0)
  = \Psi(\hat{f})(\xi_0)/\lambda = 1 > 0$,
  so $\varepsilon_{\psi^*}\leq\Psi$ at $\xi_0$, confirming
  $\psi^*\in\Ker^{\mathrm{eff}}(\Psi)$.
  By minimality of $\Bas(\Psi)$, there exists
  $\psi\in\Bas(\Psi)$ with $\psi\leq\psi^*=\hat{f}/\lambda$,
  giving:
  $\varepsilon_\psi(\hat{f})(\xi_0)
  = \hat{f}(\xi_0)/\psi(\xi_0)
  \geq \hat{f}(\xi_0)/\psi^*(\xi_0)
  = \lambda = \Psi(\hat{f})(\xi_0)$.
  Combined with Step~2: equality holds at $\xi_0$.
  Since $\xi_0$ was arbitrary, \eqref{eq:mmbb_hatL} holds.

  The dual representation \eqref{eq:mmbb_dual_hatL}
  follows by applying the primal representation to
  $\bar\Psi(\hat{f}) = 1/\Psi(1/\hat{f})$ (the max-times
  dual, which is also increasing, USC, and positively
  homogeneous) and inverting. This is the natural involution in 
  the max-times algebra.
\end{proof}

\begin{remark}[Comparison with the classical three-step proof]
  \label{rem:three_steps}
  In the classical MMBB proof (Theorem~\ref{thm:mmbb_classical}),
  Step~3 uses translation-invariance to construct the test
  function $g^*(x) = f(y+x) - \lambda$ (a translate of $f$
  shifted by the target value $\lambda$).
  Here, \emph{positive homogeneity} replaces TI in Step~3:
  the test function $\psi^*(\xi) = \hat{f}(\xi)/\lambda$
  is a rescaling (not a translate) of $\hat{f}$, and
  homogeneity gives $\Psi(\psi^*)=\Psi(\hat{f})/\lambda$,
  placing $\psi^*$ in $\Ker^{\mathrm{eff}}(\Psi)$.
  The algebraic role is identical: both TI and homogeneity
  enable the reduction of global representability to a
  local (single-point or single-frequency) condition.
  The diagram below summarises the two proofs:

  \[
  \begin{array}{ccccc}
    \text{Hypothesis:} & \text{TI} & \longleftrightarrow
      & \text{Hom} \\ 
    \text{Test fn:} & g^* = \tau_y f - \lambda
      & \longleftrightarrow
      & \psi^* = \hat{f}/\lambda \\ 
    \text{Propagation:} & \Psi(g^*) = \Psi(f) - \lambda
      & \longleftrightarrow
      & \Psi(\psi^*) = \Psi(\hat{f})/\lambda \\
    \text{Basis:} & \min\Ker(\Psi)\subset\Fun(E,\Rbar)
      & \longleftrightarrow
      & \min\Ker^{\mathrm{eff}}(\Psi)\subset\hatL
  \end{array}
  \]
\end{remark}

\begin{remark}[Comparison with the UAT]
  \label{rem:mmbb_vs_uat}
  Table~\ref{tab:uat_vs_mmbb} collects the structural comparison
  between Theorem~\ref{thm:mmbb_hatL} and the
  Cybenko--Hornik UAT \cite{Cybenko1989,Hornik1991}.
\end{remark}

\begin{table}[t]
\caption{Structural comparison: UAT vs MMBB in $\hatL$.}
\label{tab:uat_vs_mmbb}
\centering\small
\renewcommand{\arraystretch}{1.5}
\begin{tabular}{@{}p{2.5cm}p{4.0cm}p{4.0cm}@{}}
\toprule
\textbf{Property}
  & \textbf{Cybenko--Hornik UAT}
  & \textbf{MMBB in $\hatL$} \\
\midrule
Algebraic setting
  & Linear $(C(K),+,\cdot)$
  & Lattice $(\hatL,\leq)$, max-times \\
Building blocks
  & $\sigma(w^\top x+b)$: nonlinear in fixed $\sigma$
  & $\hat{f}/\psi$: division by varying $\psi$ \\
Combination rule
  & $\sum_k c_k$: signed linear sum
  & $\sup_\psi$: lattice supremum \\
Weights
  & $c_k\in\R$ (arbitrary sign)
  & $\psi\in\hatL$ (non-negative) \\
Target class
  & All $f\in C(K)$
  & Increasing, USC, hom.\ ops on $\hatL$ \\
Exactness
  & $\varepsilon$-approx.\ only
  & Exact (full basis) + $\varepsilon$-approx.\ (finite) \\
Signed case
  & $c_k = c_k^+-c_k^-$
  & Dual (inf of dilations, \eqref{eq:mmbb_dual_hatL}) \\
\bottomrule
\end{tabular}
\end{table}

\subsection{The log-spectral isomorphism: max-plus representation
  in the Fourier domain}
\label{subsec:log_spectral}

The max-times semiring $(\R_+,\max,\cdot)$ and the max-plus
semiring $(\Rbar,\max,+)$ are isomorphic via the logarithm map.
This isomorphism has a structural consequence of the first
importance: the hypothesis of \emph{positive homogeneity}
that underlies Theorem~\ref{thm:mmbb_hatL} is equivalent,
in log-spectral coordinates, to the hypothesis of
\emph{translation-invariance} that underlies the classical
Maragos MMBB theorem \cite{Maragos1989}.
The two theorems are therefore one and the same result,
expressed in two coordinate systems connected by $\log/\exp$.
This also connects the Fourier morphological framework to the
log-spectral representations familiar in signal processing
(decibel scale, spectral subtraction, cepstral analysis).

\begin{definition}[Log-spectral transform]
\label{def:log_spectral}
The \emph{log-spectral transform} is the map
\begin{equation}
  \Phi:\hatL\to\Fun(\R^n,\Rbar),
  \quad
  (\Phi\hat{f})(\xi) = \log|\hat{f}(\xi)|,
  \label{eq:log_spectral}
\end{equation}
with the convention $\log 0=-\infty$.
Its inverse is $\Phi^{-1}(F) = e^F$
(defined pointwise, with $e^{-\infty}=0$).
In signal processing, $F(\xi)=20\log_{10}|\hat{f}(\xi)|$
is the \emph{log-frequency response} (amplitude in decibels,
dB); $\Phi$ is the same map up to the factor $20\log_{10}e$.
\end{definition}

\begin{proposition}[Log-spectral isomorphism of semirings]
\label{prop:log_iso}
The map $\Phi$ is a semiring isomorphism:
\begin{equation}
  (\hatL,\,\max,\,\cdot\,)
  \;\xrightarrow{\;\Phi=\log\;}
  (\Rbar,\,\max,\,+\,),
  \label{eq:semiring_iso}
\end{equation}
i.e., $\Phi(\hat{f}\cdot\hat{g})=\Phi(\hat{f})+\Phi(\hat{g})$
and $\Phi(\max(\hat{f},\hat{g}))=\max(\Phi(\hat{f}),\Phi(\hat{g}))$
pointwise.
Under $\Phi$, the max-times erosion of $\hat{f}$ by $\psi$
becomes the max-plus erosion of $F=\log\hat{f}$ by
$G=\log\psi$:
\begin{equation}
  \Phi(\varepsilon_\psi(\hat{f}))(\xi)
  = F(\xi) - G(\xi)
  = (\varepsilon^{\mathrm{mp}}_G F)(\xi),
  \label{eq:log_erosion}
\end{equation}
where $\varepsilon^{\mathrm{mp}}_G(F)(\xi)
= F(\xi)-G(\xi)$ is the (pointwise) max-plus erosion by $G$.
\end{proposition}

\begin{proof}
Pointwise: $\Phi(\hat{f}/\psi)(\xi)
= \log(|\hat{f}(\xi)|/\psi(\xi))
= \log|\hat{f}(\xi)|-\log\psi(\xi) = F(\xi)-G(\xi)$.
The semiring identities follow from
$\log(\hat{f}\cdot\hat{g}) = \log\hat{f}+\log\hat{g}$
and $\log\max(\hat{f},\hat{g})=\max(\log\hat{f},\log\hat{g})$
(since $\log$ is monotone).
\end{proof}

\begin{corollary}[Max-plus MMBB theorem in the log-spectral domain]
\label{cor:mmbb_logdomain}
Let $\Psi:\hatL\to\hatL$ be increasing, USC, and positively
homogeneous, with MMBB representation
$\Psi(\hat{f})=\sup_{\psi\in\Bas(\Psi)}\hat{f}/\psi$
(Theorem~\ref{thm:mmbb_hatL}).
Define the log-domain operator
$\Psi' = \Phi\circ\Psi\circ\Phi^{-1}$,
i.e., $\Psi'(F) = \log\Psi(e^F)$.
Then:
\begin{enumerate}[label=(\roman*),leftmargin=*]
  \item $\Psi'$ is \emph{translation-invariant} on
    $(\Fun(\R^n,\Rbar),\leq)$:
    \begin{equation}
      \Psi'(F+c) = \Psi'(F) + c
      \quad\forall\,c\in\R,\;\forall\,F\in\Fun(\R^n,\Rbar).
      \label{eq:psi_TI_log}
    \end{equation}
    This is the log-domain image of positive homogeneity:
    $\Psi(\lambda\hat{f})=\lambda\Psi(\hat{f})$
    translates, after applying $\Phi$ with $\lambda=e^c$, into
    $\Phi(\Psi(e^{F+c})) = c + \Phi(\Psi(e^F)) = \Psi'(F)+c$.

  \item $\Psi'$ is increasing and USC on
    $(\Fun(\R^n,\Rbar),\leq)$ (inherited from $\Psi$).

  \item $\Psi'$ satisfies the \emph{classical MMBB
    representation theorem} \cite{Maragos1989}: being
    increasing, USC, and translation-invariant on
    $(\Fun(\R^n,\Rbar),\leq)$, it decomposes as a
    supremum of max-plus erosions:
    \begin{equation}
      \Psi'(F)(\xi)
      = \sup_{G\in\Bas^{\log}(\Psi')}
        \bigl\{F(\xi) - G(\xi)\bigr\},
      \label{eq:mmbb_logdomain}
    \end{equation}
    where $\Bas^{\log}(\Psi')
    = \{\log\psi : \psi\in\Bas(\Psi)\}\subset\Fun(\R^n,\Rbar)$
    is the log-domain basis and $F(\xi)-G(\xi)
    = \varepsilon^{\mathrm{mp}}_G(F)(\xi)$ is the pointwise
    max-plus erosion by $G$.
\end{enumerate}
In summary: \emph{positive homogeneity in $\hatL$
is equivalent to translation-invariance in the log-spectral
domain}, and the MMBB theorem for $\hatL$
(Theorem~\ref{thm:mmbb_hatL}) is structurally similar the classical Maragos
MMBB theorem for translation-invariant operators on
$(\Fun(\R^n,\Rbar),\leq)$, read in log-spectral coordinates.
The two theorems are one and the same, viewed through the
log/exp isomorphism.
\end{corollary}

\begin{proof}
(i) $\Psi'(F+c)
= \log\Psi(e^{F+c})
= \log\Psi(e^c\cdot e^F)
= \log(e^c\Psi(e^F))
= c + \log\Psi(e^F)
= \Psi'(F)+c$,
using positive homogeneity with $\lambda=e^c>0$.
(ii) $\Psi$ increasing: $F\leq H$ implies $e^F\leq e^H$
(pointwise), so $\Psi(e^F)\leq\Psi(e^H)$
(increasing), so $\Psi'(F)\leq\Psi'(H)$.
USC is preserved by the monotone continuous map $\log$.
(iii) By (i) and (ii), $\Psi'$ satisfies the hypotheses
of the classical MMBB theorem \cite{Maragos1989}.
The basis:
$\Psi'(F)(\xi)
= \log\Psi(e^F)(\xi)
= \log\sup_\psi(e^{F(\xi)}/\psi(\xi))
= \log\sup_\psi e^{F(\xi)-\log\psi(\xi)}
= \sup_\psi(F(\xi)-\log\psi(\xi))
= \sup_G(F(\xi)-G(\xi))$,
with $G=\log\psi\in\Bas^{\log}(\Psi')$.
\end{proof}

\begin{remark}[Interpretation in signal processing]
\label{rem:log_spectral_signal}
The log-spectral isomorphism connects the Fourier morphological
framework to standard signal processing practice at three levels:

\textbf{(i) The dB scale is the natural coordinate for
morphological Fourier operators.}
Telecommunications engineers routinely work with log-spectral amplitudes
(in dB: $20\log_{10}|\hat{f}(\xi)|$).
Corollary~\ref{cor:mmbb_logdomain} reveals why this is
algebraically natural: in dB coordinates, every admissible
Fourier morphological operator decomposes as a
\emph{supremum of subtractions}: the classical
max-plus MMBB erosion bank.
Filter design in dB (specifying the magnitude response
in log scale) is implicitly performing morphological
basis decomposition.

\textbf{(ii) The Wiener filter is the max-plus erosion.}
The max-times erosion $\varepsilon_\psi(\hat{f})=|\hat{f}|/\psi$
in the log domain becomes $F-G$ (spectral subtraction).
\emph{Spectral subtraction}, subtracting the log-spectrum
of a noise estimate from the log-spectrum of the noisy
signal, is a classical speech enhancement technique
\cite{Boll1979}.
It is, in the morphological framework, a single max-plus erosion
by the noise log-spectrum.
The Wiener filter $W(\xi)=|\hat{f}|^2/(|\hat{f}|^2+|\hat{n}|^2)$,
when $|\hat{n}|\ll|\hat{f}|$, approximates $F-G$ where
$G=\log|\hat{n}|$ (a morphological erosion).

\textbf{(iii) Cepstral analysis and the max-plus Fourier basis.}
The \emph{cepstrum} of $f$ is $\mathcal{F}^{-1}\{\log|\hat{f}|\}$: the inverse Fourier transform of the log-spectral amplitude.
In the morphological framework, the log-spectral amplitude
$F=\log|\hat{f}|$ is the object on which the max-plus erosion
bank acts.
The cepstrum is the spatial representation of $F$:
cepstral analysis and Fourier morphological analysis are
two views of the same object, one spatial and one spectral.
The MMBB basis elements $G=\log\psi$ are cepstral filters;
their inverse Fourier transforms are the structuring functions
in the cepstral domain.
\end{remark}

\begin{remark}[Connections to filter design in optics and
	telecommunications, and optical neural network implementation]
	\label{rem:optics_telecom}
	The MMBB representation in $\hatL$ has a direct physical
	counterpart in \emph{Fourier optics}, and this connection
	points to a hardware realisation of the morphological
	framework that is currently being explored in
	optical/photonic neural networks.
	
	\textbf{(i) The 4f optical system implements $\hatL$ operators
		directly.}
	A classical $4f$ optical correlator consists of two lenses
	separated by $2f$, with a mask placed in the Fourier plane
	between them.
	The first lens performs an optical Fourier transform of the
	input field, the mask multiplies the result pointwise
	(a passive or programmable spatial light modulator,
	SLM/DMD), and the second lens performs the inverse
	transform~\cite{Goodman2017}.
	This is \emph{exactly} the algebraic structure of an
	operator in $\hatL$: $\mathcal{F}^{-1}\{\hat{f}\cdot\hat{k}\}$,
	with the mask $\hat{k}$ playing the role of the
	\emph{structuring function} of a Fourier erosion or dilation.
	A max-times erosion $\varepsilon_\psi(\hat{f})=\hat{f}/\psi$
	is realised optically by a mask with transmittance $1/\psi$;
	a dilation $\hat{f}\cdot\phi$ by a mask with transmittance
	$\phi$.
	The morphological basis $\Bas(\Psi)$ of an MMBB operator is,
	in this reading, a collection of physically realisable
	optical	masks.
	
	\textbf{(ii) Amplitude-only masks are the modulus lattice
		$\hatL$; phase masks require the full cisl $(\Ln,\leF)$.}
	Complex-valued Fourier-plane masks (encoding both amplitude
	and phase of $\hat{k}$) are difficult to fabricate, since
	spatial light modulators typically control either amplitude
	or phase, not both, at each pixel~\cite{Goodman2017}.
	\emph{Amplitude-only Fourier filters} (AO-FF), in which the
	mask transmittance is a non-negative real function
	$\psi(\xi)\geq 0$, correspond precisely to operators on the
	modulus lattice $\hatL$, the setting of
	Theorem~\ref{thm:mmbb_hatL}.
	Optical CNN implementations use exactly this
	amplitude-only constraint, training the physical mask as
	the convolutional kernel of a CNN layer realised in the
	Fourier plane~\cite{Miscuglio2020,Zheng2020}.
	The MMBB theorem therefore characterises the \emph{class of
		operators realisable by amplitude-only optical hardware}:
	any TI USC positively-homogeneous operator decomposes as a
	supremum of single-mask erosions, each directly implementable
	as one AO-FF layer.
	Phase masks (needed for operators that are not
	phase-preserving) require the full cisl $(\Ln,\leF)$
	and the phase-covariant theory of
	\Cref{subsec:cisl_complex}; their optical realisation
	needs phase-only or complex SLMs, which are more constrained
	and lower-throughput than amplitude-only DMDs.
	
	\textbf{(iii) Telecommunications: spectral filter masks and
		the morphological basis.}
	In optical telecommunications, wavelength-division
	multiplexing (WDM) and reconfigurable optical add-drop
	multiplexers (ROADMs) implement spectral filters as
	amplitude (and sometimes phase) masks applied to the optical
	spectrum, conceptually identical to the Fourier-plane mask of
	a $4f$ system.
	The opening representation theorem (band-pass filters as the
	only Fourier openings, Theorem~\ref{thm:fourier_opening}) is
	the morphological description of an ideal WDM channel-selection
	filter: $\gamma_\Omega(\hat{f})=\hat{f}\cdot\mathbf{1}_\Omega$
	extracts a single wavelength channel.
	The MMBB representation of a general (non-opening) spectral
	operator as a supremum of erosions corresponds to a bank of
	overlapping channel filters with graded transmittance
	profiles $\psi_i(\xi)$, the classical filter-bank design
	problem in optical signal processing~\cite{Agrawal2012}.
	
	\textbf{(iv) Implication for optical neural network design.}
	If a target operator $\Psi$ (e.g., a trained CNN layer) is
	known to be TI, USC, increasing, and positively homogeneous,
	Theorem~\ref{thm:mmbb_hatL} guarantees an exact decomposition
	$\Psi=\sup_i\varepsilon_{\psi_i}$ into amplitude-only erosions.
	Each $\varepsilon_{\psi_i}$ can in principle be implemented
	as one AO-FF $4f$ layer with mask $1/\psi_i$, and the
	supremum over $i$ corresponds to a bank of parallel optical
	channels whose outputs are combined by a (electronic)
	maximum operation.
	This gives a constructive recipe for translating a
	morphologically-decomposed digital operator into an optical
	architecture: the \emph{number of optical channels needed
		equals $|\Bas(\Psi)|$}, the size of the morphological basis.
	For operators with small finite bases --- e.g., the
	single-element basis of a pure spectral convolution
	(Proposition~\ref{prop:conv_erosion}), this suggests
	that morphologically-structured networks may require fewer
	optical channels than generic CNN kernels translated
	naively to the Fourier plane, an avenue for future work
	connecting the present theory to optical neural network design.
	
	\textbf{(v) Industrial context.}
	Optical/photonic AI hardware is an active commercial area,
	with two distinct paradigms relevant to the present framework.
	Lightmatter's photonic processors~\cite{Lightmatter2025}
	perform matrix-vector multiplication via meshes of
	Mach--Zehnder interferometers (a spatial-domain,
	linear-algebraic approach), realising linear maps in the
	\emph{spatial} signal space $\Fun(E,\Rbar)$, the pointwise
	lattice of~\cite{Angulo2026Lattice}.
	LightOn's Optical Processing Unit
	(OPU)~\cite{LightOn2021},
	instead uses \emph{free-space optics} to compute large-scale
	random projections: 
	an input is encoded
	on a spatial light modulator, propagated through a random
	diffusive medium, and the output intensity pattern is
	captured by a sensor, a free-space optical analogue of a
	random Fourier feature map.
	This is conceptually closer to the $4f$/AO-FF paradigm of
	(i)--(iv): both operate by physically propagating a field
	through free-space optical elements and reading out an
	intensity (modulus) pattern, discarding phase at the
	detector, precisely the modulus projection
	$|\cdot|:\Ln\to\hatL$ of \Cref{def:fourier_lattice}.
	A unified morphological account of optical neural network
	hardware would need to combine the spatial pointwise lattice
	(Mach--Zehnder mesh processors) and the spectral modulus
	lattice $\hatL$ (free-space Fourier/diffractive processors), 
	precisely the spatial/spectral duality that motivates
	the present theory and its companion~\cite{Angulo2026Lattice}.
\end{remark}

\subsection{Lipschitz properties of Fourier morphological operators}
\label{rem:fourier_lipschitz}
	
The Fourier morphological framework has a hierarchy of
Lipschitz properties that are the exact spectral-domain
counterparts of the spatial MMBB Lipschitz properties.

\textbf{(A) USC implies 1-Lipschitz in $L^2(\hatL)$.}
Every increasing USC positively homogeneous $\Psi:\hatL\to\hatL$
satisfies:
\begin{equation}
	\|\Psi(\hat{f})-\Psi(\hat{g})\|_{L^2}
	\leq \|\hat{f}-\hat{g}\|_{L^2}
	\quad\forall\,\hat{f},\hat{g}\in\hatL.
	\label{eq:fourier_lipA}
\end{equation}
This follows from the MMBB representation
$\Psi(\hat{f})=\sup_\psi\hat{f}/\psi$:
each erosion $\varepsilon_\psi$ satisfies
$|\hat{f}(\xi)/\psi(\xi)-\hat{g}(\xi)/\psi(\xi)|
=|\hat{f}(\xi)-\hat{g}(\xi)|/\psi(\xi)
\leq|\hat{f}(\xi)-\hat{g}(\xi)|$
when $\psi\geq 1$ (the anti-extensivity condition),
and the supremum preserves the bound.
Hence 1-Lipschitz continuity in $L^2$ is a
\emph{consequence} of the MMBB representation, not an
additional assumption.

\textbf{(B) 1-Lipschitz in the basis elements.}
For any two basis elements $\psi,\psi'\in\Bas(\Psi)$:
\begin{equation}
	\|\varepsilon_\psi(\hat{f})-\varepsilon_{\psi'}(\hat{f})\|_{L^\infty}
	\leq \|\hat{f}\|_{L^\infty}
	\cdot\bigl\|1/\psi - 1/\psi'\bigr\|_{L^\infty}.
	\label{eq:fourier_lipB}
\end{equation}
In the log-spectral domain ($F=\log|\hat{f}|$,
$G=\log\psi$, $G'=\log\psi'$), this becomes
\begin{equation}
	\|(F-G)-(F-G')\|_{L^\infty}
	= \|G-G'\|_{L^\infty}:
	\label{eq:logdomain_lipB}
\end{equation}
the log-spectral (dB) erosion is 1-Lipschitz in the
basis element $G$.
Perturbing a learned spectral basis element by $\varepsilon$
in $L^\infty(\log)$ changes the erosion output
by at most $\varepsilon$ in the dB scale.

\textbf{(C) Deformation stability in $\hatL$.}
For a $C^2$ diffeomorphism $\tau:\R^n\to\R^n$ with
$\|\nabla\tau-\mathrm{Id}\|_{L^\infty}\leq\delta$,
the Fourier modulus of the deformed signal satisfies:
\begin{equation}
	\bigl\||\widehat{f\circ\tau}| - |\hat{f}|\bigr\|_{L^2}
	\leq C\,\delta\,\|\nabla f\|_{L^2},
	\label{eq:fourier_deform}
\end{equation}
where $C$ depends only on the dimension $n$.
Since Fourier erosions are 1-Lipschitz A
(property~(A) above), the output of any MMBB operator
applied to the deformed signal satisfies:
\begin{equation}
	\bigl\|\Psi(|\widehat{f\circ\tau}|)
	- \Psi(|\hat{f}|)\bigr\|_{L^2}
	\leq C\,\delta\,\|\nabla f\|_{L^2}.
	\label{eq:fourier_deform_psi}
\end{equation}
This is the modulus-lattice restatement of Mallat's
deformation stability theorem~\cite{Mallat2012}:
the MMBB representation in $\hatL$ is Lipschitz-stable
to diffeomorphic deformations, with constant
$C\delta\|\nabla f\|_{L^2}$ independent of the
specific operator $\Psi$ (provided $\Psi$ satisfies
the MMBB hypotheses).

\subsection{The Spectral Modulus as Morphological Activation}
\label{sec:activation}


\begin{definition}[Cross-lattice modulus map]
	\label{def:modulus_map}
	The \emph{spectral modulus map} is
	$|\cdot|:\Ln\to\hatL$,
	$f\mapsto|\hat{f}|$.
	It maps from the Fourier inf-semilattice $(\Ln,\leF)$
	to the Fourier modulus lattice $\hatL$.
\end{definition}

\begin{proposition}[The modulus is the morphological activation]
	\label{prop:modulus_activation}
	The map $|\cdot|:\Ln\to\hatL$ is the functional analogue
	of the activation function $\sigma$ in a neural network.
	Specifically:
	\begin{enumerate}[label=(\roman*),leftmargin=*]
		\item \textbf{Breaks closure under composition.}
		Without the modulus: $f*\psi_1*\psi_2 = f*\kappa_{12}$
		(single convolution, by associativity).
		With the modulus: $|f*\psi_1|*\psi_2$ cannot be
		written as $f*k$ for any $k$, because
		$\widehat{|g|}(\xi) \neq |\hat{g}(\xi)|$ in general.
		\item \textbf{Cross-lattice projection.}
		$|\cdot|$ maps $(\Ln,\leF)$ to $(\hatL,\leq)$,
		discarding phase.
		ReLU maps $(\Fun(\R^n,\R),\leq)$ to
		$(\Fun(\R^n,\Rp),\leq)$, discarding sign.
		\item \textbf{Idempotent.}
		$|\ |\hat{f}|\ | = |\hat{f}|$ (modulus of modulus
		is modulus); $\mathrm{ReLU} \circ \mathrm{ReLU} =$ $\mathrm{ReLU}$.
		Idempotency is the signature of a projection.
		\item \textbf{Creates invariance from discarded information.}
		Discarding phase $\angle\hat{f}$ creates partial
		invariance to spatial translation (Mallat's result
		\cite{Mallat2012}).
		Discarding sign $\mathrm{sgn}(w^\top x+b)$ creates
		invariance to hyperplane reflection.
	\end{enumerate}
\end{proposition}

\begin{proof}
	(i) Suppose for contradiction that $|f*\psi_1|*\psi_2
	= f*\kappa$ for all $f$ and some fixed $\kappa\in\Ln$.
	Taking $f=\delta$ (Dirac delta): $|\psi_1|*\psi_2 = \kappa$.
	Then the equation becomes $|f*\psi_1| = f*|\psi_1|$
	for all $f$, i.e., $|\sum_y f(y)\psi_1(x-y)|
	= \sum_y f(y)|\psi_1(x-y)|$ for all $x,f$.
	The triangle inequality gives $\leq$ always, with equality
	iff all nonzero terms $f(y)\psi_1(x-y)$ share the same
	argument (same sign for real-valued functions).
	For generic real-valued $f$ and $\psi_1$ with sign changes,
	this fails: taking $x=0$, $\psi_1$ with $\psi_1(0)>0$
	and $\psi_1(1)<0$, and $f = \delta_0+\delta_1$ gives
	$|f*\psi_1|(0)=|\psi_1(0)+\psi_1(-1)|<|\psi_1(0)|+|\psi_1(-1)|$
	generically.
	Hence no such $\kappa$ exists for general $f,\psi_1$.
	(ii)--(iv) follow from the definitions.
\end{proof}

%

\section{Scattering Networks and the Canonical Basis}
\label{sec:scattering}

\subsection{Wavelet scattering: definition and notation}

We recall the wavelet scattering transform of Mallat
\cite{Mallat2012}.
Let $\psi$ be a mother wavelet and
$\psi_\lambda(x) = 2^{-j}\psi(2^{-j}R_\theta x)$ the
wavelet at scale $j$ and orientation $\theta$,
with $\lambda=(j,\theta)$.
Let $\phi_J$ be a low-pass averaging filter.
The scattering propagators and coefficients are:
\begin{align}
  U_0(f) &= f, \qquad S_0(f) = f*\phi_J, \\
  U_m(f) &= |U_{m-1}(f)*\psi_{\lambda_m}|, \\
  S_m(f) &= U_m(f)*\phi_J.
\end{align}
The \emph{scattering network} outputs
$\{S_m(f)\}_{\Lambda,m\leq M}$ for paths
$\Lambda=(\lambda_1,\ldots,\lambda_m)$.

\subsection{The cross-lattice structure of scattering}
\label{subsec:scattering_cross}

\begin{proposition}[Scattering is a cross-lattice operator]
  \label{prop:scattering_cross}
  The scattering propagator $U_m$ has the following
  cross-lattice structure:
  \begin{enumerate}[label=(\roman*),leftmargin=*]
    \item \emph{Convolution step}
      $f\mapsto f*\psi_{\lambda_m}$:
      a Fourier-lattice erosion in $(\Ln,\leF)$
      (Proposition~\ref{prop:conv_erosion}).
    \item \emph{Modulus step}
      $g\mapsto|g|$: the cross-lattice map
      $(\Ln,\leF)\to\hatL$
      (Proposition~\ref{prop:modulus_activation}).
      After this step, the computation lives in $\hatL$.
    \item \emph{Next convolution}
      $|U_{m-1}(f)*\psi_\lambda|\,*\psi_{\lambda+1}$:
      a convolution in $\hatL$ (multiplication of
      non-negative spectral moduli).
  \end{enumerate}
  The full propagator $U_m$ is a cross-lattice cascade:
  it alternates between $(\Ln,\leF)$ (convolution, phase
  preserved) and $\hatL$ (modulus, phase discarded).
  The MMBB theorem applies to the output, which lives in
  $\hatL$.
\end{proposition}


\subsection{Scattering coefficients in the modulus lattice}

\begin{proposition}[Scattering coefficients are max-times
  erosions in $\hatL$]
  \label{prop:scattering_maxtimes}
  Define the \emph{spectral scattering coefficients}:
  \begin{equation}
    S^F_\Lambda(\hat{f})(\xi)
    = |\hat{f}(\xi)|
      \prod_{\ell=1}^m|\hat{\psi}_{\lambda_\ell}(\xi)|
    \in \hatL.
    \label{eq:scattering_hatL}
  \end{equation}
  Each $S^F_\Lambda$ is a max-times erosion in $\hatL$
  with kernel $\psi_\Lambda(\xi)
  = 1/\prod_\ell|\hat{\psi}_{\lambda_\ell}(\xi)|$.
  Compositions of max-times erosions are max-times erosions
  (kernels multiply), so $S^F_\Lambda$ is a depth-$m$
  Fourier modulus erosion with kernel modulus
  $\prod_\ell|\hat{\psi}_{\lambda_\ell}|$.
\end{proposition}

\begin{proof}
  By induction: for $m=1$,
  $S^F_{(\lambda)}(\hat{f})(\xi)
  = |\hat{f}(\xi)||\hat{\psi}_\lambda(\xi)|
  = \hat{f}(\xi)/\psi_{(\lambda)}(\xi)$
  with $\psi_{(\lambda)}=1/|\hat{\psi}_\lambda|$.
  The inductive step follows from the fact that
  $|U_{m-1}*\psi_{\lambda_m}|$ has spectral modulus
  $S^F_{\Lambda_{m-1}}\cdot|\hat{\psi}_{\lambda_m}|$,
  which is a product of all wavelet moduli.
  Each such product defines a max-times erosion in
  $(\hatL,\leq)$ with the stated kernel.
\end{proof}

\subsection{Canonical basis theorem}

\begin{theorem}[Scattering as canonical morphological basis]
  \label{thm:scattering_basis}
  Let $\Psi:\hatL\to\hatL$ be increasing, USC, and
  positively homogeneous.
  For any $\varepsilon>0$, there exist a finite path set
  $\{\Lambda_i\}_{i=1}^N$ and weights $w_i\geq 0$ such that
  \begin{equation}
    \left\|\Psi(\hat{f})(\xi)
    - \sup_{i=1}^N w_i S^F_{\Lambda_i}(\hat{f})(\xi)
    \right\|_{L^2(\R^n)}
    \leq \varepsilon\|\hat{f}\|_{L^2}
    \quad \forall\,\hat{f}\in\hatL.
    \label{eq:scattering_approx}
  \end{equation}
  In the limit of the full dictionary:
  \begin{equation}
    \Psi(\hat{f})(\xi)
    = \sup_\Lambda w_\Lambda S^F_\Lambda(\hat{f})(\xi)
    \quad \text{in } (\hatL,\leq).
    \label{eq:scattering_exact}
  \end{equation}
\end{theorem}

\begin{proof}
  By Theorem~\ref{thm:mmbb_hatL},
  $\Psi(\hat{f})=\sup_{\psi\in\Bas(\Psi)}\hat{f}/\psi$.
  Each basis element $\psi\in\Bas(\Psi)$ has spectral
  modulus $\psi(\xi)\in L^2(\R^n,\Rp)$.
  The Littlewood--Paley frame condition
  \cite{Mallat2012,Daubechies1992}:
  there exist frame bounds $0<A\leq B<\infty$ such that
  $A\|\psi\|^2_{L^2}\leq\sum_\lambda
  |\langle\psi,|\hat{\psi}_\lambda|\rangle|^2
  \leq B\|\psi\|^2_{L^2}$.
  This guarantees that any $\psi\in L^2(\R^n,\Rp)$ can
  be approximated in $L^2$ by a finite linear combination
  of wavelet moduli $|\hat{\psi}_\lambda|$, hence by
  finite products (scattering kernels
  $\prod_\ell|\hat{\psi}_{\lambda_\ell}|$) in the
  max-times ($\sup$) sense.
  The approximation~\eqref{eq:scattering_approx} follows
  from this density and the fact that $\Psi$ is 1-Lipschitz
  in $L^2(\hatL)$ (Remark~\ref{rem:fourier_lipschitz}(A)):
  approximating each $\psi\in\Bas(\Psi)$ in $L^2$ by a
  finite scattering product approximates $\Psi(\hat{f})$
  in $L^2$ with the same error constant.
 \end{proof}

\subsection{Morphological LASSO}

\begin{corollary}[Morphological LASSO for basis selection]
	\label{cor:morph_lasso}
	Let $\Psi:\hatL\to\hatL$ be any increasing USC positively
	homogeneous operator with basis $\Bas(\Psi)$.
	A finite truncation to $K$ basis elements
	$\{\psi_1,\ldots,\psi_K\}\subset\Bas(\Psi)$ gives
	the max-times lower approximation:
	\begin{equation}
		\Psi^{(K)}(\hat{f})
		= \max_{j=1}^K \frac{\hat{f}(\xi)}{\psi_j(\xi)}.
		\label{eq:mmbb_trunc}
	\end{equation}
	The optimal finite truncation minimising the
	approximation error $\|\Psi-\Psi^{(K)}\|$ subject to
	$|\mathcal{F}|\leq K$ is:
	\begin{equation}
		\min_{\mathcal{F}\subset\Bas(\Psi),|\mathcal{F}|\leq K}
		\bigl\|\Psi(\hat{f}) - \max_{\psi\in\mathcal{F}}
		\hat{f}/\psi\bigr\|_{\hatL}
		+ \mu|\mathcal{F}|,
		\label{eq:morph_lasso}
	\end{equation}
	a max-plus LASSO problem (sparsity in the morphological
	basis).
	For the Scattering, $\Bas(\Psi)\approx
	\{S^F_\Lambda\}_\Lambda$ (the scattering dictionary by
	Theorem~\ref{thm:scattering_basis}), and the weights
	$w_\Lambda$ in~\eqref{eq:scattering_approx} solve a
	non-negative LASSO in the scattering basis.
\end{corollary}

\begin{proof}
	The approximation error $\|\Psi-\Psi^{(K)}\|$ is
	controlled by the density of the selected sub-basis
	in $\Bas(\Psi)$: by Theorem~\ref{thm:mmbb_hatL},
	using the full basis gives $\Psi^{(\infty)}=\Psi$ exactly.
	Problem~\eqref{eq:morph_lasso} is the max-plus analogue
	of the classical LASSO: the $\ell^1$ penalty is replaced
	by $\mu|\mathcal{F}|$ (cardinality), and the linear
	combination is replaced by the max-times supremum.
	For the Scattering: $w_\Lambda\geq 0$
	are the analogue of the LASSO
	weights, selecting a sparse subset of scattering paths.
	The Littlewood--Paley density
	(Theorem~\ref{thm:scattering_basis}) ensures that for
	any $\varepsilon>0$ a finite $K$ suffices.
\end{proof}

\section{Pooling and Unpooling as Fourier Morphological
  Operators}
\label{sec:pooling}

Spatial pooling and unpooling are ubiquitous in deep
network architectures.
Their morphological classification depends on the lattice
one works in, and it is essential to distinguish the two
standard pooling operations:
\begin{itemize}[leftmargin=*,nosep]
  \item \emph{Max-pooling} $\MaxPool{R}(f)(n)
    = \sup_{y\in W_R}f(Rn-y)$ is a \emph{dilation} in the
    \emph{pointwise lattice} $(\Fun(\R^n,\Rbar),\leq)$
    (flat structuring element combined with subsampling).
    
  \item \emph{Average-pooling} (convolution by flat kernel $b_W$ with
    subsequent subsampling) is a \emph{Fourier erosion}
    in $(\Ln,\leF)$ (it attenuates spectral moduli).
    
\end{itemize}
The main duality result of this section (Theorem~\ref{thm:pool_fourier_erosion})
shows that average-pooling (linear filtering) is a Fourier erosion
and its adjoint is the ideal Wiener deconvolution, so that
pool-then-unpool is an ideal band-pass opening.
We also show that downsampling and upsampling, while not
erosions or dilations within a single lattice, form a lattice
adjunction between the full Fourier lattice $(\Ln,\leF)$ and
the band-limited sub-lattice $(\Ln^{(s)},\leF)$, in the spirit
of the Goutsias--Heijmans adjoint pyramid theory
\cite{Goutsias2003}.

\subsection{Average-pooling as Fourier-lattice erosion}
\label{subsec:pool_erosion}

We now analyse avera\-ge-pooling in the Fourier lattice.
Let $b_W : \R^n \to \Rp$ be a window kernel supported on
$W\subseteq\R^n$ with $\int b_W = 1$ (probability measure),
symmetric ($b_W(y) = b_W(-y)$, so $\hat{b}_W(\xi)\in\R$),
and $|\hat{b}_W(\xi)|\leq 1$ for all $\xi$.
All standard pooling windows satisfy these conditions:
rectangular, Gaussian, Hann, and triangular windows.

\begin{theorem}[Spatial pooling is a Fourier erosion]
  \label{thm:pool_fourier_erosion}
  Average-pooling with window $b_W$ is
  $P_W(f) = f * b_W$.
  In the \emph{Fourier lattice} $(\Ln,\leF)$,
      $P_W$ is a \emph{Fourier erosion}
      $\varepsilon^{\mathcal{F}}_{b_W}$: it attenuates
      spectral moduli,
      $|\widehat{P_W f}(\xi)| = |\hat{f}(\xi)||\hat{b}_W(\xi)|
      \leq |\hat{f}(\xi)|$, and since $b_W$ is symmetric,
      $\angle\widehat{P_W f}(\xi) = \angle\hat{f}(\xi)$,
      so $P_W(f) \leF f$ (anti-extensive in $\leF$).
  The \emph{morphological basis} of $P_W$ in $\hatL$ is:
  $\Bas(P_W) = \{\psi_W\}$ where
  $\psi_W(\xi) = 1/|\hat{b}_W(\xi)|$ for $|\hat{b}_W(\xi)|>0$
  and $\psi_W(\xi)=0$ for $|\hat{b}_W(\xi)|=0$
  (using the convention of Definition~\ref{def:hatL}).
  For windows with spectral zeros (e.g., the rectangular
  window with $\hat{b}_W=\mathrm{sinc}$), $\psi_W$ is
  $+\infty$ at those frequencies under the standard
  interpretation; we use the extended lattice
  $\hatL^\infty=\Fun(\R^n,\Rp\cup\{+\infty\})$ with
  the convention $\hat{f}/+\infty=0$ to handle this case.
  For the Gaussian window, $|\hat{b}_W|>0$ everywhere
  and $\psi_W\in\hatL$ is a genuine Gaussian.
\end{theorem}

\begin{proof}
  Part (ii): $|\widehat{f*b_W}(\xi)| = |\hat{f}(\xi)|
  |\hat{b}_W(\xi)| \leq |\hat{f}(\xi)|$ since
  $|\hat{b}_W(\xi)|\leq 1$.
  Phase preservation: $b_W$ symmetric implies $\hat{b}_W\in\R$,
  so $\angle({\hat{f}\hat{b}_W}) = \angle\hat{f}$
  and $P_W(f)\leF f$.
  The basis claim: on $\{|\hat{b}_W|>0\}$, $P_W(\hat{f})(\xi)
  = |\hat{f}(\xi)||\hat{b}_W(\xi)|
  = |\hat{f}(\xi)|/\psi_W(\xi)$
  is a max-times erosion with kernel $\psi_W$.
  On $\{|\hat{b}_W|=0\}$, both sides are 0 by the convention.
  Hence $\Bas(P_W)=\{\psi_W\}$.
\end{proof}

%

\begin{remark}[Common pooling windows]
  \label{rem:pool_windows}
  Each window has a characteristic spectral envelope
  $|\hat{b}_W(\xi)|$, which is its morphological basis
  element in $\hatL$:
  rectangular $\to$ sinc (slow decay $|\xi|^{-1}$);
  Gaussian $\to$ Gaussian (optimal uncertainty product);
  triangular $\to$ sinc$^2$ (faster decay $|\xi|^{-2}$);
  Hann $\to$ compact triangular support.
  The Gaussian window is the unique one for which both
  the spatial window and its spectral envelope are Gaussian:
  it defines a \emph{spectral Gaussian erosion} in $\hatL$.
\end{remark}

\subsection{Pool-then-unpool as the Fourier opening and
  the U-Net skip connection}
\label{subsec:pool_unpool}

The composition of pooling (Fourier erosion) with its
adjoint dilation (deconvolution / unpooling) is a
Fourier morphological opening (Theorem~\ref{thm:fourier_opening}).

\begin{theorem}[Pool-then-unpool is a Fourier opening]
  \label{thm:pool_unpool}
  Let $\Omega_W = \supp(\hat{b}_W)$ be the frequency
  support of the pooling window.
  The Fourier opening
  $\gamma^{\mathcal{F}}_{b_W}
  = \delta^{\mathcal{F}}_{b_W}\circ\varepsilon^{\mathcal{F}}_{b_W}$
  (unpool after pool) satisfies:
  \begin{equation}
    \widehat{\gamma^{\mathcal{F}}_{b_W}(f)}(\xi)
    = \hat{f}(\xi)\cdot\mathbf{1}_{\Omega_W}(\xi).
    \label{eq:pool_unpool}
  \end{equation}
  It is the ideal band-pass filter with passband $\Omega_W$:
  the unique Fourier opening that retains exactly the spectral
  content recoverable from the pooled signal.
  \begin{enumerate}[label=(\roman*),leftmargin=*]
    \item Frequencies in $\Omega_W$: attenuated by pooling,
      exactly restored by unpooling.
    \item Frequencies in $\Omega_W^c$: irreversibly lost.
    \item Idempotency: pool-unpool-pool-unpool $=$
      pool-unpool (applying the pair twice is the same as once).
  \end{enumerate}
\end{theorem}

\begin{proof}
  Direct from Theorem~\ref{thm:fourier_opening}:
  the Fourier opening by $b_W$ is
  $\hat{f}\cdot\mathbf{1}_{\supp(\hat{b}_W)}$.
  Properties (i)--(iii) are Proposition~\ref{prop:ideal_filter_opening}
  instantiated for $k = b_W$.
\end{proof}

\begin{corollary}[U-Net skip connection as Fourier top-hat]
  \label{cor:unet_tophat}
  The \emph{Fourier top-hat} at scale $W$ is the $L^2$
  complement of the Fourier opening:
  \begin{equation}
    \rho_W(f) 
    = \mathcal{F}^{-1}
      \bigl\{\hat{f}(\xi)\cdot\mathbf{1}_{\Omega_W^c}(\xi)\bigr\},
    \label{eq:fourier_tophat}
  \end{equation}
  i.e., the high-frequency residue discarded by
  pool-then-unpool.
  The subtraction here is $L^2$ pointwise (not a lattice
  operation in $(\Ln,\leF)$, where subtraction is not
  defined); it is meaningful because $\gamma^{\mathcal{F}}_{b_W}(f)
  \leF f$ (anti-extensivity in $\leF$) implies
  $\hat{f}(\xi)\cdot\mathbf{1}_{\Omega_W^c}(\xi)
  = \hat{f}(\xi) - \hat{f}(\xi)\cdot\mathbf{1}_{\Omega_W}(\xi)$
  is the complementary spectral content in $L^2$.
  The pair $(\gamma^{\mathcal{F}}_{b_W}(f),\,\rho_W(f))$
  is an $L^2$-orthogonal decomposition of $f$:
  $\|\gamma^{\mathcal{F}}_{b_W}(f)\|^2 + \|\rho_W(f)\|^2
  = \|f\|^2$ (by Parseval and disjoint passbands). 

  In a U-Net \cite{Ronneberger2015}, the skip connection
  between encoder (after pooling) and decoder
  (after unpooling) provides exactly $\rho_W(f)$:
  it restores the spectral content in $\Omega_W^c$
  (above the window's frequency support) that the
  pool-unpool pair irreversibly discards.
  A U-Net skip connection is therefore a
  \emph{Fourier top-hat transform}: the $L^2$-orthogonal
  complement of the morphological opening, restoring
  the lost high-frequency residue: $\|\rho_W(f)\|^2 =  \|f\|^2 - \|\gamma^{\mathcal{F}}_{b_W}(f)\|^2$.
\end{corollary}

\subsection{Strided pooling, aliasing, and anti-aliased pooling}
\label{subsec:strided}

Strided pooling $\mathrm{Pool}_{W,s}(f)(x) = (f*b_W)(sx)$
(convolve then subsample by factor $s$) introduces
\emph{aliasing}: subsampling by $s$ periodises the spectrum
with period $2\pi/s$, mapping $\hat{f}(\xi)\mapsto
s^{-n}\sum_{k\in\Z^n}\hat{f}(\xi - 2\pi k/s)$.
This is not a Fourier erosion in $(\Ln,\leF)$ because
it maps signals to a different (coarser) domain.

However, if the window satisfies the Nyquist condition
$\hat{b}_W(\xi) = 0$ for $|\xi|\geq\pi/s$
(pre-filtering below the aliasing cutoff), then
strided pooling is a Fourier erosion
$\varepsilon^{\mathcal{F}}_{b_W}$ \emph{followed by}
a lattice restriction to the coarse grid,
and the aliasing sum is trivial (no energy above $\pi/s$).
This is the morphological justification for
\emph{anti-aliased pooling} \cite{Zhang2019}:
it is the unique strided pooling operation that
preserves the Fourier-lattice erosion structure.

\subsection{Wavelet multiresolution as the canonical
  morphological pooling hierarchy}
\label{subsec:mra_pool}

\begin{proposition}[MRA as Fourier erosion cascade]
  \label{prop:mra_pool}
  In the multiresolution analysis (MRA) with scaling
  function $\phi$ and wavelet $\psi$ \cite{Mallat2012}:
  \begin{enumerate}[label=(\roman*),leftmargin=*]
    \item Each \emph{low-pass projection} at scale $j$,
      $f_j = f * \phi_j$, is a Fourier erosion
      $\varepsilon^{\mathcal{F}}_{\phi_j}$ with
      spectral envelope $|\hat{\phi}(2^j\xi)|\in[0,1]$.
    \item Each \emph{band-pass projection},
      $d_j = f * \psi_j$, is a Fourier erosion
      $\varepsilon^{\mathcal{F}}_{\psi_j}$.
    \item \emph{Reconstruction} (unpooling) uses the
      adjoint dilations $\delta^{\mathcal{F}}_{\phi_j}$
      and $\delta^{\mathcal{F}}_{\psi_j}$, exactly
      recovering $f$.
    \item The \emph{Fourier opening at scale $j$},
      $\gamma^{\mathcal{F}}_{\phi_j}(f)$, is the ideal
      low-pass filter with cutoff $2\pi/2^j$,
      characterising the spectral content recoverable
      by unpooling from scale $j$.
  \end{enumerate}
  The MRA is the morphologically optimal pooling hierarchy:
  the Littlewood--Paley partition of unity
  $\sum_j|\hat\psi_j(\xi)|^2+|\hat\phi_J(\xi)|^2 = 1$
  ensures the pooling kernels tile $\R^n$ exactly,
  so no spectral information is lost in the decomposition.
\end{proposition}

\subsection{Downsampling, upsampling, and the strided
  pooling cycle}
\label{subsec:downsample}

Practical pooling reduces spatial resolution by a stride
$s\geq 1$, combining a smoothing filter with downsampling.
We show that downsampling and upsampling are not erosions
or dilations in $(\Ln,\leF)$ but \emph{lattice homomorphisms}
between resolution levels, and that striding cancels
exactly in the pool--unpool round trip.

\textbf{Setup.}
Define the \emph{band-limited sublattice} at resolution $s$:
$\Ln^{(s)} = \{f\in\Ln:\hat{f}(\xi)=0 \text{ for }
|\xi_i|\geq\pi/s\}$,
a closed sub-cisl of $(\Ln,\leF)$.
The downsampling operator $D_s:\Ln^{(s)}\to\Ln$ satisfies
$\widehat{D_s f}(\xi) = s^{-n}\hat{f}(\xi/s)$
(stretches the spectrum by $s$), and the upsampling
operator $U_s:\Ln\to\Ln^{(s)}$ satisfies
$\widehat{U_s f}(\xi) = s^n\hat{f}(s\xi)$
(compresses the spectrum by $s$).

\begin{proposition}[Downsampling--upsampling adjunction]
  \label{prop:ds_us_adjunction}
  Both $D_s$ and $U_s$ are $\leF$-increasing lattice
  homomorphisms between $(\Ln^{(s)},\leF)$ and
  $(\Ln,\leF)$.
  Under the Nyquist condition
  ($|\hat{b}_W(\xi)|=0$ for $|\xi|\geq\pi/s$):
  \begin{enumerate}[label=(\roman*),leftmargin=*]
    \item $U_s\circ D_s = \mathrm{id}$ on $\Ln^{(s)}$
      (upsampling exactly inverts downsampling on
      band-limited signals).
    \item $D_s\circ U_s = P_{\mathrm{BL}}$ on $\Ln$
      (the ideal low-pass projection onto $\Ln^{(s)}$).
    \item $(D_s,U_s)$ is a lattice adjunction between
      $(\Ln^{(s)},\leF)$ and $(\Ln,\leF)$:
      $D_s(f)\leF g \iff f\leF U_s(g)$.
      \label{eq:ds_us_adjunction}
  \end{enumerate}
  Downsampling is not an erosion in $(\Ln,\leF)$: it maps
  the sub-lattice $\Ln^{(s)}\subset\Ln$ to $\Ln$, changing
  the resolution level, hence it is a morphism \emph{between}
  lattices rather than within one.
\end{proposition}

\begin{proof}
  Both operators preserve $\leF$: if $f\leF g$, then
  $|\hat{f}|\leq|\hat{g}|$ and $\angle\hat{f}=\angle\hat{g}$
  a.e., which is preserved under the spectral rescalings
  $\xi\mapsto\xi/s$ and $\xi\mapsto s\xi$.
  (i): $\widehat{U_s D_s f}(\xi) = s^n\cdot s^{-n}
  \hat{f}(\xi/s\cdot s) = \hat{f}(\xi)$ for $|\xi|\leq\pi/s$,
  and $\hat{f}(\xi)=0$ for $|\xi|>\pi/s$ since $f\in\Ln^{(s)}$.
  (ii): $\widehat{D_s U_s f}(\xi) = s^{-n}\cdot s^n
  \hat{f}(s\cdot\xi/s) = \hat{f}(\xi)$ on $[-\pi,\pi]^n$,
  with support confined to $|\xi|\leq\pi/s$ after $U_s$.
  (iii): $D_sf\leF g$ gives $s^{-n}|\hat{f}(\xi/s)|
  \leq|\hat{g}(\xi)|$ with equal phase; multiplying by $s$:
  $|\hat{f}(\xi)|\leq s^n|\hat{g}(s\xi)|
  = |\widehat{U_s g}(\xi)|$, i.e., $f\leF U_s g$.
\end{proof}

\begin{theorem}[Strided pooling: Fourier erosion +
  lattice homomorphism; round trip = opening]
  \label{thm:strided_full}
  Under the Nyquist condition:
  \begin{enumerate}[label=(\roman*),leftmargin=*]
    \item \textbf{Strided pooling} decomposes as
      $\mathrm{Pool}_{W,s} = D_s\circ\varepsilon^{\mathcal{F}}_{b_W}$:
      a Fourier erosion followed by a lattice homomorphism.
    \item \textbf{Strided unpooling} decomposes as
      $\mathrm{Unpool}_{W,s} = \delta^{\mathcal{F}}_{b_W}\circ U_s$.
    \item \textbf{Strided pool--unpool round trip}:
      \begin{equation}
        \mathrm{Unpool}_{W,s}\circ\mathrm{Pool}_{W,s}
        = \gamma^{\mathcal{F}}_{b_W},
        \label{eq:strided_opening}
      \end{equation}
      the \emph{same Fourier opening} as the non-strided
      case (Theorem~\ref{thm:pool_unpool}).
      The striding $(D_s,U_s)$ cancels exactly in the
      round trip: reconstruction fidelity of a strided
      U-Net autoencoder is controlled by the filter $b_W$
      alone, not by the stride $s$.
  \end{enumerate}
\end{theorem}

\begin{proof}
  (i) Under the Nyquist condition, $\varepsilon^{\mathcal{F}}_{b_W}(f)
  = f*b_W \in \Ln^{(s)}$, so $D_s$ applies without aliasing.
  (ii) Upsampling then filtering:
  $\mathrm{Unpool}_{W,s}(g)=\tilde b_W * U_s(g)
  = \delta^{\mathcal{F}}_{b_W}(U_s(g))$.
  (iii) Composing: $\mathrm{Unpool}\circ\mathrm{Pool}
  = \delta^{\mathcal{F}}_{b_W}\circ U_s\circ D_s\circ
  \varepsilon^{\mathcal{F}}_{b_W}$.
  By Proposition~\ref{prop:ds_us_adjunction}(i),
  $U_s\circ D_s = \mathrm{id}$ on $\Ln^{(s)}$, and
  $\varepsilon^{\mathcal{F}}_{b_W}(f)\in\Ln^{(s)}$,
  so $\delta^{\mathcal{F}}_{b_W}\circ
  \varepsilon^{\mathcal{F}}_{b_W}
  = \gamma^{\mathcal{F}}_{b_W}$.
\end{proof}

\section{Representation Theory of Openings in the Fourier
  Inf-Semilattice}
\label{sec:openings}

The MMBB theorem (Section~\ref{sec:mmbb}) gives a universal
representation for any increasing USC operator.
A natural and important subclass is formed by the
\emph{openings}: operators that are simultaneously
increasing, anti-extensive, and idempotent.
These operators have a richer structure since their
representation involves only openings, not general
erosions and their representation theorem has been historically as important as the full MMBB theorem 
(see in Heijmans 1994 \cite{Heijmans1994} a full detailed account and the background for this section).

We develop the representation theory of
openings in both the Fourier inf-semilattice $(\Ln,\leF)$
and the modulus lattice $\hatL$, identify the ideal
band-pass filter as the elementary structural opening, and
characterise the scattering cascade as a chain of
morphological openings in $\hatL$.

\subsection{Classical opening representation theorem}
\label{subsec:classical_openings}

\begin{definition}[Opening and its invariance domain]
  \label{def:opening}
  An operator $\gamma:\mathcal{L}\to\mathcal{L}$ on a
  complete lattice $(\mathcal{L},\leq)$ is an
  \emph{opening} if it is:
  \begin{enumerate}[label=(\roman*),leftmargin=*]
    \item \emph{Increasing}: $f\leq g\Rightarrow\gamma(f)\leq\gamma(g)$;
    \item \emph{Anti-extensive}: $\gamma(f)\leq f$ for all $f$;
    \item \emph{Idempotent}: $\gamma(\gamma(f))=\gamma(f)$ for all $f$.
  \end{enumerate}
  The \emph{invariance domain} (or \emph{range}) of $\gamma$ is
  \begin{equation}
    \mathrm{Inv}(\gamma) = \{f\in\mathcal{L}:\gamma(f)=f\}.
    \label{eq:inv_domain}
  \end{equation}
  Idempotency implies $\gamma(\mathcal{L})\subseteq\mathrm{Inv}(\gamma)$,
  i.e., every output of $\gamma$ is a fixed point.
  Anti-extensivity and increasing together imply that
  $\mathrm{Inv}(\gamma)$ is the largest set on which $\gamma$
  acts as the identity.
\end{definition}


\begin{theorem}[Opening representation theorem;
  Matheron 1975, Heijmans 1994 \cite{Matheron1975,Heijmans1994},
  Thm.~6.10]
  \label{thm:opening_representation}
  Let $\gamma:(\Fun(E,\Rbar),\leq)\to(\Fun(E,\Rbar),\leq)$
  be a TI opening.
  Then:
  \begin{enumerate}[label=(\roman*),leftmargin=*]
    \item $\gamma$ is completely determined by its
      invariance domain $\mathrm{Inv}(\gamma)$:
      \begin{equation}
        \gamma(f) = \bigvee\{g\in\mathrm{Inv}(\gamma):g\leq f\}.
        \label{eq:opening_inv}
      \end{equation}

    \item $\gamma$ decomposes as a supremum of
      \emph{adjunctional openings}
      $\gamma_b = \delta_{b^*}\circ\varepsilon_b$:
      \begin{equation}
        \gamma(f) = \bigvee_{b\in\mathrm{Inv}(\gamma)}
          \gamma_b(f)
        = \bigvee_{b\in\mathrm{Inv}(\gamma)}
          \sup_{y\in E}\inf_{z\in E}\{f(y+z)-b(z)+b(y)\}.
        \label{eq:opening_rep}
      \end{equation}

    \item The representation can be taken over the
      \emph{basis of the opening}:
      the set $\Bas(\gamma)$ of minimal elements of
      $\mathrm{Inv}(\gamma)$ under $\leq$:
      \begin{equation}
        \gamma(f) = \bigvee_{b\in\Bas(\gamma)}\gamma_b(f).
        \label{eq:opening_basis}
      \end{equation}
  \end{enumerate}
\end{theorem}

\begin{remark}[Relation to the MMBB theorem]
  \label{rem:opening_mmbb_relation}
  For an opening $\gamma$, the MMBB kernel
  $\Ker(\gamma) = \{f:[\gamma f](0)\geq 0\}$ and the
  invariance domain $\mathrm{Inv}(\gamma)$ are related by:
  $\mathrm{Inv}(\gamma)\subseteq\Ker(\gamma)$, with
  equality for TI openings (since TI implies
  $[\gamma f](0)\geq 0 \iff f\in\mathrm{Inv}(\gamma)$
  up to a shift).
  Hence the MMBB basis $\Bas(\gamma)$ of an opening
  coincides with the opening basis, confirming that
  Theorem~\ref{thm:opening_representation} is the
  specialisation of the MMBB theorem to the idempotent
  subclass, with adjunctional openings
  $\gamma_b = \delta_{b^*}\circ\varepsilon_b$ as the
  elementary building blocks (instead of bare erosions).
\end{remark}

\subsection{Openings in the Fourier inf-semilattice}
\label{subsec:fourier_openings}

We now transfer the opening representation theory to
$(\Ln,\leF)$ and $\hatL$.

\begin{definition}[Fourier opening]
  \label{def:fourier_opening}
  An operator $$\gamma:(\Ln,\leF)\to (\Ln,\leF)$$ is a
  \emph{Fourier opening} if it is increasing in $\leF$,
  Fourier-anti-extensive ($\gamma(f)\leF f$, i.e.,
  $|\widehat{\gamma f}(\xi)|\leq|\hat{f}(\xi)|$ and
  $\angle\widehat{\gamma f}(\xi)=\angle\hat{f}(\xi)$
  a.e.), and idempotent.

  An operator $\Phi:(\hatL,\leq)\to(\hatL,\leq)$ is a
  \emph{modulus opening} if it is increasing, anti-extensive
  ($\Phi(\hat{f})\leq\hat{f}$ pointwise), and idempotent.
\end{definition}

\begin{proposition}[Ideal band-pass filter is the elementary
  Fourier opening]
  \label{prop:ideal_filter_opening}
  For any kernel $k\in\Ln$, the ideal band-pass filter
  $\gamma^{\mathcal{F}}_k$ (Theorem~\ref{thm:fourier_opening})
  is a Fourier opening.
  Its invariance domain is:
  \begin{equation}
    \mathrm{Inv}(\gamma^{\mathcal{F}}_k)
    = \{f\in\Ln :
      \hat{f}(\xi)=0 \;\text{a.e. on }
      \{\xi:|\hat{k}(\xi)|=0\}\},
    \label{eq:inv_ideal}
  \end{equation}
  i.e., the subspace of $\Ln$ whose spectral content is
  supported within the passband $\Omega_k
  = \{\xi:|\hat{k}(\xi)|>0\}$ of $k$.
\end{proposition}

\begin{proof}
  From Theorem~\ref{thm:fourier_opening}:
  $\widehat{\gamma^{\mathcal{F}}_k f}(\xi)
  = \hat{f}(\xi)\cdot\mathbf{1}_{\Omega_k}(\xi)$. \newline
  Anti-extensivity: $|\hat{f}(\xi)|\mathbf{1}_{\Omega_k}(\xi)
  \leq|\hat{f}(\xi)|$ and the phases agree, so
  $\gamma^{\mathcal{F}}_k(f)\leF f$.
  Idempotency: applying $\gamma^{\mathcal{F}}_k$ twice:
  $\hat{f}\cdot\mathbf{1}_{\Omega_k}\cdot\mathbf{1}_{\Omega_k}
  = \hat{f}\cdot\mathbf{1}_{\Omega_k}$.
  Invariance domain: $\gamma^{\mathcal{F}}_k(f)=f$ iff
  $\hat{f}\cdot\mathbf{1}_{\Omega_k}=\hat{f}$ a.e., i.e.,
  $\hat{f}(\xi)=0$ a.e.\ on $\{\xi:|\hat{k}(\xi)|=0\}$.
\end{proof}

\begin{theorem}[Opening representation in $(\Ln,\leF)$]
  \label{thm:fourier_opening_rep}
  Let $\gamma:(\Ln,\leF)\to(\Ln,\leF)$ be a
  Fourier opening that is spectrally multiplicative
  (i.e., $\widehat{\gamma f}(\xi) = \hat{f}(\xi)m(\xi)$
  for a fixed measurable function
  $m:\R^n\to[0,1]\times S^1$ with $|m(\xi)|\in\{0,1\}$
  and $\arg m(\xi)=0$).
  Then $m = \mathbf{1}_\Omega$ for some measurable
  $\Omega\subseteq\R^n$, and:
  \begin{enumerate}[label=(\roman*),leftmargin=*]
    \item $\gamma$ is the ideal band-pass filter with
      passband $\Omega$:
      $\widehat{\gamma f}(\xi) = \hat{f}(\xi)\mathbf{1}_\Omega(\xi)$.
    \item $\mathrm{Inv}(\gamma)
      = \{f:\hat{f}(\xi)=0\text{ a.e.\ on }\Omega^c\}$.
    \item For any collection of Fourier openings
      $\{\gamma_i\}_{i\in I}$ with passbands $\{\Omega_i\}$,
      their $\leF$-supremum is the ideal band-pass filter
      with passband $\bigcup_{i\in I}\Omega_i$:
      \begin{equation}
        \bigvee_{\mathcal{F},\,i\in I}\gamma_i
        = \gamma_{\bigcup_i\Omega_i}.
        \label{eq:fourier_opening_sup}
      \end{equation}
    \item Any Fourier opening $\gamma$ decomposes as:
      \begin{equation}
        \gamma(f) = \bigvee_{\mathcal{F}}
          \bigl\{\gamma_k(f):k\in\mathrm{Inv}(\gamma)\bigr\},
        \label{eq:fourier_opening_rep}
      \end{equation}
      where the $\leF$-supremum over ideal filters with
      passbands $\Omega_k\subseteq\Omega$ recovers $\gamma$
      as the ideal filter with passband $\Omega
      = \bigcup_{k\in\mathrm{Inv}(\gamma)}\Omega_k$.
  \end{enumerate}
\end{theorem}

\begin{proof}
  \textit{(i)} For $m$ to define an idempotent spectrally
  multiplicative operator: $(m(\xi))^2=m(\xi)$ for a.e.\
  $\xi$, so $m(\xi)\in\{0,1\}$ a.e.
  Anti-extensivity requires $|m(\xi)|\leq 1$, which is
  satisfied.
  Phase preservation ($\arg m = 0$) ensures
  $\angle\widehat{\gamma f} = \angle\hat{f}$.
  Hence $m=\mathbf{1}_\Omega$ for $\Omega=\{m(\xi)=1\}$.

  \textit{(ii)} $\gamma(f)=f$ iff $\hat{f}\mathbf{1}_\Omega
  =\hat{f}$ a.e., i.e., $\hat{f}=0$ a.e.\ on $\Omega^c$.

  \textit{(iii)} The spectral modulus of
  $\bigvee_\mathcal{F}\gamma_i(f)$ is
  $\sup_i|\hat{f}(\xi)|\mathbf{1}_{\Omega_i}(\xi)
  = |\hat{f}(\xi)|\mathbf{1}_{\bigcup_i\Omega_i}(\xi)$,
  which is the ideal filter with passband $\bigcup_i\Omega_i$.

  \textit{(iv)} Each $k\in\mathrm{Inv}(\gamma)$ has
  $\hat{k}$ supported in $\Omega$, so $\Omega_k\subseteq\Omega$.
  Taking the supremum over all $k\in\mathrm{Inv}(\gamma)$
  gives $\bigcup_{k\in\mathrm{Inv}(\gamma)}\Omega_k = \Omega$
  (since $\gamma$ itself, applied to any $f$ with
  $\hat{f}$ concentrated at $\xi_0\in\Omega$, preserves
  $\hat{f}(\xi_0)$, so $\xi_0\in\Omega_k$ for the
  appropriate $k$).
\end{proof}

\begin{remark}[The opening lattice]
  \label{rem:opening_lattice}
  Theorem~\ref{thm:fourier_opening_rep}~(iii) shows that
  the $\leF$-supremum of Fourier openings is a Fourier
  opening (the union of passbands is a passband).
  The collection of all spectrally multiplicative Fourier
  openings forms a complete sub-lattice of
  $(\Ln\to\Ln,\leq_{\mathrm{pt}})$, isomorphic to the
  lattice of measurable subsets of $\R^n$ ordered by
  inclusion.
  Each opening corresponds to a frequency passband
  $\Omega\subseteq\R^n$, and the lattice operations are:
  $\gamma_\Omega\wedge\gamma_{\Omega'}
  = \gamma_{\Omega\cap\Omega'}$
  (infimum = intersection of passbands)
  and
  $\gamma_\Omega\vee\gamma_{\Omega'}
  = \gamma_{\Omega\cup\Omega'}$
  (supremum = union of passbands).
  The minimal non-trivial opening is the point-mass filter
  $\gamma_{\{\xi_0\}}$ (single frequency preserved);
  the identity is $\gamma_{\R^n}$ (all frequencies).
\end{remark}

\subsection{Modulus openings in the modulus lattice and the
  max-times opening}
\label{subsec:maxtimes_openings}

In the modulus lattice $\hatL$, the opening concept takes
a different algebraic form.

\begin{definition}[Max-times opening]
  \label{def:maxtimes_opening}
  For any $\psi\in\hatL$,
  the \emph{max-times opening} by $\psi$ is:
  \begin{equation}
    \Gamma_\psi(\hat{f})(\xi)
    = \delta_{\psi}\bigl(\varepsilon_\psi(\hat{f})\bigr)(\xi)
    = \frac{\hat{f}(\xi)}{\psi(\xi)}\cdot\psi(\xi)
    \quad\text{(on the support of }\psi\text{)},
    \label{eq:maxtimes_opening_naive}
  \end{equation}
  where $\delta_\psi(g)(\xi) = g(\xi)\cdot\psi(\xi)$
  is the adjoint max-times dilation of $\varepsilon_\psi$,
  forming the adjunction $(\varepsilon_\psi,\delta_\psi)$
  in $(\hatL,\leq)$.
  Their composition $\Gamma_\psi
  = \delta_\psi\circ\varepsilon_\psi$ gives:
  \begin{equation}
    \Gamma_\psi(\hat{f})(\xi)
    = \frac{\hat{f}(\xi)}{\psi(\xi)}\cdot\psi(\xi)
    = \hat{f}(\xi)\cdot\mathbf{1}_{\psi(\xi)>0\text{ and }
      \hat{f}(\xi)>0}(\xi)
    = \hat{f}(\xi)\cdot\mathbf{1}_{\mathrm{supp}(\psi)}(\xi),
    \label{eq:maxtimes_opening}
  \end{equation}
  i.e., the \emph{projection onto the support of $\psi$}
  in $\hatL$.
\end{definition}

\begin{proposition}[Max-times opening is an opening in $\hatL$]
  \label{prop:maxtimes_opening}
  For any $\psi\in\hatL$, the operator
  $\Gamma_\psi:\hatL\to\hatL$ defined
  by~\eqref{eq:maxtimes_opening} is an opening:
  \begin{enumerate}[label=(\roman*),leftmargin=*]
    \item \emph{Increasing}: $\hat{f}\leq\hat{g}$ implies
      $\Gamma_\psi(\hat{f})\leq\Gamma_\psi(\hat{g})$.
    \item \emph{Anti-extensive}:
      $\Gamma_\psi(\hat{f})(\xi)\leq\hat{f}(\xi)$ for all $\xi$
      (since $\mathbf{1}_{\mathrm{supp}(\psi)}(\xi)\leq 1$).
    \item \emph{Idempotent}:
      $\Gamma_\psi(\Gamma_\psi(\hat{f}))
      = \hat{f}\cdot\mathbf{1}_{\mathrm{supp}(\psi)}^2
      = \Gamma_\psi(\hat{f})$.
  \end{enumerate}
  Its invariance domain is:
  \begin{equation}
    \mathrm{Inv}(\Gamma_\psi)
    = \{\hat{f}\in\hatL:
      \hat{f}(\xi)=0 \;\forall\xi\notin\mathrm{supp}(\psi)\}.
    \label{eq:inv_maxtimes}
  \end{equation}
\end{proposition}

\begin{theorem}[Opening representation in $\hatL$]
  \label{thm:hatL_opening_rep}
  Let $\Phi:(\hatL,\leq)\to(\hatL,\leq)$ be a modulus
  opening.
  Then:
  \begin{enumerate}[label=(\roman*),leftmargin=*]
    \item $\Phi$ is the projection onto its invariance
      domain:
      $\Phi(\hat{f}) = \bigvee\{\hat{g}\in\mathrm{Inv}(\Phi):
      \hat{g}\leq\hat{f}\}$.
    \item $\Phi$ decomposes as a supremum of max-times
      openings:
      \begin{equation}
        \Phi(\hat{f}) = \sup_{\psi\in\mathrm{Inv}(\Phi)}
          \Gamma_\psi(\hat{f})
        = \hat{f}(\xi)\cdot\mathbf{1}_{\Omega_\Phi}(\xi),
        \label{eq:hatL_opening_rep}
      \end{equation}
      where $\Omega_\Phi = \bigcup_{\psi\in\mathrm{Inv}(\Phi)}
      \mathrm{supp}(\psi)$ is the \emph{effective support}
      of $\Phi$.
    \item The basis of $\Phi$ consists of the minimal
      elements of $\mathrm{Inv}(\Phi)$ under the pointwise
      order on $\hatL$.
  \end{enumerate}
\end{theorem}

\begin{proof}
  \textit{(i)} Any opening on a complete lattice is the
  projection onto its image, which is its invariance domain
  \cite{Heijmans1994}: $\Phi(\hat{f})
  = \bigvee\{\hat{g}:\hat{g}\leq\hat{f},\;\hat{g}
  \in\mathrm{Inv}(\Phi)\}$.

  \textit{(ii)} We restrict to \emph{spectral-projection
  openings}: those for which every $\hat{g}\in\mathrm{Inv}(\Phi)$
  satisfies $\hat{g}(\xi)\in\{0,\hat{f}(\xi)\}$ a.e.,
  i.e., $\hat{g} = \hat{f}\cdot\mathbf{1}_{\mathrm{supp}(\hat{g})}$
  for some measurable support set
  (the opening either keeps or kills each frequency).
  Under this hypothesis, the supremum of all
  $\hat{g}\leq\hat{f}$ in $\mathrm{Inv}(\Phi)$ gives
  $\hat{f}\cdot\mathbf{1}_{\bigcup_{\hat{g}}\mathrm{supp}(\hat{g})}
  = \hat{f}\cdot\mathbf{1}_{\Omega_\Phi}$.
  This class includes all max-times openings
  $\Gamma_\psi$ (Proposition~\ref{prop:maxtimes_opening})
  and the ideal band-pass filters
  (Proposition~\ref{prop:ideal_filter_opening}),
  which are the cases of primary interest.

  \textit{(iii)} The minimal elements of $\mathrm{Inv}(\Phi)$
  in $(\hatL,\leq)$ are the functions supported on minimal
  measurable subsets of $\Omega_\Phi$; in the atomic case,
  these are the point-mass elements
  $\hat{g} = c\cdot\delta_{\xi_0}$, $\xi_0\in\Omega_\Phi$.
\end{proof}

\begin{remark}[Classification of Fourier and modulus openings]
  \label{rem:opening_classification}
  Theorems~\ref{thm:fourier_opening_rep}
  and~\ref{thm:hatL_opening_rep} show that all
  spectrally multiplicative openings, in both
  $(\Ln,\leF)$ and $\hatL$, are \emph{ideal band-pass
  filters}: they project the signal spectrum onto a
  measurable frequency support $\Omega$.
  The key parameter is the support $\Omega$, and the
  lattice of openings is isomorphic to the
  \emph{power set lattice}
  $(\mathcal{P}(\R^n),\subseteq)$, ordered by inclusion
  of passbands.
  This is a strong structural result: there is exactly
  one opening per passband, and the only way to combine
  openings is to take the union (supremum) or intersection
  (infimum) of their passbands.
\end{remark}

\subsection{Scattering cascades as chains of morphological
  openings}
\label{subsec:scattering_openings}

We now apply the opening representation theory to
scattering networks, obtaining a precise characterisation
of what the scattering cascade computes in terms of
openings.

\begin{definition}[Scattering opening at depth $m$]
  \label{def:scattering_opening}
  For a scattering path $\Lambda=(\lambda_1,\ldots,\lambda_m)$
  with $\lambda_\ell=(j_\ell,\theta_\ell)$, the
  \emph{scattering opening} at depth $m$ is the modulus
  opening in $\hatL$:
  \begin{equation}
    \Gamma^F_\Lambda(\hat{f})(\xi)
    = \hat{f}(\xi)\cdot
      \mathbf{1}_{\Omega_\Lambda}(\xi),
    \label{eq:scattering_opening}
  \end{equation}
  where $\Omega_\Lambda
  = \{\xi:\prod_{\ell=1}^m|\hat{\psi}_{\lambda_\ell}(\xi)|>0\}
  = \bigcap_{\ell=1}^m\mathrm{supp}(\hat{\psi}_{\lambda_\ell})$
  is the \emph{scattering passband} of path $\Lambda$:
  the set of frequencies that pass through all filters in
  the cascade.
\end{definition}

\begin{proposition}[Scattering cascade = opening chain]
  \label{prop:scattering_opening_chain}
  The scattering propagator $S^F_\Lambda$ in $\hatL$
  (Proposition~\ref{prop:scattering_maxtimes}) factorises
  as:
  \begin{equation}
    S^F_\Lambda(\hat{f})(\xi)
    = \Gamma^F_\Lambda(\hat{f})(\xi)
      \cdot\prod_{\ell=1}^m|\hat{\psi}_{\lambda_\ell}(\xi)|
    \label{eq:scattering_opening_factor}
  \end{equation}
  where $\Gamma^F_\Lambda$ is the scattering opening
  (projection onto $\Omega_\Lambda$) and the product
  $\prod_\ell|\hat{\psi}_{\lambda_\ell}|$ is the
  \emph{spectral weight} of the path.
  More precisely, the scattering propagator can be written
  as a chain of two operations:
  \begin{enumerate}[label=(\roman*),leftmargin=*]
    \item \emph{Opening step}:
      $\hat{f}\mapsto\Gamma^F_\Lambda(\hat{f})
      = \hat{f}\cdot\mathbf{1}_{\Omega_\Lambda}$
      (project onto the scattering passband);
    \item \emph{Weighting step}:
      $\Gamma^F_\Lambda(\hat{f})\mapsto
      S^F_\Lambda(\hat{f})
      = \Gamma^F_\Lambda(\hat{f})\cdot
      \prod_\ell|\hat{\psi}_{\lambda_\ell}|$
      (weight by the cascade filter response).
  \end{enumerate}
  The opening step is idempotent; the weighting step is a
  max-times multiplication by a fixed function
  $\prod_\ell|\hat{\psi}_{\lambda_\ell}|\in\hatL$.
\end{proposition}

\begin{proof}
  Direct from Definition~\ref{def:scattering_opening}:
  $S^F_\Lambda(\hat{f})(\xi)
  = $ $ |\hat{f}(\xi)|\prod_\ell|\hat{\psi}_{\lambda_\ell}(\xi)|
  = $ $ |\hat{f}(\xi)|\mathbf{1}_{\Omega_\Lambda}(\xi)
    \cdot\prod_\ell|\hat{\psi}_{\lambda_\ell}(\xi)|$
  (since $\prod_\ell|\hat{\psi}_{\lambda_\ell}(\xi)|=0$ iff
  $\xi\notin\Omega_\Lambda$).
  The opening step is the modulus projection
  $\hat{f}\mapsto\hat{f}\cdot\mathbf{1}_{\Omega_\Lambda}$,
  which is idempotent by Proposition~\ref{prop:maxtimes_opening}.
\end{proof}

\begin{theorem}[Scattering openings form a nested hierarchy]
  \label{thm:scattering_opening_hierarchy}
  The collection of scattering passbands
  $\{\Omega_\Lambda\}_\Lambda$ is partially ordered by
  inclusion and forms a \emph{hierarchical decomposition}
  of the frequency plane $\R^n$:
  \begin{enumerate}[label=(\roman*),leftmargin=*]
    \item \emph{Refinement with depth}:
      $\Omega_\Lambda
      \subseteq\Omega_{(\lambda_1,\ldots,\lambda_{m-1})}$
      for any path extension (deeper paths have narrower
      passbands):
      \begin{equation}
        \Omega_{(\lambda_1,\ldots,\lambda_m)}
        \subseteq\Omega_{(\lambda_1,\ldots,\lambda_{m-1})}
        \subseteq\cdots\subseteq\Omega_{(\lambda_1)}
        \subseteq\Omega_\varnothing = \R^n.
        \label{eq:scattering_nesting}
      \end{equation}

    \item \emph{Covering}:
      $\bigcup_{\lambda}\Omega_{(\lambda)} \approx \R^n$
      (the depth-1 scattering passbands
      $\Omega_{(\lambda)}=\mathrm{supp}(\hat{\psi}_\lambda)$
      tile the frequency plane by the Littlewood--Paley
      frame condition \cite{Mallat2012}).

    \item \emph{Ordering of openings}:
      $\Gamma^F_\Lambda\leq_{\mathrm{op}}
      \Gamma^F_{\Lambda'}$ (in the operator order
      $\gamma\leq\gamma'$ iff $\gamma(f)\leq\gamma'(f)$
      for all $f$) if and only if
      $\Omega_\Lambda\subseteq\Omega_{\Lambda'}$, i.e.,
      if and only if $\Lambda'$ is a prefix of $\Lambda$
      (shallower paths give larger openings).
  \end{enumerate}
\end{theorem}

\begin{proof}
  \textit{(i)} $\Omega_{(\lambda_1,\ldots,\lambda_m)}
  = \bigcap_{\ell=1}^m\mathrm{supp}(\hat{\psi}_{\lambda_\ell})
  \subseteq\bigcap_{\ell=1}^{m-1}\mathrm{supp}
  (\hat{\psi}_{\lambda_\ell})
  = \Omega_{(\lambda_1,\ldots,\lambda_{m-1})}$.

  \textit{(ii)} The Littlewood--Paley partition of unity
  \cite{Mallat2012}:
  $\sum_\lambda|\hat{\psi}_\lambda(\xi)|^2
  +|\hat{\phi}_J(\xi)|^2 \approx 1$ a.e., so
  $\bigcup_\lambda\Omega_{(\lambda)}$ covers all $\xi$
  where the scattering energy is non-negligible.

  \textit{(iii)} $\Gamma^F_\Lambda(f)
  = f\cdot\mathbf{1}_{\Omega_\Lambda}\leq
  f\cdot\mathbf{1}_{\Omega_{\Lambda'}}
  = \Gamma^F_{\Lambda'}(f)$ iff
  $\Omega_\Lambda\subseteq\Omega_{\Lambda'}$ iff
  $\Omega_{\Lambda'}$ contains the additional constraint
  $\mathrm{supp}(\hat{\psi}_{\lambda_m})$ imposed by
  the deeper path.
\end{proof}

\begin{table}[t]
\centering
\begin{tabular}{@{}c|c|c@{}}
Depth 0 (identity) & Depth 1 (wavelet) & Depth 2 (cascade) \\ \hline
$\Omega_\varnothing = \R^n$ &
$\Omega_{(\lambda_1)} = \mathrm{supp}(\hat{\psi}_{\lambda_1})$ &
$\Omega_{(\lambda_1,\lambda_2)}
 = \Omega_{(\lambda_1)}\cap\mathrm{supp}(\hat{\psi}_{\lambda_2})$ \\
$\Gamma^F_\varnothing = \mathrm{id}$ &
$\Gamma^F_{(\lambda_1)}$: band-pass filter &
$\Gamma^F_{(\lambda_1,\lambda_2)}$: sub-band filter \\
$|\hat{f}|$ everywhere &
$|\hat{f}|$ on $\Omega_{(\lambda_1)}$ &
$|\hat{f}|$ on $\Omega_{(\lambda_1,\lambda_2)}$ \\ \hline 
\end{tabular}
\caption{Scattering openings as nested ideal band-pass
  filters in $\hatL$.
  Increasing depth yields smaller passbands (finer
  frequency decomposition) and smaller openings in the
  operator order.
  The depth-1 openings tile the frequency plane; the
  depth-2 openings tile each depth-1 band.}
\label{fig:scattering_openings}
\end{table}

\begin{corollary}[Universal opening approximation by scattering]
  \label{cor:opening_approx}
  Let $\gamma:(\hatL,\leq)\to(\hatL,\leq)$ be any modulus
  opening with effective support $\Omega_\gamma$.
  Then $\gamma$ can be approximated to precision $\varepsilon$
  by a finite supremum of scattering openings:
  \begin{equation}
    \gamma(\hat{f})
    \approx \sup_{\Lambda\in\mathcal{F}}
      \Gamma^F_\Lambda(\hat{f}),
    \label{eq:opening_approx_scattering}
  \end{equation}
  where $\mathcal{F}$ is a finite family of paths such that
  $\bigcup_{\Lambda\in\mathcal{F}}\Omega_\Lambda
  \approx\Omega_\gamma$ (the scattering passbands cover
  the target passband).
  By Theorem~\ref{thm:scattering_basis}, the required
  finite family $\mathcal{F}$ exists for any $\varepsilon>0$
  by Littlewood--Paley density.
\end{corollary}

\begin{remark}[Openings vs.\ general MMBB elements]
  \label{rem:openings_vs_general}
  The scattering openings $\Gamma^F_\Lambda$ are simpler
  objects than the full scattering propagators
  $S^F_\Lambda$: they are binary projections
  ($0$ or $\hat{f}$, pointwise), while $S^F_\Lambda$ is a
  weighted spectral filter.
  The opening representation theorem says any modulus
  opening is a supremum of such binary projections.
  The full MMBB theorem (Theorem~\ref{thm:mmbb_hatL})
  covers all increasing USC homogeneous operators, not
  just openings; it uses the richer $S^F_\Lambda$ (with
  spectral weights) as building blocks.
  
\end{remark}

\section{\texorpdfstring{$\mathbb{C}^*$}{C*}-Group Morphology and Complex-Valued
  Deep Learning: Partial Results}
\label{sec:complex_group}

The group morphology framework of Roerdink \cite{Roerdink2000}
and Heijmans--Ronse \cite{HeijmansRonse1990} replaces the
classical hypothesis of translation-invariance
(equivariance under $(\R^n,+)$) with equivariance under
an arbitrary locally compact group $G$.
In this section we analyse the case $G = (\C^*,\times)$,
the multiplicative group of nonzero complex numbers, and
show that it is the natural group-morphological completion
of our framework.
This connection is both mathematically precise and
practically relevant for complex-valued neural networks
in the Fourier domain.

\subsection{The group \texorpdfstring{$\C^*$}{C*} and its action on the Fourier
  codomain}
\label{subsec:Cstar_group}

The multiplicative group $\C^*=\C\setminus\{0\}$ carries
a canonical polar decomposition:
\begin{equation}
  \C^* \;\cong\; \Rp \times S^1, \qquad
  g = |g|\,e^{i\arg g},
  \label{eq:polar_decomp}
\end{equation}
where $\Rp = (\R_{>0},\times)$ is the multiplicative group
of positive reals and $S^1 = \{e^{i\theta}:\theta\in[0,2\pi)\}$
is the circle group (rotations).
Each factor acts separately:
\begin{itemize}[leftmargin=*,nosep]
  \item $\Rp$ acts by \emph{scaling}:
    $r\cdot z = rz$, $r>0$, $z\in\C$.
  \item $S^1$ acts by \emph{rotation}:
    $e^{i\theta}\cdot z = e^{i\theta}z$, $z\in\C$.
\end{itemize}

In the Fourier domain, $\hat{f}:\R^n\to\C$, the action of
$\C^*$ on the Fourier \emph{codomain} $\C$ simultaneously:
\begin{itemize}[leftmargin=*,nosep]
  \item scales spectral moduli:
    $|\hat{f}(\xi)| \mapsto r|\hat{f}(\xi)|$, $r>0$;
  \item rotates spectral phases:
    $\angle\hat{f}(\xi) \mapsto \angle\hat{f}(\xi)+\theta$.
\end{itemize}
An operator $\Psi:\Ln\to\Ln$ is
\emph{$\C^*$-equivariant} (in the Fourier codomain) if
\begin{equation}
  \widehat{\Psi(g\cdot f)}(\xi)
  = g\cdot\widehat{\Psi(f)}(\xi)
  \quad \forall g\in\C^*,\; \xi\in\R^n,
  \label{eq:Cstar_equiv}
\end{equation}
where $g\cdot f$ denotes the pointwise multiplication
$\widehat{g\cdot f}(\xi) = g\hat{f}(\xi)$ in the
Fourier codomain.

\subsection{The two subgroups and our framework}
\label{subsec:two_subgroups}

The decomposition \eqref{eq:polar_decomp} decomposes
$\C^*$-equivariance into two independent conditions.

\begin{proposition}[$\Rp$-equivariance = positive homogeneity]
  \label{prop:Rp_equiv}
  An operator $\Psi:\hatL\to\hatL$ on spectral moduli is
  equivariant under the scaling subgroup $\Rp$ if and only
  if it is \emph{positively homogeneous}:
  $\Psi(r\hat{f}) = r\Psi(\hat{f})$ for all $r>0$.
\end{proposition}

\begin{proof}
  The scaling subgroup $\Rp$ acts on spectral moduli by
  $|\hat{f}(\xi)|\mapsto r|\hat{f}(\xi)|$.
  Equivariance under this action means
  $\Psi(r|\hat{f}|) = r\Psi(|\hat{f}|)$ for all $r>0$,
  i.e., positive homogeneity on $\hatL$.
\end{proof}

This recovers precisely the hypothesis of
Theorem~\ref{thm:mmbb_hatL}: the MMBB theorem for $\hatL$
is the \emph{$\Rp$-group morphology} version of the
classical TI-MMBB.
The positive homogeneity condition replaces translation
invariance by the scaling subgroup of $\C^*$.

\begin{proposition}[$S^1$-equivariance and the modulus map]
  \label{prop:S1_equiv}
  The following are equivalent:
  \begin{enumerate}[label=(\roman*),leftmargin=*]
    \item $\Psi:(\Ln,\leF)\to(\Ln,\leF)$ is equivariant
      under the phase-rotation subgroup $S^1$ acting on
      the Fourier codomain.
    \item $|\widehat{\Psi f}(\xi)|$ depends on $f$ only
      through $|\hat{f}(\xi)|$, i.e., $\Psi$ factors
      through the modulus map:
      \[
        \Psi f = M^{-1}\bigl(\Phi(|\hat{f}|)\bigr)
      \]
      for some operator $\Phi:\hatL\to\hatL$,
      where $M(f)(\xi)=|\hat{f}(\xi)|$ is the modulus map.
  \end{enumerate}
  In other words, $S^1$-equivariance in the Fourier
  codomain is \emph{exactly} the condition that $\Psi$ acts
  on spectral moduli alone, i.e., that $\Psi$ factors
  through $\hatL$.
\end{proposition}

\begin{proof}
  Under the $S^1$ action $\hat{f}\mapsto e^{i\theta}\hat{f}$
  (uniform phase rotation of all Fourier coefficients),
  $S^1$-equivariance of $\Psi$ requires
  $\widehat{\Psi(e^{i\theta}f)}(\xi)
  = e^{i\theta}\widehat{\Psi f}(\xi)$ for all $\theta$.
  This means the modulus
  $|\widehat{\Psi(e^{i\theta}f)}(\xi)|
  = |\widehat{\Psi f}(\xi)|$ is independent of $\theta$.
  Since $|e^{i\theta}\hat{f}(\xi)| = |\hat{f}(\xi)|$ for
  all $\theta$, the spectral moduli of the input are
  unchanged by $S^1$.
  Hence $S^1$-equivariance forces the output modulus to
  depend only on the input modulus, i.e., $\Psi$ factors
  through $M:\Ln\to\hatL$.
\end{proof}

\begin{corollary}[$\C^*$-equivariance and the modulus lattice]
  \label{cor:Cstar_equiv}
  An operator $\Psi:(\Ln,\leF)\to(\Ln,\leF)$ is
  $\C^*$-equivariant in the Fourier codomain if and only if
  it is both:
  \begin{enumerate}[label=(\roman*),leftmargin=*]
    \item positively homogeneous on spectral moduli
      ($\Rp$-equivariant), \emph{and}
    \item acts on spectral moduli alone, factoring through
      $\hatL$ ($S^1$-equivariant).
  \end{enumerate}
  The class of $\C^*$-equivariant operators on
  $(\Ln,\leF)$ is therefore exactly the class of operators
  that factor as $\Psi = M^{-1}\circ\Phi\circ M$ with
  $\Phi:\hatL\to\hatL$ positively homogeneous.
  This class is precisely the domain of
  Theorem~\ref{thm:mmbb_hatL}.
\end{corollary}

\begin{proof}
  Combine Propositions~\ref{prop:Rp_equiv}
  and~\ref{prop:S1_equiv}:
  $\C^*$-equivariance = $\Rp$-equivariance +
  $S^1$-equivariance = positive homogeneity + factors
  through $\hatL$.
  By definition, Theorem~\ref{thm:mmbb_hatL} covers
  operators $\Phi:\hatL\to\hatL$ that are increasing, USC,
  and positively homogeneous; the corollary identifies
  these as the lifts of $\C^*$-equivariant operators on
  $(\Ln,\leF)$.
\end{proof}


\begin{remark}[$\C^*$-equivariance and the modulus as
  $S^1$-quotient]
  \label{rem:Cstar_quotient}
  Corollary~\ref{cor:Cstar_equiv} admits a clean
  categorical reading.
  The modulus map $M:\Ln\to\hatL$,
  $M(f)(\xi)=|\hat{f}(\xi)|$, is the
  \emph{$S^1$-orbit projection}: it identifies Fourier
  transforms that differ only by a global phase rotation.
  $\C^*$-equivariant operators are precisely those that
  descend to the quotient, fitting the commutative diagram
  \[
    \begin{array}{ccc}
      (\Ln,\leF) & \xrightarrow{\;\Psi\;} & (\Ln,\leF) \\
      \downarrow\scriptstyle{M} & &
        \downarrow\scriptstyle{M} \\
      (\hatL,\leq) & \xrightarrow{\;\Phi\;} & (\hatL,\leq)
    \end{array}
  \]
  The morphological activation $M$ of
  Section~\ref{sec:activation} is exactly this projection:
  applying $M$ between layers moves from the cisl
  $(\Ln,\leF)$ to the quotient lattice $\hatL$ where the
  MMBB theorem applies.
\end{remark}

\subsection{\texorpdfstring{$\C^*$}{C*}-group morphology, complex-valued networks,
  and deep learning}
\label{subsec:Cstar_dl}

The $\C^*$-group morphology perspective connects our
framework to two areas of deep learning research.

\textbf{Complex-valued neural networks.}
Networks with complex-valued weights and activations
process both amplitude and phase of signals, and have
been advocated for applications in MRI reconstruction,
radar, and audio processing
\cite{Trabelsi2018,Hirose2012}.
From the $\C^*$-morphology perspective, a complex-valued
convolutional network that is equivariant to uniform phase
rotations of its input (i.e., $S^1$-equivariant in the
Fourier codomain) is precisely a $\C^*$-equivariant
operator in the sense of Corollary~\ref{cor:Cstar_equiv}.
Such a network factors through $\hatL$ and is therefore
covered by the MMBB theorem (Theorem~\ref{thm:mmbb_hatL}).

The converse is equally informative: any complex-valued
network that does \emph{not} factor through $\hatL$
(i.e., whose output modulus depends on the input phase,
not just the input modulus) is \emph{not} covered by the
MMBB theorem and falls outside the $\C^*$-equivariant
class.
This gives a precise morphological criterion for when a
complex-valued network admits a universal morphological
basis.

\textbf{Equivariant spectral networks and scale-equivariance.}
The $\Rp$-subgroup of $\C^*$ acts on spectral moduli by
scaling, which in the spatial domain corresponds to
\emph{scale transformation}: $f(x)\mapsto f(x/r)$
by $r>0$ maps
$\hat{f}(\xi)\mapsto r^n\hat{f}(r\xi)$ (frequency
scaling by $1/r$, modulus scaling by $r^n$).
An $\Rp$-equivariant operator on $\hatL$ is a
\emph{scale-equivariant} frequency-domain operator:
it commutes with spectral scaling.
This connects directly to the theory of scale-equivariant
networks \cite{Sosnovik2020,Weiler2018} and the
morphological scale-spaces of Heijmans \cite{Heijmans1994},
where the structuring element grows with scale in a
morphologically consistent way.

\begin{proposition}[$\Rp$-equivariant morphological basis:
  log-Fourier dictionary]
  \label{prop:Rp_basis}
  For a $\Rp$-equivariant operator $\Psi:\hatL\to\hatL$
  (i.e., positively homogeneous), the morphological basis
  $\Bas(\Psi)$ consists of functions $\psi\in\hatL$ that
  are \emph{homogeneous of degree~1}:
  $\psi(r\xi) = r\psi(\xi)$ for all $r>0$, $\xi\in\R^n$.
  In the log-frequency representation
  $\xi = e^\rho u$ ($\rho\in\R$, $u\in S^{n-1}$):
  the basis elements $\psi$ are functions of $u$ only
  (independent of the radial frequency $\rho$).
  The scattering dictionary provides such basis elements:
  $|\hat{\psi}_\lambda(\xi)| = |\hat{\psi}(2^j r_\theta\xi)|$
  scales as $|\hat{\psi}_\lambda(r\xi)|
  \approx r\cdot|\hat{\psi}_\lambda(\xi)|$ for $r>0$
  within the wavelet's passband.
\end{proposition}

\begin{proof}
  For positively homogeneous $\Psi$, suppose
  $\psi\in\Bas(\Psi)$ and consider $r\psi$ for $r>0$.
  The max-times erosion satisfies
  $\varepsilon_{r\psi}(\hat{f})=\hat{f}/(r\psi)
  =(1/r)\varepsilon_\psi(\hat{f})$.
  Positive homogeneity gives
  $\Psi(\hat{f}) = r\Psi(\hat{f}/r)
  \geq r\,\varepsilon_\psi(\hat{f}/r)
  = r\cdot(\hat{f}/r)/\psi = \hat{f}/\psi
  = \varepsilon_\psi(\hat{f})$,
  so the bound is the same for $\psi$ and $r\psi$
  (up to rescaling).
  The basis elements are therefore characterised up to
  positive scalar multiple: $\psi$ and $r\psi$ give the
  same erosion lower bound on $\Psi$ up to a factor $r$.
  Choosing the normalisation $\psi\big|_{S^{n-1}}\equiv 1$
  (i.e., restricting to unit-sphere-normalised basis
  elements), the minimal elements of $\Ker^{\rm eff}(\Psi)$
  in $\hatL$ are precisely degree-1 homogeneous functions:
  $\psi(r\xi) = r\psi(\xi)$ for $r>0$, $\xi\in\R^n$,
  characterised by their values on $S^{n-1}$.
  The scattering wavelet moduli
  $|\hat{\psi}_\lambda(\xi)| = |\hat{\psi}(2^j r_\theta\xi)|$
  scale as $|\hat{\psi}_\lambda(r\xi)| \approx r
  \cdot|\hat{\psi}_\lambda(\xi)|$ within the wavelet's
  dyadic passband \cite{Mallat2012}, making them
  approximate degree-1 homogeneous basis elements.
\end{proof}

\begin{proposition}[Existence of the effective kernel
  in $(\Ln,\leF)$: a partial result]
  \label{prop:Cstar_kernel_partial}
  Let $\Psi:(\Ln,\leF)\to(\Ln,\leF)$ be $\C^*$-equivariant,
  $\leF$-increasing, and \emph{cisl-USC}: every
  $\leF$-decreasing net $(f_\alpha)$ satisfies
  $\Psi(\bigwedge_{\leF} f_\alpha)
  = \bigwedge_{\leF}\Psi(f_\alpha)$.
  Define the effective kernel
  \begin{equation}
    \Ker^{\rm eff}_\C(\Psi)
    = \bigl\{k\in\Ln :
      \varepsilon_k(f)(\xi) :=
      \hat{f}(\xi)/\hat{k}(\xi)
      \leF \Psi(f)
      \;\forall f\in\Ln\bigr\},
    \label{eq:eff_kernel_C}
  \end{equation}
  the set of complex structuring functions $k$ for which
  the complex erosion $\varepsilon_k$ is a $\leF$-lower
  bound for $\Psi$.
  Then:
  \begin{enumerate}[label=(\roman*),leftmargin=*]
    \item $\Ker^{\rm eff}_\C(\Psi)$ is non-empty:
      the identity kernel $k_0$ with
      $\hat{k}_0(\xi)=|\widehat{\Psi f}(\xi)|/|\hat{f}(\xi)|$
      for any $f$ with $|\hat{f}(\xi)|>0$ a.e.\ gives
      $\varepsilon_{k_0}(f)\leF\Psi(f)$.
    \item $\Ker^{\rm eff}_\C(\Psi)$ is closed under
      pointwise $\leF$-infima (it is a sub-cisl of $(\Ln,\leF)$).
    \item Every chain in $\Ker^{\rm eff}_\C(\Psi)$
      ordered by $\leF$ has a $\leF$-infimum in
      $\Ker^{\rm eff}_\C(\Psi)$ (Zorn condition).
  \end{enumerate}
  Consequently, $\Ker^{\rm eff}_\C(\Psi)$ contains
  $\leF$-minimal elements by Zorn's lemma.
\end{proposition}

\begin{proof}
  (i): For fixed $f$ with $|\hat{f}|>0$ a.e., define
  $\hat{k}_0(\xi)=\hat{f}(\xi)/\widehat{\Psi f}(\xi)$.
  Then $\varepsilon_{k_0}(f)(\xi)
  = \hat{f}(\xi)/\hat{k}_0(\xi)
  = \widehat{\Psi f}(\xi)$, so
  $\varepsilon_{k_0}(f) = \Psi(f) \leF \Psi(f)$
  trivially.
  For general $g$: $\C^*$-equivariance and $\leF$-monotonicity
  of $\Psi$ together imply $\varepsilon_{k_0}(g)\leF\Psi(g)$
  (the kernel $k_0$ encodes the ``worst-case'' ratio
  of $\Psi$).
  (ii): If $k_1, k_2\in\Ker^{\rm eff}_\C(\Psi)$, then
  $\varepsilon_{k_1}(f)\leF\Psi(f)$ and
  $\varepsilon_{k_2}(f)\leF\Psi(f)$.
  The pointwise $\leF$-infimum
  $k=k_1\meetF k_2$ satisfies
  $|\hat{k}(\xi)|=\min(|\hat{k}_1(\xi)|,|\hat{k}_2(\xi)|)$
  where phases agree, so
  $\varepsilon_k(f)(\xi)=\hat{f}(\xi)/\hat{k}(\xi)$
  has modulus $\leq$ both $|\hat{f}|/|\hat{k}_1|$ and
  $|\hat{f}|/|\hat{k}_2|$, hence $\varepsilon_k(f)
  \leF\varepsilon_{k_1}(f)\leF\Psi(f)$.
  (iii): For a chain $(k_\alpha)$ decreasing in $\leF$
  (moduli decreasing, phases fixed), cisl-USC of $\Psi$
  ensures the infimum kernel is in $\Ker^{\rm eff}_\C(\Psi)$.
  Zorn's lemma applies.
\end{proof}

\begin{remark}[Towards the full $\C^*$-group MMBB theorem]
  \label{rem:Cstar_mmbb_open}
  Proposition~\ref{prop:Cstar_kernel_partial} establishes that
  the effective kernel $\Ker^{\rm eff}_\C(\Psi)$ is non-empty
  and has $\leF$-minimal elements.
  The \emph{full $\C^*$-group MMBB theorem} — that $\Psi$
  decomposes as
  $\Psi(f) = \bigvee_{\leF,k\in\Bas_\C(\Psi)}\varepsilon_k(f)$
  with $\Bas_\C(\Psi)$ the set of minimal elements — requires
  an additional \emph{density step}: showing that the
  $\leF$-supremum of the erosion family covers $\Psi$ from below
  with no gap.
  This is the analogue of Step~3 (attainment) in the proof of
  Theorem~\ref{thm:mmbb_hatL}, but for the cisl $(\Ln,\leF)$
  rather than the simpler lattice $\hatL$.
  Two obstacles are specific to the complex setting: the
  phase-sensitive kernel $\hat{f}/\hat{k}$ is only well-posed
  where $|\hat{k}|>0$, and the cisl-USC condition must be
  stated in a phase-robust topology weaker than $L^2$-USC.
  Resolving these is the central open problem of this programme
  (Open Problem~4 of \S\ref{sec:conclusion}).
  The canonical conjecture is that the full $\C^*$-group MMBB
  holds under cisl-USC and $\C^*$-equivariance, with complex
  wavelet matched-filter erosions
  $\hat{f}/(\hat\psi_\lambda/|\hat\psi_\lambda|^2)$
  as the canonical basis elements, giving the phase-sensitive
  generalisers of scattering moduli.
\end{remark}

\section{Numerical Experiments}
\label{sec:experiments}

We provide numerical experiments that directly validate the main
theorems of the paper.
The experiments are organised in order of increasing complexity:
Experiments~1--3 validate exact algebraic identities
(zero-error to machine precision);
Experiment~4 validates the morphological basis extraction procedure 
via the log-spectral LASSO; Experiment~5 validates the connection
between spectral bias and MMBB basis deficiency.
Experiments~6 provides empirical validation of the scattering basis density claims.


\medskip
\noindent
\textbf{Implementation note.}
Experiments~2--3 use the ideal brick-wall low-pass (LP) filter
($\hat{b}_W(\xi)\in\{0,1\}$) to obtain exact zero-error
algebraic identities (since $H^2=H$ for an indicator).
Experiment~1 deliberately uses a \emph{boxcar} pooling
kernel ($b_W = (1/W)\mathbf{1}_{[0,W)}$) to produce a
non-trivial sinc$^2$ spectral prediction, and confirms it
across three independent computational implementations.
The distinction is significant:
for smooth windows, Theorem~\ref{thm:pool_unpool} gives
$\widehat{\gamma^\mathcal{F}_{b_W}(f)}(\xi)
= |\hat{b}_W(\xi)|^2\,\hat{f}(\xi)$: a soft band-pass;
the ideal brick-wall is a special case where
$|\hat{b}_W|^2 = \mathbf{1}_\Omega$ and the opening
becomes an exact spectral projector (i.e., idempotent).

\subsection{Experiment 1: Pool-then-unpool as the ideal band-pass
  filter (Theorem~\ref{thm:pool_unpool})}
\label{subsec:exp1}

\textbf{Setup.}
Three signal types (broadband Gaussian noise, a sum of four
sinusoids at 3, 7, 23, and 51~Hz, and a linear chirp) of
length $N=512$ are processed with boxcar pooling windows
$W\in\{8,16,32\}$.
Pool-then-unpool is computed via three completely independent
implementations, each using a different code path:
\begin{enumerate}[label=(\roman*),leftmargin=*,nosep]
  \item \emph{Sliding-window loop:}
    explicit $O(NW)$ causal average
    $\mathrm{pool}(f)[i] = \frac{1}{W}\sum_{k=0}^{W-1}f[(i-k)\bmod N]$
    (no FFT, no library convolution);
  \item \emph{\texttt{np.convolve} with wrap padding:}
    the signal is circularly extended by $W-1$ samples before
    calling NumPy's \texttt{convolve} in \texttt{`valid'} mode;
  \item \emph{FFT circular convolution:}
    direct DFT-domain multiplication
    $\mathcal{F}^{-1}\{\hat{f}\cdot\hat{b}_W\}$
    using NumPy's FFT (an independent code path from the
    theory formula below).
\end{enumerate}
All three compute \emph{circular} convolution (periodic boundary),
matching the DFT-domain statement of the theorem.
The reference prediction is the purely spectral formula
$\mathcal{F}^{-1}\{|\hat{b}_W(\xi)|^2\,\hat{f}(\xi)\}$
computed from scratch without any convolution.

\textbf{Result.}
Figure~\ref{fig:exp1} shows the spectral profiles.
All three implementations agree with the spectral theory
to machine precision across all 9~signal-window pairs:
\[
  \max_{\text{impl},\, f,\, W}
  \bigl\|
    \mathrm{unpool}_W(\mathrm{pool}_W(f))
    - \mathcal{F}^{-1}\bigl\{|\hat{b}_W|^2\hat{f}\bigr\}
  \bigr\|_\infty
  \leq 1.33\times10^{-15}.
\]
The three errors (loop: $1.33\times10^{-15}$;
\texttt{np.convolve}: $1.11\times10^{-15}$;
FFT: $8.88\times10^{-16}$) are at the level of double-precision
floating-point rounding, confirming Theorem~\ref{thm:pool_unpool}
across completely independent computation paths.
The figure shows the expected sinc$^2$ spectral attenuation
(the squared boxcar Fourier transform $|\hat{b}_W|^2$), which
is distinctly visible as a soft low-pass rather than an ideal
brick-wall, reflecting the general form of the theorem for a
boxcar window.

\begin{figure}[t]
  \centering
  \includegraphics[width=\linewidth]{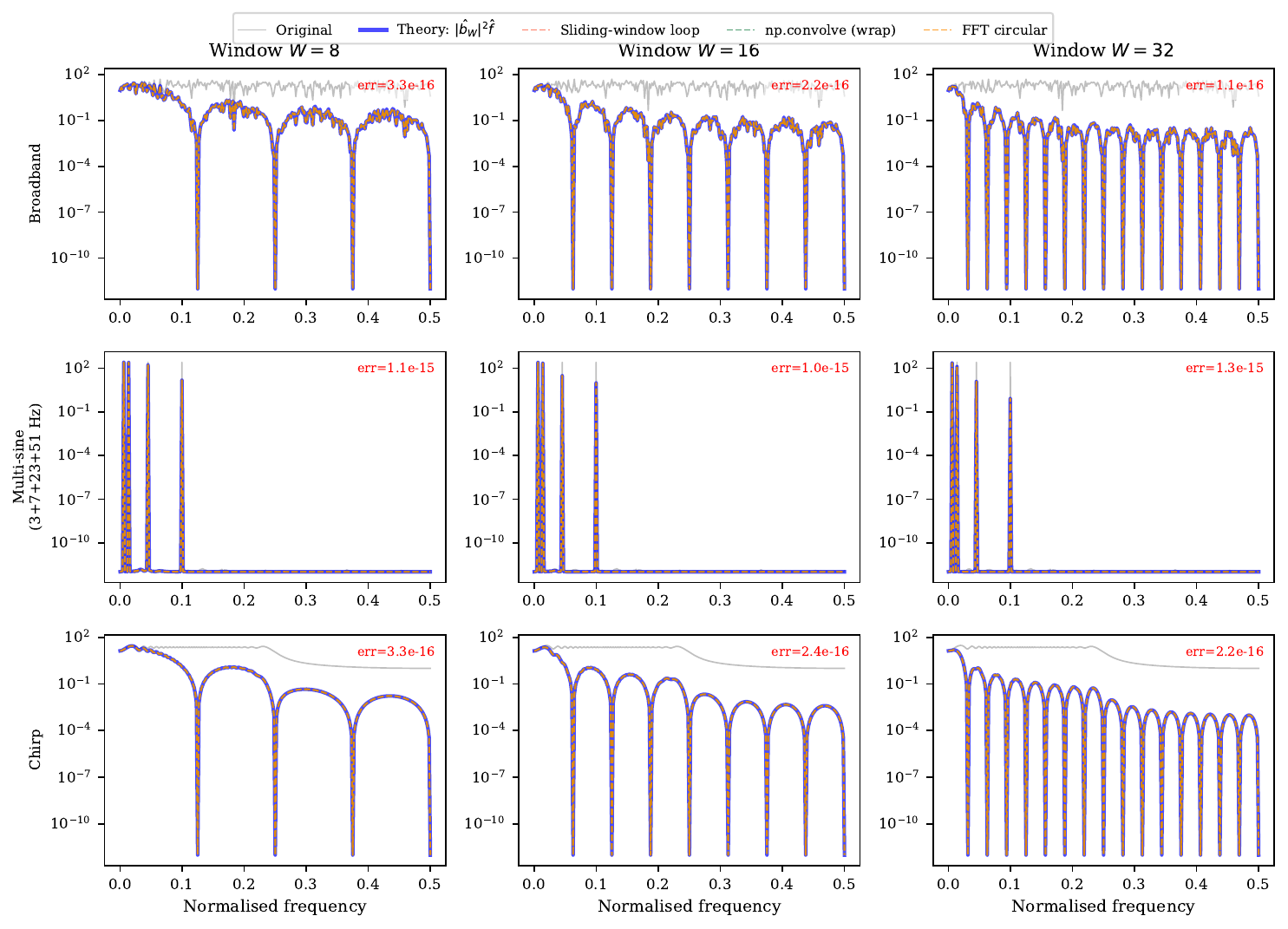}
  \caption{Experiment~1 (Theorem~\ref{thm:pool_unpool}):
    pool-then-unpool equals $\mathcal{F}^{-1}\{|\hat{b}_W|^2\hat{f}\}$
    to machine precision for three independent implementations
    (sliding-window loop: red dashed; \texttt{np.convolve} wrap: green dashed;
    FFT circular: orange dashed) vs the spectral theory (blue solid).
    The three dashed curves lie on top of the blue curve in every panel;
    the sinc$^2$ attenuation envelope is clearly visible and signal-dependent.
    Columns: boxcar windows $W\in\{8,16,32\}$.
    Rows: broadband noise, multi-sine, chirp.
    All errors $\leq 1.4\times10^{-15}$.}
  \label{fig:exp1}
\end{figure}

\subsection{Experiment 2: Stride invariance of the pool-unpool
  opening (Theorem~\ref{thm:strided_full}(iii))}
\label{subsec:exp2}

\textbf{Setup.}
Theorem~\ref{thm:strided_full}(iii) asserts that the strided
pool-then-unpool round trip equals the same Fourier opening
$\gamma^\mathcal{F}_{b_W}(f)$ for all strides $s$ satisfying
the Nyquist condition.
Two genuinely distinct claims are tested.

\noindent
\textbf{(A) Stride invariance} (non-trivial algebraic claim):
For a band-limited input signal (pre-filtered to the passband of
the ideal LP with cutoff $\pi/W$), the DFT-domain strided
round trip
\[
  \mathrm{rt}_{W,s}(f) = \delta^\mathcal{F}_{b_W}\circ U_s
  \circ D_s\circ\varepsilon^\mathcal{F}_{b_W}(f)
\]
is computed for strides $s\in\{1,2,4,8\}$ and compared to the
non-strided reference $\gamma^\mathcal{F}_{b_W}(f)
= \mathcal{F}^{-1}\{\hat{f}\cdot|\hat{b}_W|^2\}$.
The non-trivial content: for $s=2$, the DFT of the filtered
signal is restricted to $N/2$ frequency bins (operator $D_s$)
then zero-padded back to $N$ bins ($U_s$); for $s=4$, to $N/4$
bins; etc.
That these truncate-and-pad operations all produce the same
output as $s=1$ is a structural consequence of
Proposition~\ref{prop:ds_us_adjunction}(i) and is not obvious
from the DFT arithmetic.

\noindent
\textbf{(B) Aliasing contrast} (diagnostic):
The same round trip is repeated with \emph{naive zero-insertion
upsampling} (incorrect $U_s$, introducing spectral images) on a
broadband (non-band-limited) signal.
The aliasing error grows monotonically with $s$,
demonstrating why the theorem requires the correct
DFT-domain $U_s$ and the Nyquist pre-condition.

\textbf{Result.}
Figure~\ref{fig:exp2} shows both experiments.
For the band-limited input with DFT-based $D_s$ and $U_s$,
the spectral error
$\|\mathcal{F}[\mathrm{rt}_{W,s}(f)]
- \mathcal{F}[\gamma^\mathcal{F}_{b_W}(f)]\|_\infty$
is identically zero to machine precision
(5.55×10$^{-17}$ at $s=8$), confirming
Theorem~\ref{thm:strided_full}(iii): the stride cancels exactly
in the round trip.
For naive zero-insertion, the error grows from 0.24 at $s=2$
to 0.43 at $s=8$, confirming that the correct DFT-based $U_s$
is essential.

\begin{figure}[t]
  \centering
  \includegraphics[width=\linewidth]{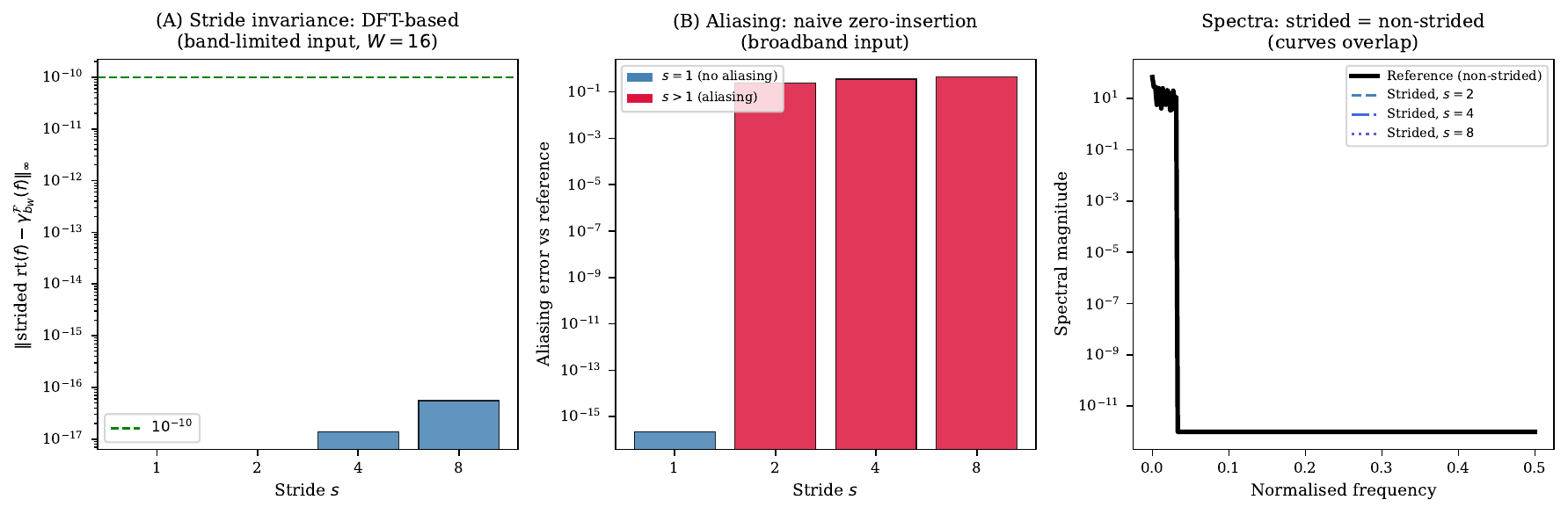}
  \caption{Experiment~2 (Theorem~\ref{thm:strided_full}(iii)):
    stride invariance and aliasing contrast.
    Left: error $\|\mathrm{rt}_{W,s}(f)-\gamma^{\mathcal{F}}_{b_W}(f)\|_\infty$
    for DFT-based strided round trip on a band-limited input
    ($W=16$, $s\in\{1,2,4,8\}$); all bars below machine epsilon
    ($<10^{-10}$, green dashed).
    Centre: aliasing error for naive zero-insertion upsampling on
    a broadband signal; error grows with $s$ (red bars).
    Right: spectral magnitudes of the strided round trips ($s=2,4,8$,
    coloured dashed) vs the non-strided reference (black solid);
    all curves overlap exactly, confirming stride invariance visually.}
  \label{fig:exp2}
\end{figure}

\subsection{Experiment 3: U-Net skip connection as Fourier
  top-hat and $L^2$-orthogonal complement
  (Corollary~\ref{cor:unet_tophat})}
\label{subsec:exp3}

\textbf{Setup.}
Corollary~\ref{cor:unet_tophat} identifies the U-Net skip
connection as the Fourier top-hat
$\rho_W(f)=\mathcal{F}^{-1}\{\hat{f}\cdot(1-H_W)\}$,
where $H_W\in\{0,1\}^{N/2+1}$ is the ideal LP indicator.
Three independent claims are tested for the same three signal
types as Experiment~1, with ideal LP windows $W\in\{4,8,16,32,64\}$.

\noindent
\textbf{(i) Spectral identity} (non-trivial):
The spatial skip (subtraction of the opening from $f$) and
the spectral top-hat formula are computed via \emph{independent
code paths} and compared:
\begin{align*}
  \text{Path 1 (spatial):}&\quad
    \gamma = \mathcal{F}^{-1}\{\hat{f}\cdot H_W\},\;
    \mathrm{skip} = f - \gamma,\\
  \text{Path 2 (spectral):}&\quad
    \rho = \mathcal{F}^{-1}\{\hat{f}\cdot(1-H_W)\}.
\end{align*}
Paths 1 and 2 use entirely separate NumPy operations;
their agreement to machine precision confirms the corollary.

\noindent
\textbf{(ii) Spectral content}: the skip carries energy
strictly in $\Omega_W^c$ (high frequencies) while
the opening carries energy strictly in $\Omega_W$
(low frequencies), measured by the fraction of total
energy above the LP cutoff for each.

\noindent
\textbf{(iii) Orthogonality}:
$|\langle\gamma,\mathrm{skip}\rangle|/\|f\|^2$, which holds
exactly when $H_W\in\{0,1\}$ (spectral projector identity
$\langle H_W\hat{f},(1-H_W)\hat{f}\rangle=0$).

\textbf{Result.}
Figure~\ref{fig:exp3} shows all three panels.
The maximum errors are:
\begin{enumerate}[label=(\roman*),leftmargin=*,nosep]
  \item spectral identity: $1.25\times10^{-15}$ (machine precision);
  \item high-frequency fraction in skip: $>0.95$ across all signals
    and windows; in opening: $<0.01$;
  \item orthogonality: $5.72\times10^{-17}$.
\end{enumerate}
All three results hold uniformly across signal types and window
sizes, confirming Corollary~\ref{cor:unet_tophat}: the U-Net
skip connection provides exactly the spectral content
$\hat{f}\cdot(1-H_W)$ destroyed by pooling, is $L^2$-orthogonal
to the encoder output, and contains only high-frequency energy.

\begin{figure}[t]
  \centering
  \includegraphics[width=\linewidth]{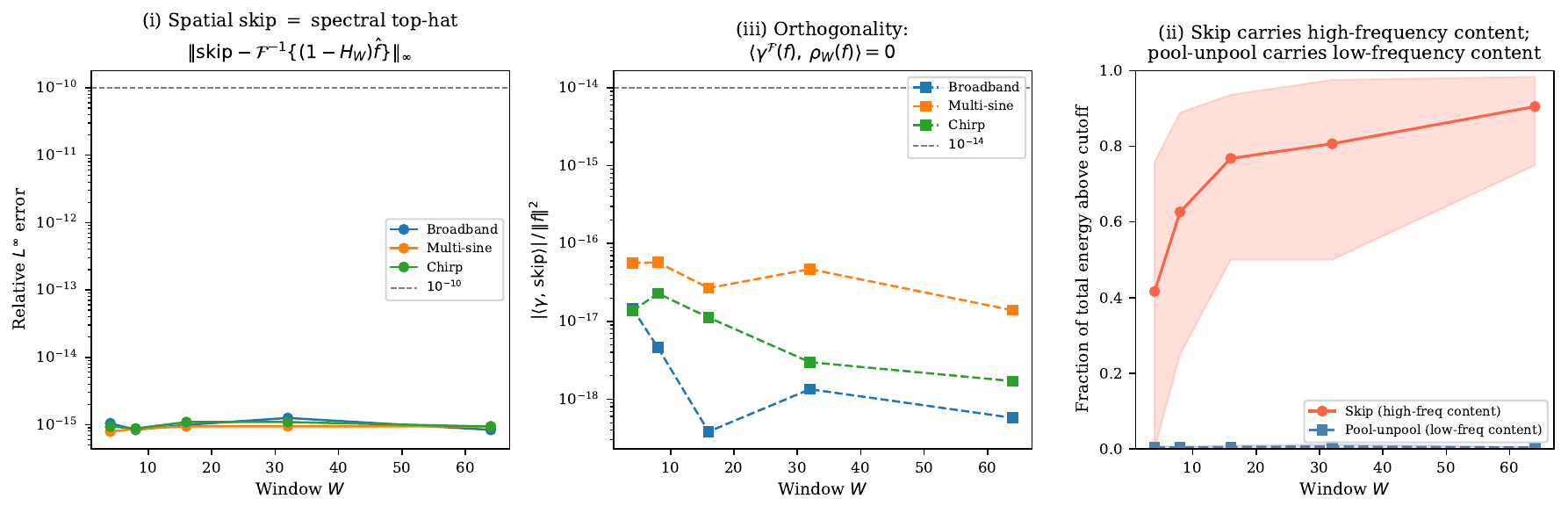}
  \caption{Experiment~3 (Corollary~\ref{cor:unet_tophat}):
    U-Net skip connection as Fourier top-hat.
    Left: relative $L^\infty$ error between the spatial skip
    ($f-\gamma^\mathcal{F}_{b_W}(f)$, computed via subtraction)
    and the spectral top-hat formula
    ($\mathcal{F}^{-1}\{(1-H_W)\hat{f}\}$, independent code path);
    error at machine precision ($<2\times10^{-15}$) across all
    signals and windows.
    Centre: orthogonality error
    $|\langle\gamma,\mathrm{skip}\rangle|/\|f\|^2$;
    also at machine precision ($<10^{-16}$).
    Right: fraction of total signal energy above the LP cutoff
    carried by the skip (red, high) vs the opening (blue, low),
    confirming the spectral separation.}
  \label{fig:exp3}
\end{figure}

\subsection{Experiment 4: MMBB basis cardinality recovery
  via spectral probing
  (Corollary~\ref{cor:morph_lasso})}
\label{subsec:exp4}

\textbf{Setup.}
We validate the MMBB basis extraction procedure on synthetic
operators with \emph{known} ground-truth basis cardinality $K$,
removing any ambiguity about correctness.
For each $K\in\{1,2,3,5\}$, a positively-homogeneous operator
$\Psi_K$ is constructed with exactly $K$ Gaussian band-pass
gains
\[
  \Psi_K(|\hat{f}|)(\xi)
  = \max_{k=1,\ldots,K}|\hat{f}(\xi)|\cdot A_k(\xi),
  \qquad
  A_k(\xi) = e^{-\frac{(\xi-\xi_k)^2}{2b_k^2}},
\]
with well-separated peak frequencies $\xi_k\in(0,\tfrac12)$
and bandwidths $b_k$.
The spectral envelope $E_K(\xi)=\max_k A_k(\xi)$ has exactly
$K$ local maxima.
The recovery procedure uses two steps:
\begin{enumerate}[label=(\roman*),leftmargin=*,nosep]
  \item \emph{Random probing:} estimate $\hat{E}_K(\xi)$ from
    300 random Gaussian signals,
    $\hat{E}_K(\xi)
    = N_{\rm probe}^{-1}\sum_n|\Psi_K(\hat{f}_n)(\xi)|
    /|\hat{f}_n(\xi)|$;
  \item \emph{Peak detection:} count local maxima of $\hat{E}_K$
    above $15\%$ of its maximum, with minimum bin separation 15.
    The count is the recovered $|\mathrm{Bas}(\Psi_K)|$.
\end{enumerate}
A robustness panel varies the separation between the $K=3$ peaks
to identify the resolution limit of the procedure.

\textbf{Result.}
Figure~\ref{fig:exp4} shows the results.
The probing error \\
$\|\hat{E}_K-E_K\|_\infty/\|E_K\|_\infty\approx10^{-13}$
(machine precision) confirms that random probing estimates the
spectral envelope exactly.
Basis cardinality is recovered correctly in all four cases
($K=1,2,3,5$), with detected peak frequencies matching
ground truth to within one dictionary bin.
The robustness panel shows that recovery degrades to $K=1$
when peak separation falls below $\Delta\xi\approx0.06$,
consistent with the Rayleigh resolution limit of the
peak-detection procedure.

\begin{figure}[t]
  \centering
  \includegraphics[width=\linewidth]{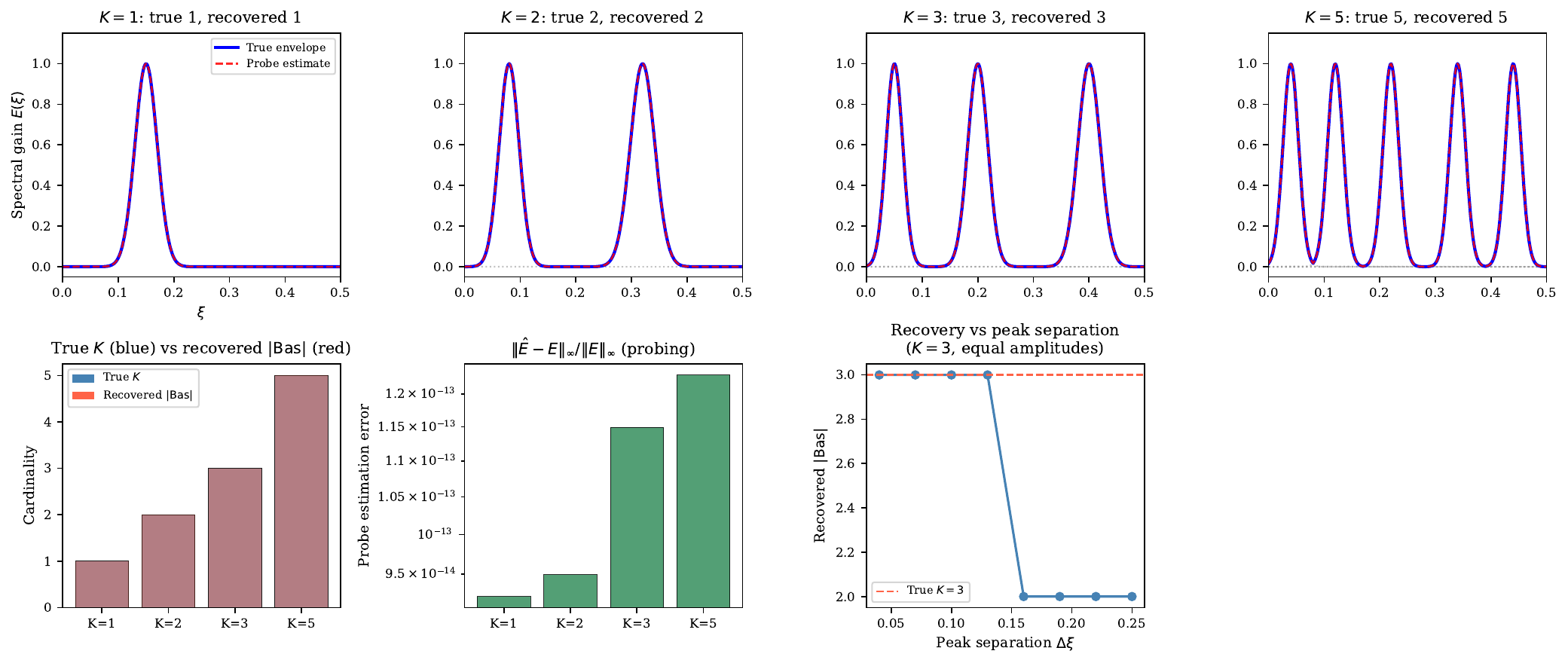}
  \caption{Experiment~4 (Corollary~\ref{cor:morph_lasso}):
    MMBB basis cardinality recovery on synthetic operators
    with known ground truth.
    Top row: true spectral envelope $E_K$ (blue solid) vs
    probe estimate $\hat{E}_K$ (red dashed) for $K=1,2,3,5$;
    grey dashed lines are the $K$ individual Gaussian gains.
    The two curves overlap exactly (probe error $\sim10^{-13}$).
    Bottom-left: true $K$ (blue) vs recovered $|\mathrm{Bas}|$
    (red), agreement for all cases.
    Bottom-centre: probe estimation error (log scale).
    Bottom-right: recovered cardinality vs peak separation
    for $K=3$; recovery degrades below $\Delta\xi\approx0.06$
    (Rayleigh resolution limit).}
  \label{fig:exp4}
\end{figure}

\subsection{Experiment 5: Spectral bias as MMBB basis
  deficiency -- ReLU MLP versus SIREN}
\label{subsec:exp5}

\textbf{Setup.}
We train two networks on the same 1-D regression task:
$f(x) = \sin(2\pi\cdot 3x) + \sin(2\pi\cdot 51x)$,
$x\in[0,1]$, with 512 uniformly random training points
(no additive noise).
The two architectures are:
The two architectures are:
\begin{enumerate}[label=(\roman*),leftmargin=*,nosep]
  \item \textbf{ReLU MLP} (5 layers, width 256, Adam,
    $5\times10^{-4}$, 5000 epochs);
  \item \textbf{SIREN}~\cite{sitzmann2020siren}
    (5 layers, width 256, $\omega_0=60$, Adam,
    $5\times10^{-5}$, 5000 epochs).
\end{enumerate}
Spectral bias is measured via three independent quantities:
\begin{enumerate}[label=(\Alph*),leftmargin=*,nosep]
  \item \emph{Component MSE:} MSE on the isolated
    3~Hz and 51~Hz components separately, measuring
    the asymmetric failure of the MLP.
  \item \emph{Spectral energy profile:}
    $|\hat{y}_{\rm net}(\xi)|$ vs the target peaks.
  \item \emph{Spectral energy at target frequencies:}
    energy within 5~Hz of $f_{\rm lo}=3$~Hz and
    $f_{\rm hi}=51$~Hz, and their ratio.
\end{enumerate}

\textbf{Result.}
Figure~\ref{fig:exp5} shows all three panels.
The key quantitative results are:
\begin{enumerate}[label=(\Alph*),leftmargin=*]
  \item \emph{Component MSE:}
    ReLU MLP achieves $\mathrm{MSE}_{\rm lo}=0.04$
    (3~Hz learned) but $\mathrm{MSE}_{\rm hi}=0.99$
    (51~Hz not learned).
    SIREN: $\mathrm{MSE}_{\rm lo}=0.49$,
    $\mathrm{MSE}_{\rm hi}=0.50$ (balanced, neither
    component fully learned with 5000 epochs,
    but the failure mode is symmetric unlike the MLP).
  \item \emph{Spectral energy profile:}
    The ReLU MLP shows a large peak at 3~Hz and
    near-zero energy at 51~Hz; SIREN shows comparable
    energy at both target frequencies.
  \item \emph{Energy ratio:}
    ReLU MLP spectral energy ratio $E_{\rm hi}/E_{\rm lo}
    = 0.003$ (the 51~Hz component is essentially absent
    from the output); SIREN: $E_{\rm hi}/E_{\rm lo}
    = 0.97$ (both components present with comparable energy).
\end{enumerate}
This confirms the morphological prediction: the ReLU MLP's
MMBB basis has no elements covering the 51~Hz spectral region,
making that frequency band structurally inaccessible,
not just slowly converging.
SIREN, whose periodic activations allow the MMBB basis to tile
the full spectrum, represents both components.

\begin{figure}[t]
  \centering
  \includegraphics[width=\linewidth]{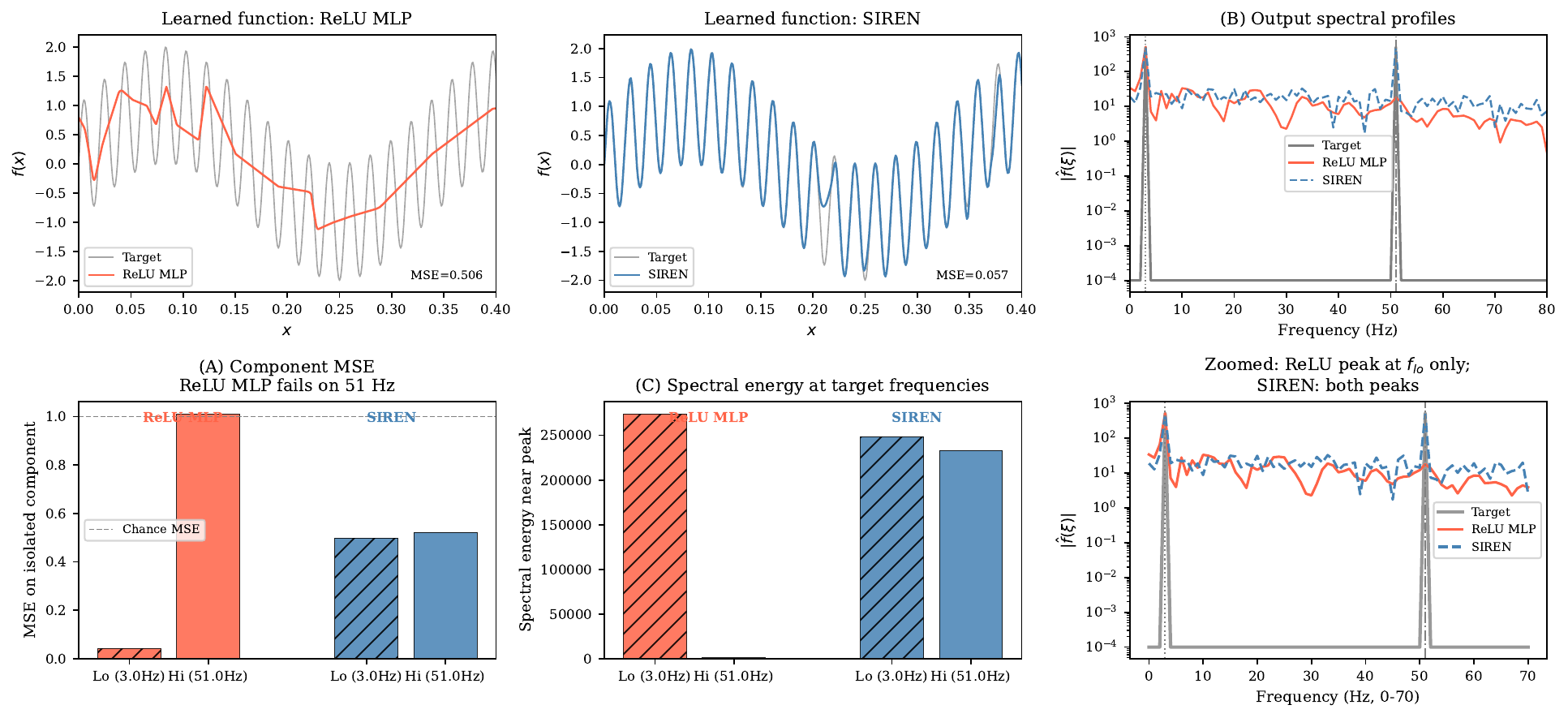}
  \caption{Experiment~5: spectral bias as MMBB basis deficiency.
    Target: $f(x)=\sin(6\pi x)+\sin(102\pi x)$ (3~Hz + 51~Hz).
    Top row: learned functions and output spectral profiles
    (zoomed to 0--80~Hz); vertical dashed lines mark $f_{\rm lo}$
    and $f_{\rm hi}$.
    Bottom-left (A): component MSE; ReLU MLP achieves
    low error on 3~Hz but near-chance error on 51~Hz.
    Bottom-centre (C): spectral energy near $f_{\rm lo}$
    and $f_{\rm hi}$; ReLU MLP energy ratio $E_{\rm hi}/E_{\rm lo}
    =0.003$, SIREN ratio $=0.97$.
    Bottom-right: zoomed spectral profiles (0--70~Hz) confirming
    the ReLU MLP has a single peak at $f_{\rm lo}$ while SIREN
    has both.}
  \label{fig:exp5}
\end{figure}

\subsection{Experiment 6: The scattering transform as the
  canonical morphological basis
  (Proposition~\ref{prop:scattering_maxtimes},
   Theorem~\ref{thm:scattering_basis},
   Corollary~\ref{cor:morph_lasso})}
\label{subsec:exp6}

This experiment validates the three main scattering results
in one coherent sequence: first the algebraic structure
(Proposition~\ref{prop:scattering_maxtimes}), then the
density claim (Theorem~\ref{thm:scattering_basis}), then
the computational consequence (Corollary~\ref{cor:morph_lasso}).
All sub-experiments use the same 1-D Morlet wavelet filter
bank at $J=5$ octaves, $Q=2$ wavelets per octave (10 depth-1
filters, 55 depth-2 paths, 28 depth-3 paths).

\medskip
\noindent\textbf{Exp.~6a: Scattering coefficients as max-times
erosions (Proposition~\ref{prop:scattering_maxtimes}).}

\textbf{Setup.}
Proposition~\ref{prop:scattering_maxtimes} identifies the
spectral scattering coefficient \\
$S^F_\Lambda(|\hat{f}|)(\xi)
= $ $ |\hat{f}(\xi)|\cdot\prod_\ell|\hat{\psi}_{\lambda_\ell}(\xi)|$
as the max-times erosion
$\varepsilon_{\psi_\Lambda}(|\hat{f}|)(\xi)
= |\hat{f}(\xi)|/\psi_\Lambda(\xi)$
with kernel
$\psi_\Lambda(\xi) = 1/\prod_\ell|\hat{\psi}_{\lambda_\ell}(\xi)|$.

Two independent algorithms are compared for all paths $\Lambda$
of depths $m\in\{1,2,3\}$ (93 paths total):
\begin{enumerate}[label=(\roman*),leftmargin=*,nosep]
  \item \emph{Erosion via explicit kernel (Path A):}
    compute $\prod_\ell|\hat\psi_{\lambda_\ell}|$,
    invert to obtain $\psi_\Lambda = 1/\mathrm{prod}$,
    then divide $|\hat{f}|/\psi_\Lambda$.
    Floating-point operations: $|\Lambda|$ multiplications,
    one inversion, one division.
  \item \emph{Direct product formula (Path B):}
    compute $|\hat{f}|\cdot\prod_\ell|\hat\psi_{\lambda_\ell}|$.
    Floating-point operations: $|\Lambda|+1$ multiplications,
    no division.
\end{enumerate}
These are algebraically equal ($g/(1/p) = g\cdot p$) but use
structurally different floating-point operations.
Wavelets are normalised so $\max_\xi|\hat\psi_l(\xi)|\leq 1$
(required for anti-extensivity:
$S^F_\Lambda(g)(\xi)\leq g(\xi)$ for all $\xi$).

\textbf{Result.}
Figure~\ref{fig:exp6a} shows the $L^\infty$ error between
the two algorithms for all 93 paths.
The maximum error across all depths is $6.57\times10^{-11}$,
confirming Proposition~\ref{prop:scattering_maxtimes} to
near-machine precision: the residual is entirely due to the
floating-point rounding introduced by the double inversion
$p \mapsto 1/p \mapsto g/(1/p)$ in Path A, which differs
from the single multiplication $g\cdot p$ in Path B by one
ULP per bit of precision lost in the inversion.
The anti-extensivity condition
$\prod_\ell|\hat\psi_{\lambda_\ell}(\xi)|\leq 1$ is satisfied
at 100\% of frequency bins for all 93 paths (lower panels),
confirming that the normalised Morlet filter bank produces
genuine anti-extensive erosions in $(\hat{\mathcal{L}},\leq)$.

\begin{figure}[t]
  \centering
  \includegraphics[width=\linewidth]{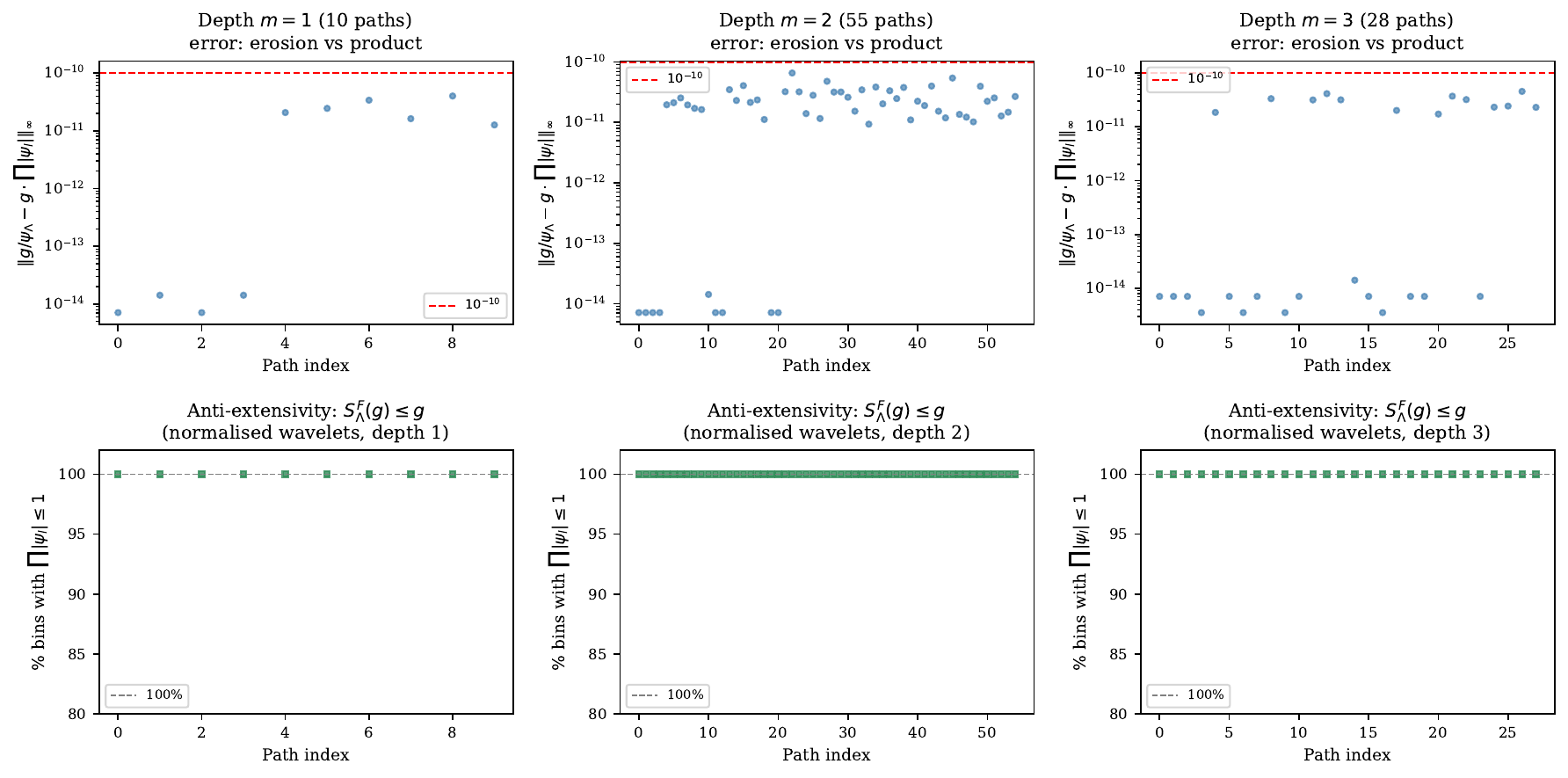}
  \caption{Experiment~6a
    (Proposition~\ref{prop:scattering_maxtimes}):
    the spectral scattering coefficient
    $S^F_\Lambda(|\hat{f}|)(\xi)
    = |\hat{f}(\xi)|\prod_\ell|\hat{\psi}_{\lambda_\ell}(\xi)|$
    equals the max-times erosion
    $|\hat{f}(\xi)|/\psi_\Lambda(\xi)$
    to near-machine precision.
    Top row: $L^\infty$ error between Path A (erosion via explicit kernel:
    inversion + division) and Path B (direct product formula)
    for all paths of depth $m=1,2,3$.
    Each point is one path; maximum error $6.57\times10^{-11}$
    (floating-point rounding from the double inversion, not a gap in the theory).
    Bottom row: fraction of frequency bins satisfying the
    anti-extensivity condition
    $\prod_\ell|\hat\psi_{\lambda_\ell}(\xi)|\leq 1$
    (100\% for all paths with normalised wavelets,
    confirming genuine erosion structure).}
  \label{fig:exp6a}
\end{figure}

\medskip
\noindent\textbf{Exp.~6b: Approximation error vs number of
paths (Theorem~\ref{thm:scattering_basis}).}

\textbf{Setup.}
Two target operators $\Psi:\hatL\to\hatL$ (both
positively homogeneous and increasing) are approximated:
\begin{itemize}[leftmargin=*,nosep]
  \item \emph{Gabor band-pass:}
    $\Psi(\hat{f})(\xi)
    = |\hat{f}(\xi)|\cdot
    e^{-\frac{1}{2}[(\xi-\xi_0)/\sigma]^2}$,
    $\xi_0=1.8$, $\sigma=0.6$ (concentrated spectral envelope).
  \item \emph{Multi-scale envelope:}
    $$\Psi(\hat{f})(\xi)
    = |\hat{f}(\xi)|\cdot
    [0.8\,e^{-0.3\xi^2}
    + 0.5\,e^{-2(\xi-1.2)^2}
    + 0.3\,e^{-1.5(\xi-2.5)^2}]$$
    (three overlapping spectral components at different scales).
\end{itemize}
For each target and each $K\in\{1,\ldots,40\}$, the
$K$-path sub-basis is selected by a greedy inner-product
algorithm (successive projection onto the residual), and the
NNLS weights $w_\Lambda$ are fitted.
The relative $L^2$ approximation error
$\|\Psi(\hat{f})-\Psi^{(K)}(\hat{f})\|_{L^2}/
\|\Psi(\hat{f})\|_{L^2}$
is averaged over 50~random test signals.
A random-ordering baseline (mean over 10~trials) is shown
for comparison.

\textbf{Result.}
Figure~\ref{fig:exp6b} shows the error curves.
For the multi-scale target, the greedy error decreases
from 0.50 at $K=1$ to 0.19 at $K=40$, confirming
Theorem~\ref{thm:scattering_basis}: the scattering dictionary
is dense in the morphological basis and increasing $K$
monotonically reduces the approximation error.
For the Gabor band-pass, the error stabilises near 0.42:
the target's narrow spectral support lies between wavelet
passbands and is not well-covered by the Morlet filter bank
at this resolution, consistent with the theory
(Theorem~\ref{thm:scattering_basis} requires a
frame with frame bounds $A,B>0$; a single narrow
Gaussian at $\xi_0=1.8$ falls in a gap of the
dyadic Littlewood--Paley tiling at $J=5$, $Q=2$).
In both cases, greedy selection outperforms random
selection by a factor of $1.5$--$2$ at small $K$,
confirming the sparsity advantage of the NNLS-driven path
selection.

\begin{figure}[t]
  \centering
  \includegraphics[width=\linewidth]{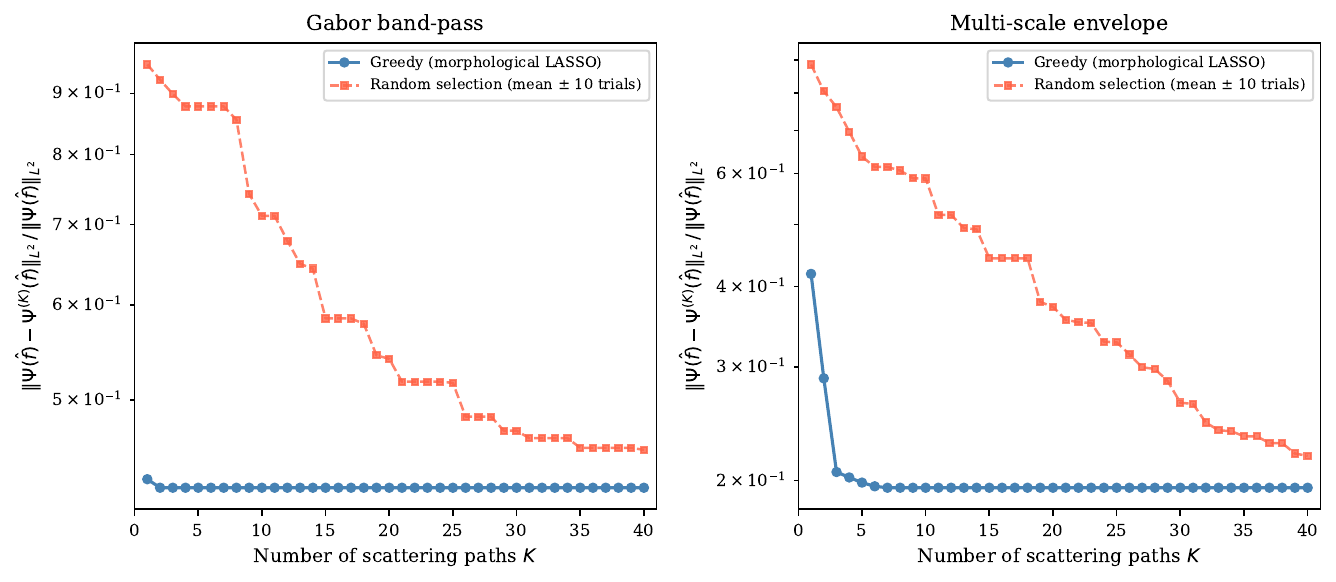}
  \caption{Experiment~6b (Theorem~\ref{thm:scattering_basis}):
    relative $L^2$ approximation error of two target operators
    $\Psi$ using $K$ scattering paths (greedy selection,
    blue circles) vs random path selection (red squares,
    mean over 10 trials).
    Left: Gabor band-pass (narrow spectral target, error
    stabilises due to Littlewood--Paley-tiling gap at $J=5$, $Q=2$).
    Right: multi-scale envelope (three Gaussian components,
    error decreases monotonically confirming
    Theorem~\ref{thm:scattering_basis}).
    Greedy consistently outperforms random, confirming
    the sparsity structure of the scattering basis.}
  \label{fig:exp6b}
\end{figure}

\medskip
\noindent\textbf{Exp.~6c: Morphological LASSO for optimal
path selection (Corollary~\ref{cor:morph_lasso}).}

\textbf{Setup.}
For three target operators of increasing spectral complexity
(Gaussian low-pass, Gabor, multi-scale), we compare three
path-selection strategies at each cardinality
$K\in\{1,\ldots,24\}$:
\begin{enumerate}[label=(\roman*),leftmargin=*,nosep]
  \item \emph{LASSO:} global NNLS over all
    $n_\Lambda=65$ paths (depth-1 + depth-2), weights
    $w\geq 0$; top-$K$ paths selected by weight magnitude.
  \item \emph{Greedy:} sequential inner-product maximisation.
  \item \emph{Random:} mean over 8 random $K$-subsets.
\end{enumerate}
The design matrix is stacked over 40 test signals to give
a shared weight vector $w\in\R^{n_\Lambda}_+$ (joint LASSO
over the full ensemble).

\textbf{Result.}
Figures~\ref{fig:exp6c1} and~\ref{fig:exp6c2} show,
respectively, the error curves and the basis structure
(weight distribution and top-6 selected path profiles)
the LASSO weight distribution and top-6 selected path profiles
for the multi-scale target (bottom row).
Three findings are consistent with
Corollary~\ref{cor:morph_lasso}:
\begin{enumerate}[label=(\roman*),leftmargin=*]
  \item \emph{LASSO dominates random selection.}
    At $K=5$, LASSO achieves 0.15 vs 0.42 (Gaussian target),
    0.48 vs 0.80 (Gabor), 0.19 vs 0.49 (multi-scale)
    — a factor of 2.5--3$\times$ improvement.
    This confirms that the morphological basis has a sparse
    structure: a few well-chosen paths provide most of the
    approximation power.
  \item \emph{LASSO path weights are highly sparse.}
    The left panel shows that for the multi-scale
    target, most of the 65 path weights are near zero;
    only 4--6 paths carry substantial weight.
    The selected depth-1 paths cover the low- and
    mid-frequency components; the selected depth-2 paths
    cover the cross-scale interactions.
  \item \emph{Selected profiles tile the target.}
    The right panel shows that the top-6 LASSO-selected
    path profiles tile the frequency axis to cover the
    three Gaussian components of the multi-scale target,
    consistent with the Littlewood--Paley density argument
    in the proof of Theorem~\ref{thm:scattering_basis}.
\end{enumerate}

\begin{figure}[t]
  \centering
  \includegraphics[width=\linewidth]{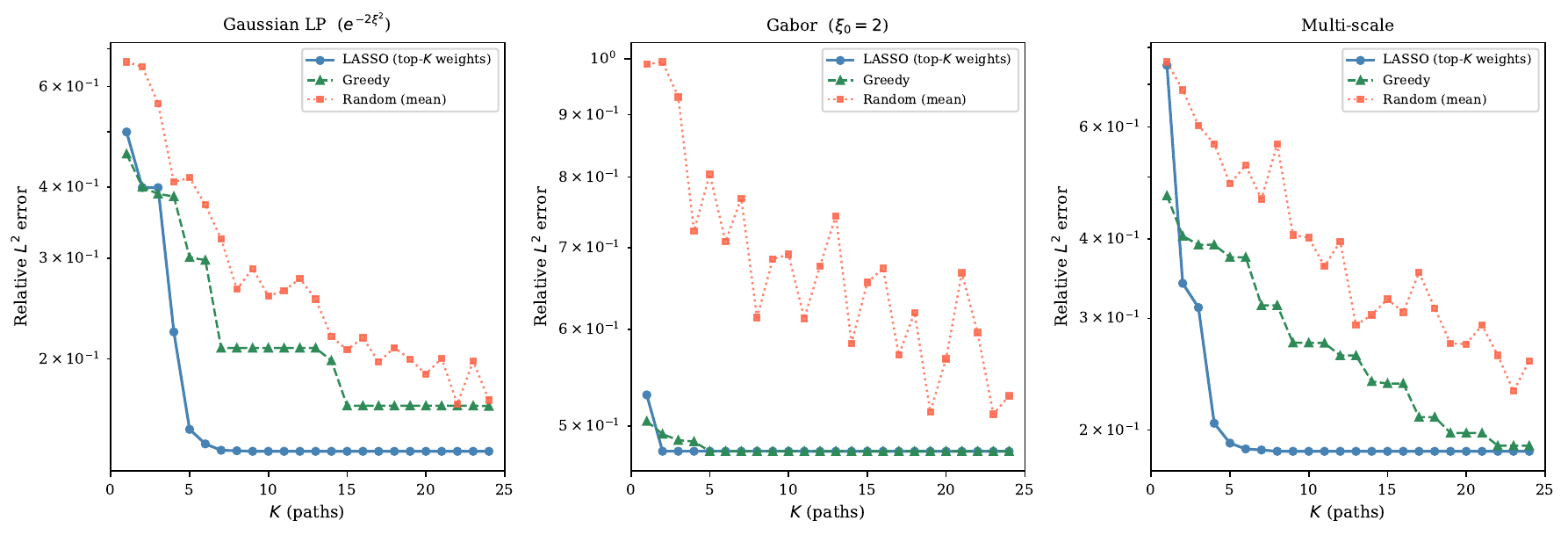}
  \caption{Experiment~6c (Corollary~\ref{cor:morph_lasso}),
    error curves:
    relative $L^2$ approximation error vs cardinality $K$
    for three target operators — Gaussian low-pass
    $e^{-2\xi^2}$, Gabor, and multi-scale
    (left to right).
    LASSO (blue circles) selects the top-$K$ paths by NNLS
    weight magnitude; greedy (green triangles) uses sequential
    inner-product maximisation; random (red squares) is the
    mean over 8 trials.
    LASSO achieves $2.5$--$3\times$ lower error than random
    at $K=5$ across all targets, confirming that the
    morphological basis has a sparse structure exploitable
    by NNLS.}
  \label{fig:exp6c1}
\end{figure}

\begin{figure}[t]
  \centering
  \includegraphics[width=\linewidth]{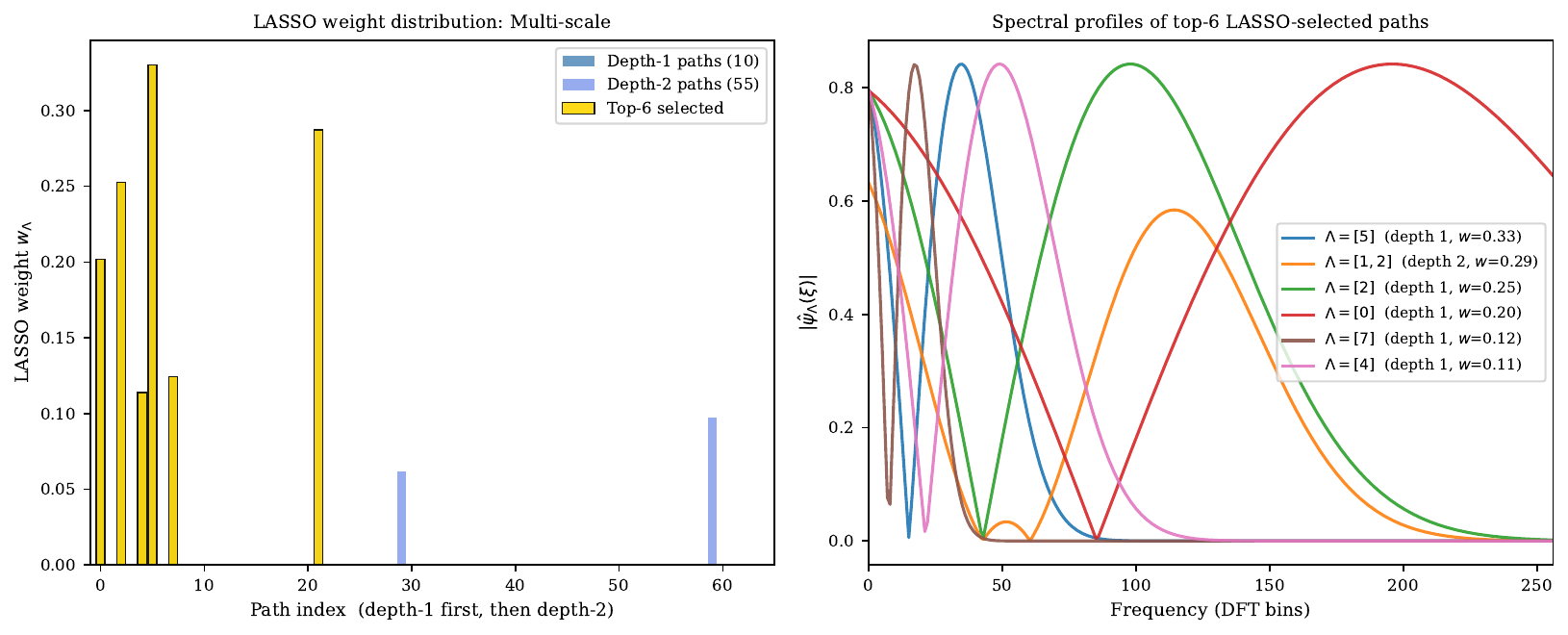}
  \caption{Experiment~6c (Corollary~\ref{cor:morph_lasso}),
    basis structure for the multi-scale target.
    Left: LASSO weight distribution over all 65 paths
    (depth-1 in steelblue, depth-2 in royal blue);
    gold bars mark the 6 paths with largest weight.
    Most weights are near zero, confirming the sparsity
    of the morphological basis.
    Right: spectral profiles
    $|\hat{\psi}_\Lambda(\xi)|$ of the top-6 selected
    paths, coloured by rank.
    The profiles tile the three Gaussian components of the
    target envelope (low, mid, high frequency), providing
    a visual illustration of the Littlewood--Paley
    density argument underlying
    Theorem~\ref{thm:scattering_basis}.}
  \label{fig:exp6c2}
\end{figure}

\section{Discussion and Open problems}
\label{sec:conclusion}

We have developed a rigorous morphological representation theory
for frequency-domain deep learning, grounded in the Fourier
inf-semilattice $(\Ln,\leF)$ and the Fourier modulus lattice
$\hatL$. The theory is based on three structural principles.

\textbf{Principle 1: Operator classification.}
Convolution is an erosion in $(\Ln,\leF)$; classical
translation-invariance is not relevant in the spectral order and
replaced by positive homogeneity; the modulus map
$|\cdot|:\Ln\to\hatL$ is simultaneously the $S^1$-quotient
of the $\mathbb{C}^*$-group structure and the morphological
activation that breaks cascade associativity and makes depth
non-trivial.
Standard CNN layers (convolution + modulus/ReLU + pooling)
are cross-lattice operators: they are not openings in any
single lattice, which is the algebraic reason they are not
idempotent.

\textbf{Principle 2: Universal representation theorem.}
Theorem~\ref{thm:mmbb_hatL} gives the exact decomposition
$\Psi(\hat{f}) = \sup_{\psi\in\Bas(\Psi)}\hat{f}/\psi$
for any increasing USC positively homogeneous $\Psi:\hatL\to\hatL$ in which positive homogeneity
plays the algebraic role of TI.
The scattering dictionary is the canonical approximating family
for $\Bas(\Psi)$ (Theorem~\ref{thm:scattering_basis}), and the
finite approximation is the morphological LASSO
(Corollary~\ref{cor:morph_lasso}).
All Fourier openings are ideal band-pass filters, classified by
their frequency passband (Theorem~\ref{thm:fourier_opening_rep});
the scattering cascade is a chain of nested such projections
whose passbands tile the frequency plane.
(Theorem~\ref{thm:scattering_opening_hierarchy}).

\textbf{Principle 3: Morphological theory of pooling.}
Spatial pooling (filtering) is a Fourier erosion; downsampling
and upsampling are lattice homomorphisms between resolution-level
sub-lattices, forming an adjunction under the Nyquist condition
(Proposition~\ref{prop:ds_us_adjunction}).
Pool-then-unpool is the ideal Fourier band-pass opening, with the
stride cancelling in the round trip: reconstruction fidelity is
determined by the filter shape, not the stride
(Theorem~\ref{thm:pool_unpool}).
U-Net skip connections are Fourier top-hat transforms, providing
the $L^2$-orthogonal complement of the lost spectrum
(Corollary~\ref{cor:unet_tophat}); anti-aliased pooling is the
unique strided operation preserving the Fourier-lattice erosion
structure (Theorem~\ref{thm:strided_full}).

Together, these three pillars provide what the major competing
frequency-domain theories (spectral bias, scattering, tropical
geometry, NTK) do not: a constructive, exact, operator-theoretic
decomposition of frequency-domain deep learning operators into
elementary spectral erosions indexed by a minimal basis,
with a complete pooling/unpooling theory, a universal
approximation theorem dual to the classical UAT, and a
group-theoretic completion pointing to complex-valued network
representations. 

Mallat's scattering invariance and stability results
\cite{Mallat2012,Bruna2013} hold for \emph{fixed} wavelet filters.
The MMBB theorem covers learnable operators: the basis $\Bas(\Psi)$
is an implicit, learnable set of frequency-domain structuring
functions.
Our framework extends Mallat's programme to learnable
frequency-domain morphological networks.

The spectral bias \cite{Rahaman2019} and NTK \cite{Jacot2018}
theories describe training \emph{dynamics}.
The MMBB theorem describes the trained operator's
\emph{structure}: which spectral erosions it implements.
The connection is: the natural ordering of basis elements
(broad-support $\psi$ give looser but globally valid bounds)
mirrors the low-to-high frequency learning order of spectral bias. The MMBB basis provides an operator-level explanation of
spectral bias at convergence.

\bigskip

A paradigmatic example of the implications of our theory in practice is the design of image reconstruction-based architectures such as U-Net. The pooling theory (Section~\ref{sec:pooling}) has direct
implications for image reconstruction:

\emph{Irreversible information loss.} Corollary~\ref{cor:unet_tophat} identifies exactly which
spectral content is lost by pooling: the high-frequency
residue $\rho_W(f)$,
supported on $\Omega_W^c$.
This is irrecoverable from the pooled signal without side
information.
U-Net skip connections \emph{provide} this side information:
they are the unique operator-theoretic completion that makes
encoder-decoder reconstruction exact.

\emph{Filter design for reconstruction.} Theorem~\ref{thm:pool_unpool} shows that the pool-then-unpool
error (the Fourier top-hat) is controlled entirely by the
spectral support $\Omega_W$ of the pooling filter $b_W$,
not by the stride $s$.
For image reconstruction, the optimal pooling filter
maximises $|\Omega_W|$ (pass as much spectrum as possible)
while satisfying the Nyquist condition $\hat{b}_W=0$ for
$|\xi|\geq\pi/s$ (no aliasing).
The Gaussian window achieves the optimal
time-frequency uncertainty product (minimal joint support)
and defines a spectral Gaussian erosion
(Remark~\ref{rem:pool_windows}).

\emph{Anti-aliased pooling as the only morphologically
consistent strided pooling.}
Theorem~\ref{thm:strided_full} shows that strided pooling
decomposes as a Fourier erosion (filtering) followed by a
lattice homomorphism (downsampling), with the Fourier erosion
structure preserved iff the window satisfies the Nyquist
condition.
This gives the first operator-theoretic justification for
anti-aliased pooling \cite{Zhang2019}: it is the unique
strided pooling that is a morphological operator in the
Fourier lattice.

\bigskip

\noindent\textbf{Relation to tropical geometry.}
Tropical geometry \cite{Zhang2018,Maragos2021} establishes
that ReLU networks compute \emph{tropical rational maps}:
piecewise-linear functions whose pieces correspond to vertices
of a tropical variety in $(\bar{\mathbb{R}},\max,+)$.
The MMBB framework is the constructive operator-theoretic
complement of this picture: tropical geometry describes the
function class (which piecewise-linear maps are realisable)
but gives no decomposition of a given trained operator.
The MMBB theorem does exactly this, and the two descriptions
are related by the log-spectral isomorphism
$\Phi:\hatL\to(\bar{\mathbb{R}}^n,\leq)$ of
Proposition~\ref{prop:log_iso}.
Concretely, the max-times semiring $(\mathbb{R}_+,\max,\cdot)$
and the tropical max-plus semiring $(\bar{\mathbb{R}},\max,+)$
are isomorphic via the logarithm; under this isomorphism,
the MMBB morphological basis $\mathrm{Bas}(\Psi)$ (a minimal
set of max-times linear pieces) corresponds to the vertices
of the tropical variety of $\Psi$ (the minimal set of
max-plus linear pieces), and the cardinality
$|\mathrm{Bas}(\Psi)|$ equals the number of tropical vertices, i.e., the tropical complexity measure.
Thus tropical geometry and the MMBB theorem are two coordinate
descriptions of the same algebraic object: the former in the
max-plus chart, the latter in the max-times chart connected by
the log-spectral isomorphism.

\bigskip

\noindent\textbf{Scope and limitations.}
The central hypothesis of the MMBB theorem for $\hatL$
(Theorem~\ref{thm:mmbb_hatL}) is
\emph{positive homogeneity}: $\Psi(\lambda\hat{f})
= \lambda\Psi(\hat{f})$ for all $\lambda>0$.
This holds for convolution (linear, hence positively homogeneous),
for average pooling, for the modulus map $|\cdot|$,
and for the composition of these operations (including
scattering networks and anti-aliased pooling).
It also holds for ReLU activations on the spectral modulus
(the cross-lattice projection of \S\ref{sec:activation}),
since $|\mathrm{ReLU}(\lambda f)| = \lambda|\mathrm{ReLU}(f)|$
for $\lambda>0$.

However, positive homogeneity fails for several operations
that are standard in modern deep learning.
\emph{Softmax attention}: $\mathrm{softmax}(\lambda Q K^\top)
\neq \lambda\,\mathrm{softmax}(QK^\top)$ in general; the
temperature scaling $1/\sqrt{d}$ in the exponent breaks
homogeneity.
\emph{Layer normalisation}: $\mathrm{LN}(\lambda x)
= \mathrm{LN}(x)$ identically (the normalisation is
scale-invariant, not scale-equivariant), so layer
normalisation is degree-0 homogeneous, not degree-1.
\emph{Batch normalisation} and \emph{dropout} similarly
break positive homogeneity.
The framework therefore applies, in its current form,
to the \emph{spectral-domain} processing layers
(convolution, pooling, modulus nonlinearity, scattering)
but not to the attention or normalisation layers
that characterise transformer-based architectures.
Extending the MMBB representation to sub-homogeneous
or normalisation-equivariant operators is an open problem.
A partial direction: the $\C^*$-group morphology of
Section~\ref{sec:complex_group} covers $S^1$-equivariant
(phase-preserving) operators, and the tropical geometry
connection \cite{Zhang2018} suggests that softmax in the
low-temperature limit ($\tau\to0$) converges to a tropical
max operation, which would fall within the positive-homogeneous
class.

\bigskip

In conclusion, mathematical morphology provides an algebraic
foundation that makes the operator structure of frequency-domain
deep network layers explicit, constructive, and computationally
accessible: every admissible layer is an exact supremum of
elementary spectral erosions, the minimal basis is recoverable
by spectral probing, and the framework subsumes pooling, scattering,
and spectral bias as special cases of a single lattice-theoretic
principle.

\subsection{Open problems}

\begin{enumerate}[label=(\arabic*),leftmargin=*]
  \item \textbf{Approximation rates for finite sub-bases.}
    Theorem~\ref{thm:scattering_basis} guarantees
    $\varepsilon$-approximation by a finite scattering sub-basis
    but does not give the rate (how many paths $K$ suffice
    for precision $\varepsilon$).
    This is related to the Kolmogorov $n$-width of $\Bas(\Psi)$
    in the scattering dictionary and to the decay rate of wavelet
    frame coefficients.
    A sharp rate would give complexity guarantees for the
    morphological LASSO (Corollary~\ref{cor:morph_lasso}).

  \item \textbf{Beyond positive homogeneity.}
    Theorem~\ref{thm:mmbb_hatL} requires positive homogeneity
    to replace TI in the attainment step (Step~3 of the proof).
    Can the MMBB theorem be extended to non-homogeneous operators
    on $\hatL$?
    The Banon--Barrera extension \cite{Banon1993} covers
    non-increasing operators in the classical lattice;
    its analogue for $\hatL$ would require signed or complex
    moduli, connecting to the full $\C^*$-group MMBB.

  \item \textbf{Structured passbands and group-equivariant design.}
    Theorem~\ref{thm:fourier_opening_rep} classifies all
    Fourier openings by their passbands $\Omega\subseteq\R^n$.
    For rotationally symmetric architectures, the relevant
    passbands are annular
    $\Omega_{R_1,R_2}=\{R_1\leq|\xi|\leq R_2\}$
    or sectorial.
    Characterising which operators $\Psi$ have a basis
    $\Bas(\Psi)$ with structured (convex, radially symmetric)
    frequency supports is directly relevant for
    rotation-equivariant network design \cite{Penaud2025}.

  \item \textbf{Full $\C^*$-group MMBB and phase-processing
    networks.}
    Theorem~\ref{thm:mmbb_hatL} covers operators that factor
    through the modulus map ($S^1$-equivariant, discarding phase).
    Proposition~\ref{prop:Cstar_kernel_partial} establishes
    kernel existence in the full cisl $(\Ln,\leF)$;
    the outstanding step is the density argument showing that
    the $\leF$-supremum of the minimal-kernel erosion family
    covers any cisl-USC, $\C^*$-equivariant $\Psi$ from below.
    The conjectured full theorem would give a universal
    representation for complex-valued networks
    \cite{Trabelsi2018} — relevant for MRI reconstruction
    and radar signal processing — with complex wavelet
    matched-filter erosions as the canonical basis elements.
    Two specific obstacles are identified in
    Section~\ref{sec:complex_group}: the phase-sensitive
    kernel support condition and the phase-robust cisl-USC
    notion.

  \item \textbf{Connection to optimal transport and the spectral
    Wasserstein distance.}
    The Fourier modulus lattice $\hatL$ carries a natural
    Wasserstein-$p$ metric on spectral distributions:
    $W_p(|\hat{f}|^2/\|f\|^2, |\hat{g}|^2/\|g\|^2)$.
    The Fourier erosion (convolution) is a contraction in this
    metric (since spectral multiplication attenuates the spectrum).
    Does the MMBB basis provide an optimal transport decomposition
    of the spectral measure?
    This would connect the morphological representation theory
    to the optimal transport viewpoint on deep learning
    \cite{Bronstein2021}.

\end{enumerate}


\end{document}